\documentclass[11pt]{amsart}

\usepackage{amsxtra,amssymb,amsthm,amsmath,amscd,url,listings}
\usepackage[utf8]{inputenc}
\usepackage{eucal}
\usepackage{fullpage}
\usepackage{scrtime}
\usepackage{hyperref}

\renewcommand{\leq}{\leqslant}
\renewcommand{\geq}{\geqslant}
 
\numberwithin{equation}{section}

\def\stacksum#1#2{{\stackrel{{\scriptstyle #1}}
{{\scriptstyle #2}}}}

\newcommand{\Cc}{\mathbf{C}}
\newcommand{\Nn}{\mathbf{N}}
\newcommand{\Aa}{\mathbf{A}}

\newcommand{\Zz}{\mathbf{Z}}
\newcommand{\Pp}{\mathbf{P}} 
\newcommand{\Rr}{\mathbf{R}}
\newcommand{\Gg}{\mathbf{G}}

\newcommand{\Hh}{\mathbf{H}}
\newcommand{\Qq}{\mathbf{Q}}
\newcommand{\Fp}{{\mathbf{F}_p}}

\newcommand{\Ff}{\mathbf{F}}
\newcommand{\Tt}{\mathbf{T}}

\newcommand{\mods}[1]{\,(\mathrm{mod}\,{#1})}

\newcommand{\wwd}{\mathcal{C}}
\newcommand{\haut}{\mathbf{G}}

\newcommand{\hautk}[2]{\haut_{{#1},{#2}}}

\newcommand{\hauti}[3]{\mathbf{G}^{{#1}}_{{#2},{#3}}}

\newcommand{\tsum}{\mathcal{S}}
\newcommand{\skl}[1]{\sheaf{K}^{({#1})}}
\newcommand{\hk}[1]{\sheaf{K}\ell_{{#1}}}
\newcommand{\mutw}[3]{\mu_{{#3},{#2}}}
\newcommand{\frtr}[3]{(\Tr{{#1}})({#2},{#3})}
\DeclareMathOperator{\hypk}{Kl}
\newcommand{\HYPK}{\mathcal{K}\ell}
\newcommand{\tnorm}[2]{\|{#1}\|_{\mathrm{tr},{#2}}}

\newcommand{\ra}{\rightarrow}
\newcommand{\lra}{\longrightarrow}

\newcommand{\injecte}{\hookrightarrow}

\newcommand{\fleche}[1]{\stackrel{#1}{\lra}}

\DeclareMathOperator{\rank}{rank}

\DeclareMathOperator{\frob}{\mathrm{Fr}}
\DeclareMathOperator{\Gal}{Gal}
\DeclareMathOperator{\Ind}{Ind}

\DeclareMathOperator{\Tr}{tr}

\DeclareMathOperator{\swan}{Swan}

\DeclareMathOperator{\ft}{FT}
\DeclareMathOperator{\cond}{cond}

\DeclareMathOperator{\dual}{D}

\newcommand{\eps}{\varepsilon}
\renewcommand{\rho}{\varrho}

\DeclareMathOperator{\SL}{SL}
\DeclareMathOperator{\GL}{GL}
\DeclareMathOperator{\PGL}{PGL}
\DeclareMathOperator{\rmT}{T}
\DeclareMathOperator{\rmN}{N}
\DeclareMathOperator{\rmU}{U}

\newcommand{\demi}{{\textstyle{\frac{1}{2}}}}

\newcommand{\sheaf}[1]{\mathcal{{#1}}}

\DeclareMathSymbol{\gena}{\mathord}{letters}{"3C}
\DeclareMathSymbol{\genb}{\mathord}{letters}{"3E}

\def\dblsum{\mathop{\sum \sum}\limits}

\def\multsum{\mathop{\sum\cdots \sum}\limits}

\def\sums{\mathop{\sum \Bigl.^{*}}\limits}
\def\dblsums{\mathop{\sum\sum \Bigl.^{*}}\limits}

\theoremstyle{plain}
\newtheorem{theorem}{Theorem}[section]
\newtheorem{lemma}[theorem]{Lemma}
\newtheorem{corollary}[theorem]{Corollary}

\newtheorem{proposition}[theorem]{Proposition}
\newtheorem*{proposition*}{Proposition}

\theoremstyle{remark}

\newtheorem*{rem}{Remark}

\newtheorem{remark}[theorem]{Remark}

\theoremstyle{definition}

\newtheorem{definition}[theorem]{Definition}

\newcommand{\abs}[1]{\lvert#1\rvert}

\newcommand{\mcL}{\mathcal{L}}
\newcommand{\mcS}{\mathcal{S}}

\newcommand{\mcF}{\mathcal{F}}
\newcommand{\mcG}{\mathcal{G}}
\newcommand{\mcB}{\mathcal{B}}
\newcommand{\mcE}{\mathcal{E}}

\newcommand{\mfa}{\mathfrak{a}}

\renewcommand{\geq}{\geqslant}
\renewcommand{\leq}{\leqslant}
\renewcommand{\Re}{\mathfrak{Re}\,}
\renewcommand{\Im}{\mathfrak{Im}\,}

\newcommand{\refs}{\eqref}
\newcommand{\bash}{\backslash}

\newcommand{\ov}[1]{\overline{#1}}

\newcommand{\peter}[1]{\langle{#1}\rangle}
\newcommand\sumsum{\mathop{\sum\sum}\limits}

\begin{document}

\title{Algebraic twists of modular forms and Hecke orbits}
 
\author{\'Etienne Fouvry}
\address{Universit\'e Paris Sud, Laboratoire de Math\'ematique\\
  Campus d'Orsay\\ 91405 Orsay Cedex\\France}
\email{etienne.fouvry@math.u-psud.fr} \author{Emmanuel Kowalski}
\address{ETH Z\"urich -- D-MATH\\
  R\"amistrasse 101\\
  CH-8092 Z\"urich\\
  Switzerland} \email{kowalski@math.ethz.ch} \author{Philippe Michel}
\address{EPFL/SB/IMB/TAN, Station 8, CH-1015 Lausanne, Switzerland }
\email{philippe.michel@epfl.ch}

\date{\today,\ \thistime} 

\thanks{Ph. M. was partially supported by
  the SNF (grant 200021-137488) and the ERC (Advanced Research Grant
  228304). \'E. F. thanks ETH Z\"urich, EPF Lausanne and the Institut
  Universitaire de France for financial support.  }


\begin{abstract}
  We consider the question of the correlation of Fourier coefficients
  of modular forms with functions of algebraic origin. We establish
  the absence of correlation in considerable generality (with a power
  saving of Burgess type) and a corresponding equidistribution
  property for twisted Hecke orbits. This is done by exploiting the
  amplification method and the Riemann Hypothesis over finite fields,
  relying in particular on the $\ell$-adic Fourier transform
  introduced by Deligne and studied by Katz and Laumon.
\end{abstract}

\maketitle
\setcounter{tocdepth}{1}
\tableofcontents

\section{Introduction and statement of results}

This paper concerns a certain type of sums involving Fourier 
coefficients of modular forms, which we call ``algebraic
twists''. Their study can be naturally motivated either from a point
of view coming from analytic number theory, or from geometric
considerations involving Hecke orbits on modular curves. We will
present them using the first approach, and discuss the geometric
application in Section~\ref{ssec-orbits}.
\par
We will be considering either holomorphic cusp forms or Maass
forms. Precisely, the statement \emph{$f$ is a cusp form} will mean,
unless otherwise indicated, that $f$ is either (1) a non-zero
holomorphic cusp form of some even weight $k\geq 2$ (sometimes denoted
$k_f$) and some level $N\geq 1$; or (2) a non-zero Maass cusp form of
weight $0$, level $N$ and Laplace eigenvalue written $1/4+t_f^2$. In
both cases, we assume $f$ has trivial Nebentypus for simplicity.
\par
The statement that a cusp form $f$ of level $N$ is \emph{a Hecke
  eigenform} will also, unless otherwise indicated, mean that $f$ is
an eigenfunction of the Hecke operators $T_n$ with $(n,N)=1$.

\subsection{Algebraic twists of modular forms}

Let $f:\Hh\ra\Cc$ be a cusp form (as discussed above). We have
$f(z+1)=f(z)$, so $f$ that admits a Fourier expansion at infinity, and
we denote the $n$-th Fourier coefficient of $f$ by
$\rho_f(n)$. Explicitly, if $f$ is holomorphic of weight $k$, the
Fourier expansion takes the form
$$
f(z)=\sum_{n\geq 1}{n^{(k-1)/2}\rho_f(n)e(nz)},
$$
and if $f$ is a Maass form, the Fourier expansion is normalized as
in~(\ref{eq-fourier-expansion}) below. It follows from Rankin-Selberg
theory that the Fourier coefficients $\rho_f(n)$ are bounded on
average, namely
\begin{equation}
\label{RankinSelberg}
\sum_{n\leq x}|\rho_f(n)|^2=c_f x+O(x^{3/5})
\end{equation}
for some $c_f>0$. For individual terms, we have
\begin{equation}
\label{ksbound}
\rho_f(n)\ll_{\eps,f}n^{7/64+\eps}
\end{equation}
for any $\eps>0$ by the work of Kim and Sarnak~\cite{kim-sarnak}, and
moreover, if $f$ is holomorphic, it follows from Deligne's proof of
the Ramanujan-Petersson conjecture that the $\rho_f(n)$ are almost
bounded, so that
$$
\rho_f(n)\ll_{\eps,f} n^\eps
$$
for any $\eps>0$.

On the other hand, it is also well-known that the Fourier coefficients
oscillate quite substantially, as the estimate
\begin{equation}\label{oscillate}
\sum_{n\leq x}\rho_f(n)e(\alpha n)\ll x^{1/2}(\log 2x)
\end{equation}
valid for $x\geq 1$ and $\alpha\in\Rr$, with an implied constant
depending on $f$ only, shows (see, e.g.,~\cite[Th. 5.3]{topics}
and~\cite[Th. 8.1]{IwI}).
\par
One may ask, more generally, whether the sequence $(\rho_f(n))_{n\geq
  1}$ \emph{correlates} with another bounded (or essentially bounded)
sequence $K(n)$. This may be defined formally as follows: $(K(n))$
does \emph{not} correlate with the Fourier coefficients of $f$ if we
have
$$
\sum_{n\leq x}\rho_f(n)K(n)\ll x(\log x)^{-A}
$$
for all $A\geq 1$, the implied constant depending on $A$.\footnote{\
  It is not enough to ask that the sum be $o(x)$ because this is then
  true for $K(n)$ equal to the sign of $\rho_f(n)$, see for
  instance~\cite{elliott}.} There are many known examples, of which
we list only a few particularly interesting ones:
\par
\begin{itemize}
\item For $K(n)=\mu(n)$, the M\"obius function, the non-correlation is
  an incarnation of the Prime Number Theorem, and is a consequence of
  the non-vanishing of the Hecke $L$-function $L(f,s)$ for $\Re s=1$
  when $f$ is primitive; more generally, for $K(n)=\mu(n)e(n\alpha)$
  where $\alpha\in\Rr/\Zz$, non-correlation has been obtained recently
  by Fouvry and Ganguly~\cite{fouvry-ganguly};
\item When $K(n)=\overline{\rho_g(n)}$ for $g$ any modular form which
  is orthogonal to $f$, non-correlation is provided by Rankin-Selberg
  theory;
\item For $K(n)=\rho_g(n+h)$ with $h\not=0$ and $g$ any modular form,
  whether it is orthogonal to $f$ or not, non-correlation follows from
  the study of \emph{shifted-convolution sums}, and has crucial
  importance in many studies of automorphic $L$-functions.
\end{itemize}
\par
In this paper we are interested in the absence of correlation of the
coefficients $(\rho_f(n))_n$ against sequences $(K(n))_{n\geq 1}$
where
$$
K\,:\, \Zz/p\Zz\ra \Cc
$$
is a function defined modulo $p$, for some prime $p$, which is
extended to all of $\Zz$ by periodicity. We will then consider sums of
the shape
$$
\sum_{n\leq p}\rho_f(n)K(n),
$$
or rather smoothed versions of these, which we denote
$$
\tsum(f,K;p)=\tsum_V(f,K;p)=\sum_{n\geq 1}{ \rho_f(n)K(n)V(n/p)},
$$
for $V$ a smooth compactly supported function on $]0,+\infty[$ (often
$V$ will be omitted from the notation).
\par
By \refs{RankinSelberg}, the trivial bound for these sums is
$$
\tsum(f,K;p)\ll p\Bigl(\frac{1}p\sum_{n\leq p}{|K(n)|^2}\Bigr)^{1/2}\ll
p\max_{1\leq n\leq p}|K(n)|,
$$
where the implied constant depends on $f$ and $V$, and our aim will be to
improve this bound; we will prove estimates of the shape
\begin{equation}
\label{oscillatep}
\tsum(f,K;p)\ll p^{1-\delta}
\end{equation}

for some absolute $\delta>0$, where the implied constant depends only
on $f$, $V$ and easily controlled invariants of $K$, such as
$$
\|K\|_2=\Bigl(\frac{1}p\sum_{n\leq p}{|K(n)|^2}\Bigr)^{1/2} \text{ or
}\|K\|_\infty=\max |K(n)|.
$$
\par
A first (slightly degenerate) example is a (normalized) \emph{Dirac}
function located at some $u\in\Fp$, i.e., $K(n)=p^{1/2}\delta_{n\equiv
  u\mods{p}}$. Here $\|K\|_\infty=p^{1/2}$ is large, but $\|K\|_2=1$
and
\begin{equation}
\label{diracbound}
\tsum(f,K;p)=p^{1/2}\sum_{n\equiv u\mods{p}}\rho_f(n)V(n/p)\ll p^{1-\delta}
\end{equation}
for any $\delta<1-7/64$ by \refs{ksbound}.
\par
Another non-trivial choice (somewhat simpler than the previous one) is
an additive character modulo $p$ given by $K(n)=e(an/p)$ for some
fixed $a\in \Zz$. In that case, $|K(n)|\leq 1$ and the
bound~\refs{oscillate} gives \refs{oscillatep} for any $\delta<1/2$,
with an implied constant depending only on $f$ and $V$.
\par
A third interesting example is given by $K(n)=\chi(n)$,
where $\chi$ is a non-trivial Dirichlet character modulo $p$ (extended
by $0$ at $p$). In that case, the bound \refs{oscillatep}, with an
implied constant depending only on $f$ and $V$, is essentially
equivalent to a \emph{subconvex} bound for the twisted $L$-function
$L(f\otimes\chi,s)$ in the level aspect, i.e., to a bound
$$
L(f\otimes\chi,s)\ll_{s,f} p^{1/2-\delta'},
$$
for some $\delta'>0$ and any fixed $s$ on the critical line. Such an
estimate was obtained for the first time by Duke-Friedlander-Iwaniec
in~\cite{dfi} for any $\delta'<1/22$. This bound was subsequently
improved to any $\delta'<1/8$ (a Burgess type exponent) by Bykovski
and Blomer-Harcos\footnote{\ We are very grateful to G. Harcos for
  pointing out the relevance of these two papers for the present one.}
\cite{byk,blomer-harcos}, and to $\delta'<1/6$ (a Weyl type exponent)
when $\chi$ is quadratic by Conrey-Iwaniec \cite{CI}.
\par
There are many other functions which occur naturally. We highlight two
types here. First, given rational functions $\phi_1,\phi_2$, say
$$
\phi_i(X)=\frac{R_i(X)}{S_i(X)}\in \Qq(X),\ i=1,2
$$
with $R_i,\ S_i\in \Zz[X]$ coprime (in $\Qq[X]$), and given a
non-trivial Dirichlet character $\chi\mods{p}$, one can form
\begin{equation}\label{eq-weight-mixed}
K(n)=\begin{cases}
  e\Bigl(\frac{\phi_1(n)}{p}\Bigr)\chi(\phi_2(n)),&\text{ if }
  p\nmid S_1(n)S_2(n),\\
  0,&\text{otherwise},
\end{cases}
\end{equation}
where inverses are computed modulo $p$ and with the usual convention
$\chi(0)=0$.  We will show that~(\ref{oscillatep}) holds
for such functions with an absolute exponent of Burgess type (see
Corollary~\ref{cor-main} below). The proof depends ultimately on the
Riemann Hypothesis over finite fields, which is applied in order to
estimate exponential sums in $3$ variables with square-root
cancellation, using Deligne's results~\cite{weilii}.
\par
Second, for $m\geq 1$ and $a\in\Ff_p^\times$ let
$$
\hypk_m(a;p)=\frac{1}{p^{\frac{m-1}2}}\multsum_\stacksum{x_1\cdots
  x_m=a}{x_i\in\Ff_p} e\Bigl(\frac{x_1+\cdots+x_m}p\Bigr)
$$ 
be the normalized hyper-Kloosterman sum in $m-1$ variables. Recall
that by the work of Deligne~\cite[Sommes Trig., (7.1.3)]{deligne} we
have
$$
|\hypk_m(a;p)|\leq m,
$$
and sums involving Kloosterman sums or hyper-Kloosterman sums are
frequent visitors of analytic number theorists. Consider now, for
$\phi=\frac{R(X)}{S(X)}\in\Qq[X]$ a non-constant rational function
with $R,S\in\Zz[X]$, $S\not=0$ and $\Phi(U,V)\in\Cc[U,V]$ a polynomial
in two variables, the function
\begin{equation}\label{eq-hypk-weight}
K(n)=\begin{cases} \Phi\big(\hypk_m(\phi(n);p),
  \ov{\hypk_m(\phi(n);p)}\big),&\text{ if }
  p\nmid S(n)\\
  0&\text{otherwise}.
\end{cases}
\end{equation}
\par
We will also show a bound of the type~(\ref{oscillatep}) for these
rather wild functions.
\par
The precise common feature of these examples is that they arise as
linear combination of \emph{Frobenius trace functions} of certain
$\ell$-adic sheaves over the affine line $\Aa^1_{\Ff_p}$ (for some
prime $\ell\not=p$). We therefore call these functions \emph{trace
  functions}, and we will give the precise definition below. To state
our main result, it is enough for the moment to know that we can
measure the complexity of a trace function modulo $p$ with a numerical
invariant called its \emph{conductor} $\cond(K)$. Our result is,
roughly, that when $\cond(K)$ remains bounded, $K(n)$ does not
correlate with Fourier coefficients of modular forms.
\par
As a last step before stating our main result, we quantify the
properties of the test function $V$ that we handle. Given $P>0$ and
$Q\geq 1$ real numbers, we define:

\begin{definition}[Condition $(V(C,P,Q))$]\label{Condition(V)}
  Let $P>0$ and $Q\geq 1$ be real numbers and let $C=(C_\nu)_{\nu\geq 0}$ be a sequence of non-negative real numbers.  A smooth compactly
  supported function $V$ on $[0,+\infty[$ satisfies {\rm Condition}
  $(V(C,P,Q))$ if
\begin{enumerate}
\item \label{support} The support of $V$ is contained in the dyadic
  interval $[P,2P]$; 
\item \label{derivative} For all $x>0$ and all integers $\nu\geq 0$ we
  have the inequality
$$
\bigl|\,x^\nu V^{(\nu)}(x)\,\bigr|\leq C_\nu Q ^{\nu}.
$$
\end{enumerate}
\par
In particular, $|V(x)|\leq C_0$ for all $x$.
\end{definition}
\begin{rem}
A smooth dyadic sum corresponds to cases where $P=1/2$ and $Q$ is
absolutely bounded.  This is the most important situation to consider,
in a first reading at least.  In other situations, we have in mind
that $PQ$ is also absolutely bounded.

As a referee pointed out, the sequence $C=(C_\nu)_{\nu\geq 0}$ should
grow sufficiently fast in order for the set of functions satisfying
$(V(C,P,Q))$ be non-trivial: for instance if $(C_\nu)_{\nu\geq 0}$,
any such function $V$ would have to be analytic hence identically zero
since compactly supported.
\end{rem}

Our main result is:

\begin{theorem}\label{th-traceweight}
  Let $f$ be a Hecke eigenform, $p$ be a prime number and $V$ a
  function satisfying $(V(C,P,Q))$. Let $K$ be an \emph{isotypic} trace
  function of conductor $\cond(K)$, as defined in
  Section~\ref{intro-l-adic}.
\par
There exists $s\geq 1$ absolute such that we have 
$$
\tsum_V(f,K;p)\ll \cond(K)^sp^{1-\delta}(PQ)^{1/2}(P+Q)^{1/2}
$$
for any $\delta<1/8$, where the implied constant depends only on $C$, $f$
and $\delta$.
\end{theorem}

\begin{remark} 
  The Burgess type subconvex bounds for $L(f\otimes\chi,1/2)$ of
  Bykovski and Blomer-Harcos mentioned above can easily be retrieved
  from the special case $K(n)=\chi(n)$.
\end{remark}

As a consequence of this and~(\ref{RankinSelberg}), one has the
following non-trivial estimate for sums over intervals, whose proof is
given in Section~\ref{sec-intervals}:

\begin{corollary}\label{cor-intervals}
  Under the same assumptions as above, for any interval $I\subset
  [1,p]$, we have
\begin{equation}\label{eq-sharp}
\sum_{n\in I}\rho_f(n)K(n)\ll \cond(K)^sp^{1-\delta/2}
\end{equation}
for any $\delta<1/8$, where the implied constant depends only on $f$ and
$\delta$.
\end{corollary}

This result applies almost directly to the
functions~(\ref{eq-weight-mixed}) and~(\ref{eq-hypk-weight}) and to a
wide range of algebraic exponential sums. We refer to Section
\ref{sec-app} for these and for more elaborate applications.
\par
An important point is that estimates like~(\ref{oscillatep}) are
obviously linear with respect to $K$, but the notion of an isotypic
function is not. This justifies the following definition:

\begin{definition}[Trace norms]
  Let $p$ be a prime number, and let $K\,:\, \Fp\lra \Cc$ be any
  function defined modulo $p$. Let $s\geq 1$ be an integer. The
  \emph{$s$-trace norm} of $K$ is
$$
\tnorm{K}{s}=\inf\Bigl\{
\sum_{i}{|\lambda_i|\cond(K_i)^s}+\sum_{j}{|\mu_j|}+
\sum_{k}{|\eta_k|} \Bigr\}
$$
where the infimum runs over all decompositions of $K$ as a finite
linear combination
\begin{equation}\label{eq-short-decomp}
K(x)=\sum_{i}{\lambda_i K_i(x)}+\sum_{j}{\mu_j p^{1/2}\delta_{a_j}(x)}
+\sum_{k}{\eta_k e\Bigl(\frac{b_kx}{p}\Bigr)},
\end{equation}
where $\lambda_i$, $\mu_j$, $\eta_k\in\mathbb{C}$, $a_j$, $b_k\in\Fp$,
and $K_i$ is an isotypic trace function.
\end{definition}

The decomposition of a function in Dirac functions shows that these
norms are well-defined.

We then have:

\begin{corollary}[Trace norm estimate]\label{cor-trace-norm}
  There exists an absolute constant $s\geq 1$ with the following
  property: for any cusp form $f$, any prime $p$, any function $K$
  modulo $p$, for any function $V$ satisfying $(V(C,P,Q))$, we have
$$
\tsum_V(f,K;p)\ll \tnorm{K}{s}p^{1-\delta}(PQ)^{1/2}(P+Q)^{1/2},
$$
for any $\delta<1/8$, where the implied constant depends only on
$(C,f,\delta)$.
\end{corollary}

\begin{proof}
  Indeed, for a decomposition~(\ref{eq-short-decomp}), we can apply
  Theorem~\ref{th-traceweight} for the isotypic trace functions $K_i$,
  with the value of $s$ in that theorem, while we
  use~(\ref{oscillate}) for the components $\eta_ke(b_kx/p)$,
  and~(\ref{ksbound}) for the delta functions.
\end{proof}

\begin{remark}
  It is important to remark that this depends on~(\ref{oscillate}),
  and thus this corollary does not hold for Eisenstein series. For the
  latter, one can define analogues of the trace norms which consider
  decompositions~(\ref{eq-short-decomp}) with no additive characters.
\end{remark}

\subsection{Good functions and correlating matrices} 

To deal with the level of generality we consider, it is beneficial at
first to completely forget all the specific properties that $K$ might
have, and to proceed abstractly. Therefore we consider the problem of
bounding the sum $\tsum_V(f,K;p)$ for $K:\Zz/p\Zz\ra \Cc$ a general
function, assuming only that we know that $|K(n)|\leq M$ for some $M$
that we think as fixed.
\par
For the case of Dirichlet characters, Duke, Friedlander and
Iwaniec~\cite{dfi} amplified $K(n)=\chi(n)$ among characters with a
fixed modulus. Given the absence of structure on $K$ in our situation,
this strategy seems difficult to implement.  Instead, we use an idea
found in~\cite{CI}:\footnote{As pointed out in \cite{CI}, this idea
  occured already in the work of Bykovsky \cite{byk} and was also used
  by Blomer and Harcos \cite{blomer-harcos}.}  we consider $K$
``fixed'', and consider the family of sums $\tsum_V(g,K;p)$ for $g$
varying over a basis of modular cusp forms of level $Np$, viewing $f$
(suitably normalized) as an old form at $p$. Estimating the
amplified second moment of $\tsum_V(g,K;p)$ over that family by the
Petersson-Kuznetzov formula and the Poisson formula, we ultimately
have to confront some sums which we call \emph{correlation sums},
which we now define.
\par
We denote by $\hat{K}$ the (unitarily normalized) Fourier transform
modulo $p$ of $K$, given by
$$
\hat{K}(z)=\frac{1}{p^{1/2}}\sum_{x\mods{p}}{K(x)e\Bigl(\frac{zx}{p}\Bigr)}.
$$
\par
For any field $L$, we let $\GL_2(L)$ and $\PGL_2(L)$ act on
$\Pp^1(L)=L\cup \{\infty\}$ by fractional linear transformations as usual.
Now for $\gamma=\begin{pmatrix}a&b\\c&d
\end{pmatrix}
\in\GL_2(\Fp)$ or in $\PGL_2(\Fp)$, we define the correlation sum
$\wwd(K;\gamma)$ by
\begin{equation}\label{eq-correlation}
\wwd(K;\gamma)=\sum_{\stacksum{z\in\Ff_p}{z\not=-d/c}}
    \hat{K}(\gamma\cdot z)\overline{\hat{K}(z)}.
\end{equation}
\par
The matrices $\gamma$ which arise in our amplification are the
reduction modulo $p$ of integral matrices parameterized by various
coefficients from the amplifier, and we need the sums $\wwd(K;\gamma)$
to be as small as possible. 
\par 
If $\|K\|_\infty\leq M$ (or even $\|K\|_2\leq M$), then the
Cauchy-Schwarz inequality and the Parseval formula show that
\begin{equation}\label{eq-bound-wwd}
|\wwd(K;\gamma)|\leq M^2p.
\end{equation}
\par
This bound is, unsurprinsingly, insufficient. Our method is based on
the idea that $\wwd(K;\gamma)$ should be significantly smaller for
most of the $\gamma$ which occur (even by a factor $p^{-1/2}$,
according to the square-root cancellation philosophy) and that we can
control the $\gamma$ where this cancellation does \emph{not} occur. By
this, we mean that these matrices (which we call the set of
\emph{correlation matrices}) is nicely structured and rather small,
unless $\hat K$ is constant, a situation which means that $K(n)$ is
proportional to $e(\frac{an}p)$ for some $a\in\Zz$, in which case we
can use~(\ref{oscillate}) anyway.
 
In this paper, the structure we obtain is algebraic. To discuss it, we
introduce the following notation concerning the algebraic subgroups of
$\PGL_2$: 
\par
-- we denote by $B\subset \PGL_2$ the subgroup of upper-triangular
matrices, the stabilizer of $\infty\in \Pp^1$;
\par
-- we denote by $w=\left(\begin{array}{cc}0 & 1 \\1 &
    0\end{array}\right)$ the Weyl element, so that $Bw$ (resp. $wB$)
is the set of matrices mapping $0$ to $\infty$ (resp $\infty$ to $0$);
\par
-- we denote by $\PGL_{2,par}$ the subset of matrices in $\PGL_2$
which are parabolic, i.e., which have a single fixed point in $\Pp^1$;
\par
-- Given $x\not=y$ in $\Pp^1$, the pointwise stabilizer of $x$ and $y$
is denoted $\rmT^{x,y}$ (this is a maximal torus), and its normalizer
in $\PGL_2$ (or the stabilizer of the set $\{x,y\}$) is denoted
$\rmN^{x,y}$ .

\begin{definition}[Correlation matrices and good
  functions]\label{def-admissible}
  Let $p$ be a prime and $K\,:\, \Ff_p\ra \Cc$ an arbitrary
  function. Let $M\geq 1$ be such that $\|K\|_2\leq M$. 
\par
(1) We let
\begin{equation}\label{eq-hautk}
\hautk{K}{M}=\{\gamma\in\PGL_2(\Ff_p)\,\mid\,
|\wwd(K;\gamma)|>Mp^{1/2}\},
\end{equation}
the set of \emph{$M$-correlation matrices}.
\par
(2) We say that $K$ is \emph{$(p,M)$-good} if there exist at most $M$
pairs $(x_i,y_i)$ of distinct elements in $\Pp^1(\bar{\Ff}_p)$ such
that
\begin{equation}\label{eq-def-admi}
\hautk{K}{M}=\hauti{b}{K}{M}\cup \hauti{p}{K}{M}
\cup \hauti{t}{K}{M}\cup \hauti{w}{K}{M},
\end{equation}
where
\begin{gather*}
  \hauti{b}{K}{M}\subset B(\Ff_p)\cup B(\Ff_p)w\cup wB(\Ff_p),
\quad  \hauti{p}{K}{M}\subset \PGL_{2,par}\\
  \hauti{t}{K}{M}\subset \bigcup_{i}{\rmT^{x_i,y_i}(\Fp)},
\quad
  \hauti{w}{K}{M}\subset \bigcup_{i}{(
\rmN^{x_i,y_i}-\rmT^{x_i,y_i})(\Fp)}.
\end{gather*}
\end{definition}

In other words: given $M\geq 1$ and $p$ a prime, a $p$-periodic function
$K$ is $(p,M)$-good if the only matrices for which the estimate
$|\wwd(K;\gamma)|\leq Mp^{1/2}$ fails are either (1) upper-triangular
or sending $0$ to $\infty$ or $\infty$ to $0$; or (2) parabolic; or
(3) elements which permute two points defined by at most $M$ integral
quadratic (or linear) equations. We note that if we fix such data, a
``generic'' matrix is \emph{not} of this type.
\par
This notion has little content if $M$ is larger that $p^{1/2}$, but we
will already present below some elementary examples of
$(p,M)$-good functions, together with their sets of correlation
matrices for $M$ fixed and $p$ arbitrary large (not surprisingly, all
these examples come from trace functions).
\par
Given a $(p,M)$-good function $K$, we next show using counting
arguments that the set of matrices $\gamma$ constructed from the
amplifier does not intersect the set of correlating matrices in a too
large set and we eventually obtain our main technical result:

\begin{theorem}[Bounds for good
  twists]\label{th-from-admissible-to-orthogonality}
  Let $f$ be a Hecke eigenform, $p$ be a prime number and $V$ a
  function satisfying $(V(C,P,Q))$. Let $M\geq 1$ be given, and let $K$
  be a $(p,M)$-good function modulo $p$ with $\|K\|_\infty\leq M$.
\par
There exists $s\geq 1$ absolute such that 
$$
  \tsum_V(f,K;p)\ll M^sp^{1-\delta}(PQ)^{1/2}(P+Q)^{1/2},
$$
for any $\delta<1/8$, where the implied constant depends only on
$(C,f,\delta)$.
\end{theorem}

\begin{remark} 
  Although it is an elementary step (compare \refs{eq-e-general} and
  \refs{eq-sum-as-corr} in the proof) the beautiful modular
  interpretation of correlation sums is a key observation for this
  paper. It gives a group theoretic interpretation and introduce
  symmetry into sums, the estimation of which might otherwise seem to
  be hopeless.
\end{remark}

\subsection{Trace functions of $\ell$-adic sheaves}\label{intro-l-adic}

The class of functions to which we apply these general considerations
are the \emph{trace functions} modulo $p$, which we now define formally.
\par
Let $p$ be a prime number and $\ell\not=p$ an auxiliary prime. The
functions $K(x)$ modulo $p$ that we consider are the trace functions of
suitable constructible sheaves on $\Aa^1_{\Ff_p}$ evaluated at
$x\in\Ff_p$. To be precise, we will consider $\ell$-adic constructible
sheaves on $\Aa^1_{\Ff_p}$.
The trace function of such a sheaf $\sheaf{F}$ takes values in an
$\ell$-adic field so we also fix an isomorphism $\iota\,:\,
\bar{\Qq}_{\ell}\lra \Cc$, and we consider the functions of the shape
\begin{equation}\label{eq-iota-trace}
  K(x)=\iota(\frtr{\sheaf{F}}{\Ff_p}{x})
\end{equation}
for $x\in \Ff_p$, as in~\cite[7.3.7]{katz-esde}.

\begin{definition}[Trace sheaves]\label{def-admissible-sheaf}
  (1) A constructible $\bar{\Qq}_{\ell}$-sheaf $\sheaf{F}$ on
  $\Aa^1_{\Ff_p}$ is a \emph{trace sheaf} if it is a middle-extension
  sheaf whose restriction
  to any non-empty open subset $U\subset \Aa_\Fp^1$ where $\sheaf{F}$ is
  lisse and \emph{pointwise $\iota$-pure} of weight $0$.
\par
(2) A trace sheaf $\sheaf{F}$ is called a \emph{Fourier trace sheaf}
if, in addition, it is a Fourier sheaf in the sense of
Katz~\cite[Def. 8.2.2]{katz-gkm}.
\par
(3) A trace sheaf is an \emph{isotypic trace sheaf} if it is a Fourier
sheaf and if, for any open set $U$ as in (1), the restriction of
$\sheaf{F}$ to $U$ is {\em geometrically isotypic} when seen as a
representation of the geometric fundamental group of $U$: it is the
direct sum of several copies of some (necessarily non-trivial)
irreducible representation of the geometric fundamental group of $U$
(see~\cite[\S 8.4]{katz-gkm}).
\par
If $\sheaf{F}$ is geometrically irreducible (instead of being
geometrically isotypic), the sheaf will be called an \emph{irreducible
  trace sheaf}.
\end{definition}

We use similar terminology for the trace functions:

\begin{definition}[Trace function]\label{def-trace-weight}
  Let $p$ be a prime number. A $p$-periodic function $K(n)$ defined for
  $n\geq 1$, seen also as a function on $\Ff_p$, is a \emph{trace
    function} (resp. \emph{Fourier trace function}, \emph{isotypic trace
    function}) if there is some trace sheaf (resp. Fourier trace sheaf,
  resp. isotypic trace sheaf) $\sheaf{F}$ on $\Aa^1_{\Ff_p}$ such that
  $K$ is given by~(\ref{eq-iota-trace}).
\end{definition}

We need an invariant to measure the geometric complexity of a trace
function, which may be defined in greater generality.

\begin{definition}[Conductor]\label{def-conductor}
  For an $\ell$-adic constructible sheaf $\sheaf{F}$ on $\Aa^1_{\Fp}$,
  of rank $\rank(\sheaf{F})$ with $n(\sheaf{F})$ singularities in
  $\Pp^1$, and with
$$
\swan(\sheaf{F})=\sum_{x}\swan_{x}(\sheaf{F})
$$
the (finite) sum being over all singularities of $\sheaf{F}$, we
define the \emph{(analytic) conductor} of $\sheaf{F}$ to be
\begin{equation}\label{eq-conductor}
\cond(\sheaf{F})=\rank(\sheaf{F})+n(\sheaf{F})+\swan(\sheaf{F}).
\end{equation}
\par
If $K(n)$ is a trace function modulo $p$, its \emph{conductor} is the
smallest conductor of a trace sheaf $\sheaf{F}$ with trace function
$K$.
\end{definition}

With these definitions, our third main result, which together with
Theorem \ref{th-from-admissible-to-orthogonality} immediately implies
Theorem \ref{th-traceweight}, is very simple to state:

\begin{theorem}[Trace functions are good]
  \label{th-interpret-admissible}
  Let $p$ be a prime number, $N\geq 1$ and $\sheaf{F}$ an isotypic
  trace sheaf on $\Aa^1_{\Ff_p}$, with conductor $\leq N$. Let $K$ be
  the corresponding isotypic trace function. Then $K$ is $(p,aN^s)$-good
  for some absolute constants $a\geq 1$ and $s\geq 1$.
\end{theorem}

\begin{rem} 
  (1) This sweeping result encompasses the
  functions~(\ref{eq-weight-mixed}) and~(\ref{eq-hypk-weight}) and a
  wide range of algebraic exponential sums, as well as point-counting
  functions for families of algebraic varieties over finite fields.
  From our point of view, the uniform treatment of trace functions is
  one of the main achievements in this paper. In fact our results can
  be read as much as being primarily about trace functions, and not
  Fourier coefficients of modular forms. Reviewing the literature, we
  have, for instance, found several fine works in analytic number
  theory that exploit bounds on exponential sums which turn out to be
  special cases of the correlation sums \eqref{eq-correlation}
  (see~\cite{FrIw,HB3,IwActa,pitt,munshi}). Recent works of the
  authors confirm the usefulness of this notion
  (see~\cite{FKM2,FKMd3}).
\par
(2) Being isotypic is of course not stable under direct sum, but using
Jordan-H\"older components, any Fourier trace function can be written
as a sum (with non-negative integral multiplicities) of isotypic trace
functions, which allows us to extend many results to general trace
functions (see Corollary~\ref{cor-trace-norm}).
\end{rem}

\subsection{The $\ell$-adic Fourier transform and the Fourier-M\"obius
  group}\label{sec-ladic-fourier}

We now recall the counterpart of the Fourier transform at the level of
sheaves, which was discovered by Deligne and developped especially by
Laumon~\cite{laumon}. This plays a crucial role in our work.
\par
Fix a non-trivial additive character $\psi$ of $\Ff_p$ with values in
$\bar{\Qq}_{\ell}$.  For any Fourier sheaf $\sheaf{F}$ on $\Aa^1$, we
denote by $\sheaf{G}_{\psi}=\ft_{\psi}(\sheaf{F})(1/2)$ its
(normalized) \emph{Fourier transform sheaf}, where the Tate twist is
always defined using the choice of square root of $p$ in
$\bar{\Qq}_{\ell}$ which maps to $\sqrt{p}>0$ under the fixed
isomorphism $\iota$ (which we denote $\sqrt{p}$ or $p^{1/2}$). We will
sometimes simply write $\sheaf{G}$, although one must remember that
this depends on the choice of the character $\psi$. Then $\sheaf{G}$
is another Fourier sheaf, such that
$$
  \frtr{\sheaf{G}}{\Ff_p}{y}=
  -\frac{1}{p^{1/2}}\sum_{x\in \Ff_p}{\frtr{\sheaf{F}}{\Ff_p}{x}\psi(xy)}
$$
for any $y\in \Ff_p$ (see~\cite[Th. 7.3.8, (4)]{katz-esde}).
\par
In particular, if $K$ is given by~(\ref{eq-iota-trace}) and $\psi$ is
such that
$$
\iota(\psi(x))=e\Bigl(\frac{x}{p}\Bigr)
$$
for $x\in\Ff_p$ (we will call such a $\psi$ the ``standard character''
relative to $\iota$), then we have
\begin{equation}\label{eq-fourier-trace}
\iota(\frtr{\sheaf{G}}{\Ff_p}{y})= -\hat{K}(y)
\end{equation}
for $y$ in $\Zz$.
\par
A key ingredient in the proof of Theorem~\ref{th-interpret-admissible}
is the following geometric analogue of the set of correlation
matrices:

\begin{definition}[Fourier-M\"obius group]\label{def-fm}
  Let $p$ be a prime number, and let $\sheaf{F}$ be an isotypic trace
  sheaf on $\Aa^1_{\Ff_p}$, with Fourier transform $\sheaf{G}$ with
  respect to $\psi$.  The \emph{Fourier-M\"obius group}
  $\haut_{\sheaf{F}}$ is the subgroup of $\PGL_2(\bar{\Ff}_p)$ defined
  by
$$
\haut_{\sheaf{F}}=\{\gamma\in \PGL_2(\bar{\Ff}_p)\,\mid\,
\gamma^*\sheaf{G}\text{ is geometrically isomorphic to } \sheaf{G}
\}.
$$
\end{definition}

The crucial feature of this definition is that $\haut_{\sheaf{F}}$ is
visibly a group (it is in fact even an algebraic subgroup of
$\PGL_{2,\Fp}$, as follows from constructibility of higher-direct
image sheaves with compact support, but we do not need this in this
paper; it is however required in the sequel~\cite{FKM2}). The
fundamental step in the proof of Theorem~\ref{th-interpret-admissible}
is the fact that, for $\mcF$ of conductor $\leq M$, the set
$\Gg_{K,M}$ of correlation matrices is, for $p$ large enough in terms
of $M$, a subset of $\haut_{\sheaf{F}}$. This will be derived from the
Riemann Hypothesis over finite fields in its most general form (see
Corollary~\ref{cor-interpret}).

\subsection{Basic examples}
\label{sec-pedagogical}

We present here four examples where $\hautk{K}{M}$ can be determined
``by hand'', though sometimes this may require Weil's results on
exponential sums in one variable or even optimal bounds on exponential
sums in {\em three} variables.  This already gives interesting
examples of good functions.
\par
(1) Let $K(n)=e(u n/p)$. Then $\hat{K}(v)=p^{1/2}\delta_{v\equiv
  -u\mods{p}},$ so that $\wwd(K;\gamma)=0$ unless $\gamma\cdot
(-u)=-u$, and in the last case we have $\wwd(K;\gamma)=p$. Thus, if
$M\geq 1$, we have
$$
\hautk{K}{M}=\{\gamma\in\PGL_2(\Ff_p)\,\mid\, \gamma\cdot (-u)=-u\}
$$
and, for $1\leq M<p^{1/2}$, the function $K$ is \emph{not}
$(p,M)$-good (yet non-correlation holds). 
\par
Dually, we may consider the function 
$$
K(n)=p^{1/2}\delta_{n\equiv u \mods p}
$$
for some fixed $u\in\Fp$, for which the Fourier transform is $\hat
K(v)=e(uv/p)$. Then we get
$$
\wwd(K;\gamma)=\sum_{z\not=-d/c}e\Bigl(u\frac{z-(az+b)\ov{(cz+d)}}p\Bigr)\hbox{ for }\gamma=\begin{pmatrix}a&b\\c&d
\end{pmatrix}.
$$
\par
If $u=0$, this sum is $\geq p-1$ for {\em every} $\gamma$ and for
$1\leq M<p^{1/2}-1$, the function $K$ is \emph{not} $(p,M)$-good.
\par
For $u\not=0$, we get $|\wwd(K;\gamma)|=p$ if $a-d=c=0$,
$\wwd(K;\gamma)=0$ if $a-d\not=0$ and $c=0$ and otherwise, the sum is
a Kloosterman sum so that $|\wwd(K;\gamma)| \leq 2p^{1/2}$, by Weil's
bound.  In particular, for $M\geq 3$ and $p$ such that $p>3\sqrt{p}$,
$$
\hautk{K}{M}=\Bigl\{\begin{pmatrix}1&t\\0&1
\end{pmatrix}
\Bigr\}\subset \PGL_2(\Ff_p).
$$
\par
Thus $K$ is $(p,3)$-good for all $p\geq 17$.
\par
(2) Recall that the classical Kloosterman sums are defined by
$$
S(e,f;q)=\sum_{x\in (\Zz/q\Zz)^{\times}}{
e\Bigl(\frac{ex+f\bar{x}}{q}\Bigr)
}
$$
for $q\geq 1$ an integer and $e,f\in\Zz$.
\par
We consider $K(n)=S(1,n;p)/\sqrt{p}$ for $1\leq n\leq p$. By Weil's
bound for Kloosterman sums, we have $|K(n)|\leq 2$ for all $n$. We get
$\hat{K}(v)=0$ for $v=0$ and
$$
\hat{K}(v)=e\Bigl(-\frac{\bar{v}}{p}\Bigr)
$$
otherwise. For $\gamma=\begin{pmatrix}a&b\\c&d
\end{pmatrix}\in\PGL_2(\Ff_p)$, we find
$$
\wwd(K;\gamma)=
\sums_{z}{e\Bigl(\frac{\bar{z}-(cz+d)\overline{(az+b)}}{p}\Bigr)}
$$
where $\sums$ restricts the sum to those $z\notin\{0,-d/c,-b/a\}$ in
$\Ff_p$. According to the results of Weil, we have
$|\wwd(K;\gamma)| \leq 2p^{1/2}$ \emph{unless} the rational function
\begin{equation}\label{eq-which-fn}
\frac{1}{X}-\frac{cX+d}{aX+b}\in\Ff_p(X)
\end{equation}
is of the form $\phi(X)^p-\phi(X)+t$ for some constant $t\in\Ff_p$
and $\phi\in\Ff_p(X)$ (and of course, in that case the sum is $\geq
p-3$). Looking at poles we infer that in that later case $\phi$ is necessarily constant. Therefore, for $M\geq 3$ and
$p$ such that $p-3>3\sqrt{p}$, the set $\hautk{K}{M}$ is the set of
$\gamma$ for which~(\ref{eq-which-fn}) is a constant. A moment's
thought then shows that
$$
\hautk{K}{M}=\Bigl\{\begin{pmatrix}1&0\\t&1
\end{pmatrix}
\Bigr\}\subset \PGL_2(\Ff_p).
$$
\par
Thus $K$ is $(p,3)$-good for all $p\geq 17$.  
\par
(3) Let $K(n)=e(n^2/p)$. For $p$ odd, we get 
$$
\hat{K}(v)=\frac{\tau_p}{p^{1/2}}e\Bigl(-\frac{\bar{4}v^2}{p}\Bigr)
$$
by completing the square, where $\tau_p$ is the quadratic Gauss sum. 
Since $|\tau_p|^2=p$, we find for $\gamma\in\PGL_2(\Ff_p)$ as above
the formula
$$
\wwd(K;\gamma)=\sum_{z\not=-d/c}{
e\Bigl(\frac{\bar{4}(z^2-(az+b)^2\overline{(cz+d)}^2)}{p}\Bigr)
}
.
$$
For $p\geq 3$, Weil's theory shows that $|\wwd(K;\gamma)|\leq 2p^{1/2}$
for all $\gamma$ such that the rational function
$$
X^2-\frac{(aX+b)^2}{(cX+d)^2}
$$
is not  constant and otherwise $|\wwd(K;\gamma)|\geq p-1$.
\par
Thus for $M\geq 2$ and $p\geq 7$ (when $p-1>2p^{1/2}$), the set
$\hautk{K}{M}$ is the set of $\gamma$ for which this function is
constant: this requires $c=0$ (the second term can not have a pole),
and then we get the conditions $b=0$ and $(a/d)^2=1$, so that
$$
\hautk{K}{M}=\Bigl\{
1,\begin{pmatrix}-1&0\\0&1
\end{pmatrix}
\Bigr\}\subset B(\Ff_p)\subset \PGL_2(\Ff_p).
$$
\par
Thus that function $K$ is $(p,2)$-good for all primes
$p\geq 7$.
\par
(4) Let $K(n)=\chi(n)$ where $\chi$ is a non-trivial Dirichlet
character modulo $p$. Then we have
$\hat{K}(v)=\bar{\chi}(v)\frac{\tau(\chi)}{p^{1/2}}$ for all $v$, where
$$
\tau(\chi)=\sum_{x\in\Fp}{\chi(x)e\Bigl(\frac{x}{p}\Bigr)}
$$
is the Gauss sum associated to $\chi$. Then for $\gamma$ as above, we
have
\begin{align*}
  \wwd(K;\gamma)=
  \sum_{z\not=-b/a}{\bar{\chi}(\gamma\cdot z)\chi(z)}
    =\sum_{z\not=-b/a}{ \chi\Bigl(z\frac{cz+d}{az+b}\Bigr) }.
\end{align*}
\par
Again from Weil's theory, we know that $ |\wwd(K;\gamma)| \leq
2p^{1/2} $ \emph{unless} the rational function 
$$
\frac{X(cX+d)}{(aX+b)}
$$ 
is of the form $tP(X)^h$ for some $t\in\Ff_p$ and $P\in\Ff_p(X)$,
where $h\geq 2$ is the order of $\chi$ (and in that case, the sum has
modulus $\geq p-3$). This means that for $M\geq 2$, and $p\geq 11$,
the set $\hautk{K}{M}$ is the set of those $\gamma$ where this
condition is true. Looking at the order of the zero or pole at $0$, we
see that this can only occur if either $b=c=0$ (in which case the
function is the constant $da^{-1}$) or, in the special case $h=2$,
when $a=d=0$ (and the function is $cb^{-1}X^2$). In other words, for
$p\geq 11$ and $M\geq 2$, we have
$$
\hautk{K}{M}=\Bigl\{
\begin{pmatrix}
a&0\\
0&d
\end{pmatrix}\Bigr\}
$$
if $h\not=2$, and 
$$
\hautk{K}{M}=\Bigl\{
\begin{pmatrix}
a&0\\
0&d
\end{pmatrix}\Bigr\}
\cup
\Bigl\{
\begin{pmatrix}
0&b\\
c&0
\end{pmatrix}\Bigr\}
$$
if $\chi$ is real-valued. In both cases, these matrices are all in
$B(\Ff_p)\cup B(\Ff_p)w$, so that the function $\chi(n)$ is
$(p,2)$-good, for all $p\geq 11$.

\subsection{Notation} 

As usual, $|X|$ denotes the cardinality of a set, and we write
$e(z)=e^{2i\pi z}$ for any $z\in\Cc$.  If $a\in\Zz$ and $n\geq 1$ are
integers and $(a,n)=1$, we sometimes write $\bar{a}$ for the inverse
of $a$ in $(\Zz/n\Zz)^{\times}$; the modulus $n$ will always be clear
from context. We write $\Ff_p=\Zz/p\Zz$.
\par
By $f\ll g$ for $x\in X$, or $f=O(g)$ for $x\in X$, where $X$ is an
arbitrary set on which $f$ is defined, we mean synonymously that there
exists a constant $C\geq 0$ such that $|f(x)|\leq Cg(x)$ for all $x\in
X$. The ``implied constant'' refers to any value of $C$ for which this
holds. It may depend on the set $X$, which is usually specified
explicitly, or clearly determined by the context. We write $f(x)\asymp
g(x)$ to mean $f\ll g$ and $g\ll f$. The notation $n\sim N$ means that 
the integer $n$ satisfies the inequalities  $N<n\leq 2N$.
We denote the divisor function by $d(n)$.
\par
Concerning sheaves, for
  $a\not=0$, we will write $[\times a]^*\sheaf{F}$ for the pullback of
  a sheaf $\sheaf{F}$ on $\Pp^1$ under the map $x\mapsto ax$.
\par
For a sheaf $\sheaf{F}$ on $\Pp^1/k$, where $k$ is an algebraic
closure of a finite field, and $x\in\Pp^1$, we write $\sheaf{F}(x)$
for the representation of the inertia group at $x$ on the geometric
generic fiber of $\sheaf{F}$, and $\sheaf{F}_x$ for the stalk of
$\sheaf{F}$ at $x$. 
\par
For $\sheaf{F}$ a sheaf on $\Pp^1/k$, where now $k$ is a finite field
of characteristic $p$, and for $\nu$ an integer or $\pm 1/2$, we also
write $\sheaf{F}(\nu)$ for the Tate twist of $\sheaf{F}$, with the
normalization of the half-twist as discussed in
Section~\ref{sec-ladic-fourier} using the underlying isomorphism
$\iota\,:\, \bar{\Qq}_{\ell}\ra \Cc$. From context, there should be no
confusion between the two possible meanings of the notation
$\sheaf{F}(x)$.

\subsection{Acknowledgments} 

This paper has benefited from the input of many people. We would like
to thank V. Blomer, T. Browning, J. Ellenberg, C. Hall, H. Iwaniec,
N. Katz, E. Lindenstrauss, P. Nelson, R. Pink, G. Ricotta, P. Sarnak,
A. Venkatesh and D. Zywina for input and encouraging comments. We also
thank B. L\"offel and P. Nelson for their careful readings of the
manuscript. We particularly thank G. Harcos, whose decisive comments
on an earlier version of this paper have led to a significant
improvement on the value of the exponents as well as the referee who
read the paper with considerable attention, caught many slips and made
many helpful comments on the penultimate version of this paper.

\section{Some applications}\label{sec-app}

\subsection{Proof of
  Corollary~\ref{cor-intervals}}\label{sec-intervals}

We explain here how to derive bounds for sums over intervals with
sharp cut-offs from our main results. 
\par
Taking differences, it is sufficient to prove the following slightly
more precise bound: for any $\delta<1/8$ and any $1\leq X\leq p$, we have
$$
\sum_{1\leq n\leq X}\rho_f(n)K(n)
\ll_{\cond(K),f,\delta} X^{3/4}p^{1/4-\delta/2},
$$
since the right-hand side is always $\ll p^{1-\delta/2}$.

\begin{remark} Observe that, by taking $\delta$ close enough to $1/8$,
  we obtain here a stronger bound than the ``trivial'' estimate of
  size $\ll_{\cond(K),f} X$ coming from \eqref{RankinSelberg}, as long as
  $X\geq p^{3/4+\eta}$ for some $\eta>0$.
\end{remark}

By a dyadic decomposition it is sufficient to prove that for $1\leq
X\leq p/2$, we have
$$
\sum_{X\leq n\leq 2X}\rho_f(n)K(n)
\ll_{\cond(K),f,\delta}
X^{3/4}p^{1/4-\delta/2}
$$ 
for any $\delta<1/8$. We may assume that 
\begin{equation}\label{lowerboundforX}
X> 16p^{1-2\delta}
\end{equation}
for otherwise the trivial bound (see the previous remark) implies the
required bound.

Let $\Delta<1/2$ be a parameter, and let $W\,:\, [0,+\infty[\lra
[0,1]$ be a smooth function with $0\leq W\leq 1$, compactly supported
on the interval $[1-\Delta,2+\Delta]$, equal to $1$ on $[1,2]$ and
satisfying
$$
x^jW^{(j)}(x)\ll \Delta^{-j}
$$
for any $j\geq 0$. Then, provided $\Delta X\gg p^{3/5}$, we deduce
from~\eqref{RankinSelberg} that
$$
\sum_{X\leq n\leq 2X}\rho_f(n)K(n)= \sum_{n\geq
  1}\rho_f(n)K(n)W\Bigl(\frac{n}{X}\Bigr)+O(\|K\|_{\infty}\Delta
X),
$$
where the implied constant depends only on $f$. By
Theorem~\ref{th-traceweight} applied to $V(x)=W(px/X)$ with
$Q=\Delta^{-1}>2$ and $P=X/p\leq 1$, we have
$$
  \sum_{n\geq 1}\rho_f(n)K(n)W\Bigl(\frac{n}{X}\Bigr)
  \ll p^{1-\delta}(PQ)^{1/2}(P+Q)^{1/2}\\
  \ll \Delta^{-1}X^{1/2}p^{1/2-\delta}
$$
for any $\delta<1/8$ where the implied constant depends on $f$, $\cond(K)$ and $\delta$.
 Hence we derive
$$
\sum_{X\leq n\leq 2X}\rho_f(n)K(n)
\ll X\Bigl(\Delta+\Delta^{-1}p^{1/2-\delta}X^{-1/2}\Bigr).
$$
\par
We pick 
$$
\Delta=\Bigl(p^{1/2-\delta}X^{-1/2}\Bigr)^{1/2}.
$$
which is $<1/2$ by \eqref{lowerboundforX}.Then we get
$$
\Delta X\geq p^{-\delta/2}X\geq p^{1-5\delta/2}>p^{11/16}>p^{3/5}
$$
so the above inequality applies to give
$$
\sum_{X\leq n\leq 2X}\rho_f(n)K(n)
\ll X^{3/4}p^{1/4-\delta/2}.
$$
as we wanted.


\subsection{Characters and Kloosterman sums}

We first spell out the examples of the introduction involving the
functions \refs{eq-weight-mixed} and \refs{eq-hypk-weight}. We give the
proof now to illustrate how concise it is given our results, referring
to later sections for some details.

\begin{corollary}\label{cor-main}
  Let $f$ be any cusp form, $p$ a prime and $K$ given
  by
$$K(n)=\begin{cases}
  e\Bigl(\frac{\phi_1(n)}{p}\Bigr)\chi(\phi_2(n)),&\text{ if }
  p\nmid S_1(n)S_2(n)\\
  0&\text{otherwise}
\end{cases}
$$ 
or by
$$
K(n)=\begin{cases} \Phi\big(\hypk_m(\phi(n);p),
  \ov{\hypk_m(\phi(n);p)}\big),&\text{ if }
  p\nmid S(n)\\
  0&\text{otherwise}.
\end{cases}
$$
\par
Let $V$ satisfy $(V(C,P,Q))$. Then for any $\delta<1/8$, we have
$$
\tsum(f,K;p)\ll p^{1-\delta}(PQ)^{1/2}(P+Q)^{1/2},
$$
and 
$$
\sum_{n\in I}{\rho_f(n)K(n)}\ll p^{1-\delta/2}
$$
for any interval $I\subset[1,p]$, where the implied constant depends
only on $C,f,\delta$, $\phi_1$ and $\phi_2$ or $\phi$ and $\Phi$.
\end{corollary}

\begin{proof} 
  The first case follows directly from Theorem \ref{th-traceweight} if
  $\phi_1$ and $\phi_2$ satisfy the assumption of
  Theorem~\ref{pr-char-trace-weights}. Otherwise we have
  $K(n)=e(\frac{an+b}p)$ and the bound follows from \refs{oscillate}.
\par
In the second case, we claim that $\tnorm{K}{s}\ll 1$, where the
implied constant depends only on $(m,\phi,\Phi)$, so that
Corollary~\ref{cor-trace-norm} applies. Indeed, the triangle
inequality shows that we may assume that $\Phi(U,V)=U^uV^v$ is a
non-constant monomial. Let $\HYPK_{m,\phi}$ be the hyper-Kloosterman
sheaf discussed in \S \ref{hyperklo}, $\widetilde\HYPK_{m,\phi}$ its
dual. We consider the sheaf of rank $m^{u+v}$ given by
$$
\mcF=\HYPK_{m,\phi}^{\otimes
  u}\otimes\widetilde\HYPK_{m,\phi}^{\otimes v}
$$
with associated trace function
$$
K(n)=\bigl((-1)^{m-1}\hypk_m(\phi(n);p)\bigr)^u
\bigl((-1)^{m-1}\ov{\hypk_m}(\phi(n);p)\bigr)^v.
$$
\par
We have
$$
\cond(\mcF)\leq 5^{\alpha_{u+v}}(2m+1+\deg(RS))^{\beta_{u+v}}
$$
by combining Proposition~\ref{pr-hyperklo} and
Proposition~\ref{pr-bound-h1} (3) for some constants $\alpha_{n}$ and
$\beta_n$ (determined by $\alpha_0=0$, $\alpha_{n+1}=2\alpha_n+1$,
$\beta_0=1$, $\beta_{n+1}=2\beta_n+2$; note that this rought bound
could be improved easily).
\par
We replace ${\sheaf{F}}$ by its semisimplification (without changing
notation), and we write
$$
\sheaf{F}={\sheaf{F}}_1\oplus {\sheaf{F}}_2,\quad\quad 
K=K_1+K_2
$$
where ${\sheaf{F}}_2$ is the direct sum of the irreducible components
of $\sheaf{F}$ which are geometrically isomorphic to Artin-Schreier
sheaves $\sheaf{L}_{\psi}$, and $\sheaf{F}_1$ is the direct sum of the
other components. The trace function $K_2$ of $\sheaf{F}_2$ is a sum
of at most $m^{u+v}$ additive characters (times complex numbers of
modulus $1$) so
$$
\tnorm{K_2}{s}\leq m^{u+v}.
$$ 
\par
On the other hand, each geometrically isotypic component of
$\sheaf{F}_1$ have conductor bounded by that of $\sheaf{F}$, and
therefore
$$
\tnorm{K_1}{s}\leq (5m)^{u+v} (2m+1+\deg(RS))^{2s(u+v)}
$$
(Compare with Proposition~\ref{pr-shifted}).
\end{proof}

\subsection{Distribution of twisted Hecke orbits and horocycles}
\label{ssec-orbits}

We present here a geometric consequence of our main result. Let
$Y_0(N)$ denote the mo\-dular curve $\Gamma_0(N)\bash \Hh$. For a
prime $p$ coprime to $N$, we denote by $\tilde{T}_p$ the geometric
Hecke operator that acts on complex-valued functions $f$ defined on
$Y_0(N)$ by the formula
$$
\tilde{T}_p(f)(z)=\frac{1}{p+1}\sum_{t\in\Pp^1(\Ff_p)}f(\gamma_t\cdot
z)
$$
where
$$
\gamma_{\infty}=
\begin{pmatrix}p & 0 \\0 & 1\end{pmatrix},\quad
\gamma_t=\begin{pmatrix} 1 & t \\0 & p\end{pmatrix},\quad \text{ for
}t\in\Ff_p
$$ 
(note that this differs from the usual Hecke operator
$T_p=(p+1)p^{-1/2}\tilde{T}_p$ acting on Maass forms, defined
in~(\ref{eq-hecke-operator})).
\par
As we will also recall more precisely in
Section~\ref{sec-automorphic}, the $L^2$-space
$$
\mcL^2(N)=\Bigl\{g\,:\,  Y_0(N)\lra \Cc\,\mid\,
\int_{Y_0(N)}{|g(z)|^2\frac{dxdy}{y^2}}<+\infty\Bigr\},
$$ 
has a basis consisting of $\tilde{T}_p$-eigenforms $f$, which are
either constant functions, Maass cusp forms or combinations of
Eisenstein series, with eigenvalues $\nu_f(p)$ such that
\begin{equation}
\label{nontrivialhecke}
|\nu_f(p)|\leq 2p^{\theta-1/2}
\end{equation}
for some absolute constant $\theta<1/2$ (e.g., one can take
$\theta=7/64$ by the work of Kim and Sarnak~\cite{kim-sarnak}). This
bound implies the well-known equidistribution of the Hecke orbits
$\{\gamma_t\cdot \tau\}$ for a fixed $\tau\in Y_0(N)$, as $p$ tends to
infinity. Precisely, let
$$
\mu_{p,\tau}=
\frac{1}{p+1}\sum_{t\in\Pp^1(\Ff_p)}\delta_{\Gamma_0(N)\gamma_t\cdot
  \tau}
$$
where, for any $\tau\in\Hh$, $\delta_{\Gamma_0(N)\tau}$ denotes the
Dirac measure at $\Gamma_0(N)\tau\in Y_0(N)$. Then
$$
\mu_{p,\tau}\ra \mu
$$
as $p\ra+\infty$, in the weak-$*$ sense, where $\mu$ is the hyperbolic
probability measure on $Y_0(N)$.
\par
Note that all but one point of the Hecke orbit lie on the horocycle at
height $\Im(\tau)/p$ in $Y_0(N)$ which is the image of the segment
$x+i\Im(\tau)/p$ where $0\leq x\leq 1$, so this can also be considered
as a statement on equidistribution of discrete points on such
horocycles.
\par
We can then consider a variant of this question, which is suggested by
the natural parameterization of the Hecke orbit by the
$\Ff_p$-rational points of the projective line. Namely, given a
complex-valued function
$$
K\,:\, \Ff_p\ra \Cc 
$$
and a point $z\in Y_0(N)$, we define a \emph{twisted measure}
\begin{equation}\label{eq-twisted-measure}
\mutw{p}{\tau}{K}=\frac{1}{p}\sum_{t\in
  \Ff_p}K(t)\delta_{\Gamma_0(N)\gamma_t\cdot \tau},
\end{equation}
which is now a (finite) signed measure on $Y_0(N)$.
\par
We call these ``algebraic twists of Hecke orbits'', and we ask how
they behave when $p$ is large. For instance, $K$ could be a
characteristic function of some subset $A_p\subset \Ff_p$, and
we would be attempting to detect whether the subset $A_p$ is somehow
biased in such a way that the corresponding fragment of the Hecke
orbit always lives in a certain corner of the curve $Y_0(N)$. We will
prove that, when $1_{A_p}$ can be expressed or approximated by a
linear combination of the constant function $1$ and trace functions with
bounded conductors, this type of behavior is forbidden. For instance
if $A_p=\square(p)$ is the set of quadratic residues modulo $p$ one
has 
$$
1_{\square(p)}(t)=\frac{1}2\Bigl(1+\Bigl(\frac tp\Bigr)\Bigr),
$$ 
for $(\frac\cdot p)$ the Legendre symbol; this case is discussed
in~\cite[\S 1.2, 1.3]{mv}, where it is pointed out that it is
intimately related to the Burgess bound for short character sums and
to subconvexity bounds for Dirichlet $L$-functions of real characters
and twists of modular forms by such characters.
\par
Our result is the following:
 
\begin{theorem}\label{weightedshorthorocycles} 
  Let $M\geq 1$. For each prime $p$, let $K_p$ be an isotypic trace
  function modulo $p$ with conductor $\leq M$ and $I_p\subset [1,p]$
  an interval.
\par
  Let $\mutw{p}{\tau}{K_p,I_p}$ be the signed measure
$$
\mutw{p}{\tau}{K_p,I_p}=\frac{1}{|I_p|}\sum_{t\in
  I_p}K_p(t)\delta_{\Gamma_0(N)\gamma_t\cdot \tau}.
$$ 
\par
Then, for any given $\tau\in \Hh$, and $I_p$ such that $|I_p|\geq
p^{1-\delta}$ for some fixed $\delta<1/8$, the measures
$\mutw{p}{\tau}{K_p,I_p}$ converge to $0$ as $p\ra+\infty$.
\end{theorem}    

Here is a simple application where we twist the Hecke orbit by putting
a multiplicity on the $\gamma_t$ corresponding to the value of a
polynomial function on $\Ff_p$.

\begin{corollary}[Polynomially-twisted Hecke orbits]
  \label{cor-dist-supermorse}
  Let $\phi\in\Zz[X]$ be an arbitrary non-constant polynomial.  For
  any $\tau\in Y_0(N)$ and any interval of length $|I_p|\geq
  p^{1-\delta}$ for some fixed $\delta<\tfrac{1}{8}$, the sequence of
  measures
\begin{equation}\label{eq-short-supermorse}
\frac{1}{|I_p|}\sum_{\stacksum{x\in \Ff_p}{\phi(x)\in
    I_p}}\delta_{\Gamma_0(N)\gamma_{\phi(x)}\cdot \tau}
\end{equation}
converge to the hyperbolic probability measure $\mu$ on $Y_0(N)$ as
$p\ra +\infty$.
\end{corollary}

For $\phi$ non-constant, the set $A_p=\{\phi(t)\,\mid\,
t\in\Ff_p\}\subset \Ff_p$ of values of $\phi$ has positive density in
$\Ff_p$ for $p$ large, but the limsup of the density $|A_p|/p$ is
usually strictly less than $1$. The statement means, for instance,
that the points of the Hecke orbit of $\tau$ parameterized by $A_p$
can not be made to almost all lie in some fixed ``half'' of $Y_0(N)$,
when $\phi$ is fixed.
\par
These result could also be interpreted in terms of equidistribution of
weighted $p$-adic horocycles; similar questions have been studied in
different contexts for rather different weights in \cite{Strom,Ven,US}
(e.g., for short segments of horocycles).  Also, as pointed out by
P. Sarnak, the result admits an elementary interpretation in terms of
representations of $p$ by the quaternary quadratic form
$\det(a,b,c,d)=ad-bc$ (equivalently in terms of of integral matrices
of determinant $p$). Let
$$
M^{(p)}_{2}(\Zz)=\Bigl\{\gamma=\begin{pmatrix}a & b \\c &
  d\end{pmatrix}\in M_2(\Zz)\,\mid\, ad-bc=p\Bigr\}.
$$ 
\par
It is well-known that the non-trivial bound \refs{nontrivialhecke}
implies the equidistribution of $p^{-1/2}M^{(p)}_2(\Zz)$ on the
hyperboloid
$$
M^{(1)}_{2}(\Rr)=\Bigl\{\begin{pmatrix}x & y \\z & t\end{pmatrix}\in
M_2(\Rr)\,\mid\, xt-yz=1\Bigr\}=\SL_2(\Rr)
$$
with respect to the Haar measure on $\SL_2(\Rr)$ (see
\cite{SarnakKyoto} for much more general statements). Now, any matrix
$\gamma\in M^{(p)}_{2}(\Zz)$ defines a non-zero singular matrix modulo
$p$ and determines a point $z({\gamma})$ in $\Pp^1(\Fp)$, which is
defined as the kernel of this matrix (e.g. $z({\gamma_t})=-t$. By
duality, our results imply the following refinement: for any
non-constant polynomial $\phi\in\Zz[X]$, the subsets
$$
{M}_2^{(p),\phi}(\Zz)=\{\gamma\in M^{(p)}_2(\Zz)\,\mid\,
z({\gamma})\in \phi(\Fp)\},
$$ 
are still equidistributed as $p\ra\infty$ (compare
with~\cite[Cor. 1.4]{US}). 

\subsection{Trace functions over the primes} 

In the paper~\cite{FKM2}, we build on our results and on further
ingredients to prove the following statement:

\begin{theorem} 
  Let $K$ be an isotypic trace function modulo $p$, associated to a
  sheaf $\mcF$ with conductor $\leq M$, and such that $\mcF$ is not
  geometrically isomorphic to a direct sum of copies of a tensor
  product $\sheaf{L}_{\chi(X)}\otimes\sheaf{L}_{\psi(X)}$ for some
  multiplicative character $\chi$ and additive character $\psi$. Then
  for any $X\geq 1$, we have
$$
\sum_\stacksum{q\ \mathrm{prime}}{q \leq X}K(q)\ll
X(1+p/X)^{1/12}p^{-\eta},
$$
and
$$
\sum_{n \leq X}\mu(n)K(n)\ll
X(1+p/X)^{1/12}p^{-\eta}
$$
for any $\eta<1/48$. The implicit constants depend only on $\eta$ and
$M$. Moreover, the dependency $M$ is at most polynomial.
\end{theorem}

These bounds are non-trivial as long as $X\geq p^{3/4+\eps}$ for some
$\eps>0$, and for $X\geq p$, we save a factor $\gg_\eps p^{1/48-\eps}$
over the trivial bound. In other terms, trace functions of bounded
conductor do not correlate with the primes or the M\"obius function
when $X$ is greater than $X\geq p^{3/4+\eps}$.
\par 
This theorem itself has many applications when specialized to various
functions. We refer to \cite{FKM2} for these.

\section{Preliminaries concerning automorphic forms}
\label{sec-automorphic}

\subsection{Review of Kuznetsov formula}\label{automorphicforms}

We review here the formula of Kuznetsov which expresses averages of
products of Fourier coefficients of modular forms in terms of sums of
Kloosterman sums. The version we will use here is taken mostly
from~\cite{BHM}, though we use a slightly different normalization of
the Fourier coefficients.

\subsubsection{Hecke eigenbases}

Let $q\geq 1$ be an integer, $k\geq 2$ an even integer. We denote by
$\mcS_{k}(q)$, $\mcL^2(q)$ and $\mcL^{2}_{0}(q)\subset \mcL^2(q)$,
respectively, the Hilbert spaces of holomorphic cusp forms of weight
$k$, of Maass forms and of Maass cusp forms of weight $k=0$, level
$q$ and trivial Nebentypus (which we denote $\chi_0$), with respect to
the Petersson norm defined by
\begin{equation}\label{eq-petersson}
  \|g\|^2_{q}=\int_{\Gamma_0(q)\bash\Hh}|g(z)|^2y^{k_g}\frac{dxdy}{y^2},
\end{equation}
where $k_g$ is the weight for $g$ holomorphic and $k_g=0$ if $g$ is a
Maass form.
\par
These spaces are endowed with the action of the (commutative) algebra
$\Tt$ generated by the Hecke operators $\{T_{n}\mid n\geq
1\}$, where
\begin{equation}\label{eq-hecke-operator}
T_ng(z)=\frac{1}{\sqrt{n}}\sum_{\stacksum{ad=n}{(a,q)=1}}{
\Bigl(\frac{a}{d}\Bigr)^{k_g/2}
\sum_{0\leq b<d}g\Bigl(\frac{az+b}{d}\Bigr)
},
\end{equation}
where $k_g=0$ if $g\in \mcL^2(q)$ and $k_g=k$ if $g\in \mcS_k(q)$
(compare with the geometric operator $\tilde{T}_p$ of
Section~\ref{ssec-orbits}).
\par
Moreover, the operators $\{T_{n}\mid (n,q)=1\}$ are self-adjoint, and
generate a subalgebra denoted $\Tt^{(q)}$.  Therefore, the spaces
$\mcS_{k}(q)$ and $\mcL^2_{0}(q)$ have an orthonormal basis made of
eigenforms of $\Tt^{(q)}$ and such a basis can be chosen to contain
all $L^2$-normalized Hecke newforms (in the sense of Atkin--Lehner
theory). We denote such bases by $\mcB_{k}(q)$ and $\mcB(q)$,
respectively, and in the remainder of this paper, we tacitly assume
that any basis we select satisfies these properties.
\par
The orthogonal complement to $\mcL^2_{0}(q)$ in $\mcL^2(q)$ is spanned
by the Eisenstein spectrum $\mcE(q)$ and the one-dimensional space of
constant functions. The space $\mcE(q)$ is continuously spanned by a
``basis'' of Eisenstein series indexed by some finite set which is
usually taken to be the set $\{\mfa\}$ of cusps of $\Gamma_{0}(q)$.
It will be useful for us to employ another basis of Eisenstein series
formed of Hecke eigenforms: the adelic reformulation of the theory of
modular forms provides a natural spectral expansion of the Eisenstein
spectrum in which the Eisenstein series are indexed by a set of
parameters of the form
\begin{equation}\label{Eisparameters}
  \{(\chi,g)\,\mid\,  g\in\mcB(\chi)\},
\end{equation} where
$\chi$ ranges over the characters of
modulus $q$ and $\mcB(\chi)$ is some finite (possibly empty) 
set depending on $\chi$ (specifically, $\mcB(\chi)$
corresponds to an orthonormal basis in the space of the principal series
representation induced from the pair $(\chi,\ov\chi)$, 
but we need not be more precise). 
\par
With this choice, the spectral expansion for $\psi\in \mcE(q)$ can be
written
$$
\psi(z)=\sumsum_\stacksum{\chi}{g\in \mcB(\chi)} \int_{\Rr}
\peter{\psi,E_{\chi,g}(t)} E_{\chi,g}(t)\frac{dt}{4\pi}
$$
where the Eisenstein series $E_{\chi,g}(t)$ is itself a function from
$\Hh$ to $\Cc$. When needed, we denote its value at $z\in\Hh$ by
$E_{\chi,g}(z,t)$. 
\par
The main advantage of these Eisenstein series is the fact that they
are Hecke eigenforms for $\Tt^{(q)}$: for $(n,q)=1$, one has
$$
T_{n}E_{\chi,g}(t)=\lambda_{\chi}(n,t)E_{\chi,g}(t)
$$
with
$$
\lambda_{\chi}(n,t)=\sum_{ab=n}\chi(a)\ov{\chi(b)}
\Bigl(\frac{a}{b}\Bigr)^{it}.
$$

\subsubsection{Multiplicative and boundedness properties of Hecke
  eigenvalues}

Let $f$ be any Hecke eigenform of $\Tt^{(q)}$, and let
$\lambda_{f}(n)$ denote the corresponding eigenvalue for $T_{n}$,
which is real. Then for $(mn,q)=1$, we have
\begin{equation}\label{eq31}
  \lambda_f(m)\lambda_f(n) = \sum_{d \mid (m, n)} 
  \lambda_f(mn/d^2).
\end{equation}
\par
This formula \eqref{eq31} is valid for all $m$, $n$ if $f$ is an
eigenform for all of $\Tt$, with an additional multiplicative factor
$\chi_0(d)$ in the sum.

We recall some bounds satisfied by the Hecke eigenvalues. First, if
$f$ belongs to $\mcB_{k}(q)$ (i.e., is holomorphic) or is an
Eisenstein series $E_{\chi,f}(t)$, then we have the
Ramanujan-Petersson bound
\begin{equation}\label{eq-deligne-bound}
|\lambda_{f}(n)|\leq d(n)\ll_{\eps} n^\eps
\end{equation}
for any $\eps>0$. For $f\in\mcB(q)$, this is not known, but we will be
able to work with suitable average versions, precisely with the second
and fourth-power averages of Fourier coefficients. First, we have
\begin{equation}\label{secondpowerbound}
  \sum_{n \leq x} \abs{\lambda_f(n)}^2\ll x(q(1+|t_f|))^{\eps}, 
\end{equation}
uniformly in $f$, for any $x \geq 1$ and any $\eps>0$, where the
implied constant depends only on $\eps$
(see~\cite[Prop. 19.6]{DFIa}). Secondly, we have
\begin{equation}\label{fourthpowerbound}
  \sum_{\stacksum{n \leq x}{n\text{ squarefree}}} 
  \abs{\lambda_f(n)}^4 \ll_{f} x(\log x)
\end{equation}
for any $x \geq 1$ (see,
e.g.,~\cite[(3.3), (3.4)]{krw}).

\subsubsection{Hecke eigenvalues and Fourier coefficients}

For $z=x+iy\in\Hh$, we write the Fourier expansion of a modular form
$f$ as follows:
\begin{gather}
  f(z)=\sum_{n\geq 1} \rho_{f}(n)n^{(k-1)/2}e(nz)\quad\text{for}\quad
  f\in\mcB_{k}(q), \nonumber
  \\
  f(z)=\sum_{n \neq 0} \rho_f(n)|n|^{-1/2} W_{it_f}(4 \pi
  \abs{n}y)e(nx)\quad\text{for}\quad f\in\mcB(q),
\label{eq-fourier-expansion}
\intertext{where $1/4+t_f^2$ is the Laplace eigenvalue, and}
E_{\chi,g}(z,t)=c_{1,g}(t)y^{1/2+it}+c_{2,g}(t)y^{1/2-it} +\sum_{n
  \neq 0} \rho_{g}(n,t)|n|^{-1/2} W_{it}(4 \pi \abs{n}y)e(nx),
\nonumber
\end{gather}
where
\begin{equation}
\label{Wit}
W_{it}(y)=
\frac{e^{-y/2}}{\Gamma(it+\frac12)}
\int_0^\infty e^{-x}x^{it-1/2}\Bigl(1+\frac{x}y\Bigr)^{it-1/2}dx
\end{equation}
is a Whittaker function (precisely, it is denoted $W_{0,it_f}$
in~\cite[\S 4]{DFIa}; see also~\cite[9.222.2,9.235.2]{gr}.)
\par
When $f$ is a Hecke eigenform, there is a close relationship between
the Fourier coefficients of $f$ and its Hecke eigenvalues
$\lambda_{f}(n)$: for $(m,q)=1$ and any $n\geq 1$, we have
\begin{equation}\label{eig}
  \lambda_{f}(m)\rho_f(n)  
  = \sum_{d\mid(m,n)}\rho_{f}\left(\frac{mn}{d^2}\right),
\end{equation}
and moreover, these relations hold for all $m$, $n$ if $f$ is a
newform, with an additional factor $\chi_0(d)$.
\par
In particular, for $(m,q)=1$, we have
\begin{equation}\label{eig2}
  \lambda_{f}(m)\rho_f(1)=\rho_{f}({m}).
\end{equation} 

\subsubsection{The Petersson formula}

For $k\geq 2$ an even integer, the Petersson trace formula expresses
the average of product of Fourier coefficients over $\mcB_k(q)$ in
terms of sums of Kloosterman sums (see, e.g.~\cite[Theorem~9.6]{IwI}
and \cite[Proposition~14.5]{ant}): we have
\begin{equation}\label{pet}
  \frac{(k-2)!}{(4\pi)^{k-1}}
  \sum_{f\in\mcB_{k}(q)}\rho_{f}(n)\ov{\rho_{f}(m)}=
  \delta(m,n) + \Delta_{q,k}(m,n),
\end{equation}
with
\begin{equation}
\label{DeltaHdef}
\Delta_{q,k}(m,n)=2 \pi i^{-k} \sum_{q\mid c} \frac{1}{c} S(m, n;
c)J_{k-1}\left(\frac{4\pi \sqrt{mn}}{c}\right).
\end{equation}

\subsubsection{The Kuznetsov formula}

Let $\phi : [0, \infty[ \rightarrow \Cc$ be a smooth function
satisfying 
$$
  \phi(0) = \phi'(0) = 0,\quad\quad
  \phi^{(j)}(x) \ll_\eps (1+x)^{-2-\eps}\quad \text{ for } 0 \leq j
  \leq 3.
$$
\par
Let
\begin{equation}\label{besseltrans}
\begin{split}
  \dot{\phi}(k) &= i^k\int_0^{\infty} J_{k-1}(x) \phi(x) \frac{dx}{x},\\
  \tilde{\phi}(t) & = \frac{i}{2\sinh(\pi t)}\int_0^{\infty}
  \left(J_{2it}(x) - J_{-2it}(x)\right)\phi(x)\frac{dx}{x}, \\
  \check{\phi}(t) & = \frac{2}{\pi}\cosh(\pi
  t)\int_0^{\infty}K_{2it}(x)\phi(x)\frac{dx}{x}
\end{split}
\end{equation}
be Bessel transforms. Then for positive integers $m$, $n$ we have the
following trace formula due to Kuznetsov:
\begin{equation}
\label{Kuz}
\Delta_{q,\phi}(m,n)=\sum_{q \mid c} \frac{1}{c}
S(m, n;c)\phi\left(\frac{4\pi\sqrt{mn}}{c}\right)
\end{equation}
with
\begin{multline}\label{Deltadef} 
  \Delta_{q,\phi}(m,n) = \sumsum_\stacksum{k \equiv 0 \mods{2},\
    k>0}{g\in\mcB_k(q)} \dot{\phi}(k)\frac{(k-1)!}{\pi(4\pi)^{k-1}}
  \rho_{g}(m)\ov{\rho_{g}(n)} + \sum_{g\in\mcB(q)}
  \tilde{\phi}(t_g)\frac{4 \pi }{\cosh(\pi
    t_g)}\rho_{g}(m)\ov{\rho_g(n)}\\
  + \,\sumsum_\stacksum{\chi}{g\in
    \mcB(\chi)}\int_{-\infty}^{\infty}\tilde{\phi}(t)\frac{1}{\cosh(\pi
    t)} \rho_g\left(m,t\right)\ov{\rho_g\left(n, t\right)}\,dt.
\end{multline}
  
\subsection{Choice of the test function} \label{test choice}

For the proof of Theorem~\ref{th-from-admissible-to-orthogonality}, we
will need a function $\phi$ in Kuznetsov formula such that the
transforms $\dot{\phi}(k)$ and $\tilde\phi(t)$ are non-negative for
$k\in 2\Nn_{>0}$ and $t\in\Rr\cup(-i/4,i/4)$. Such $\phi$ is obtained
as a linear combination of the following explicit functions. For
 $2 \leq b < a$ two odd integers, we
take
\begin{equation}\label{defphi}
  \phi_{a, b}(x) = i^{b-a} J_a(x) x^{-b}.
\end{equation}
\par
By \cite[(2.21)]{BHM} we have
\begin{equation}\label{phitrafo}
\begin{split}
  &\dot{\phi}_{a, b}(k) = \frac{b!}{2^{b+1}\pi} \prod_{j=0}^b
  \left\{ 
    \left(\frac{a+b}{2}-j\right)^2-\left(\frac{k-1}{2}\right)^2\right\}^{-1}
  \asymp_{a, b}\, \pm \,k^{-2b-2},\\
  &\tilde{\phi}_{a, b}(t) = \frac{b!}{2^{b+1} \pi} \prod_{j=0}^{b}
  \left\{t^2 + \left(\frac{a+b}{2}-j\right)^2\right\}^{-1} \asymp_{a,
    b}\, (1+|t|)^{ - 2b-2}.
\end{split}
\end{equation}
\par
In particular,
\begin{equation}\label{posdot}
\begin{cases}
  \dot{\phi}_{a, b}(k) > 0 \quad &\text{for}\quad 2\leq k \leq a-b,\\
  (-1)^{(k-(a-b))/2}\dot{\phi}_{a, b}(k) > 0 \quad &\text{for}\quad a-b<k\leq a+b\\
  \dot{\phi}_{a, b}(k) > 0 \quad &\text{for}\quad a+b< k\hbox{ (since $b+1$ is even)},\\
  \tilde{\phi}_{a, b}(t)> 0\quad &\text{for}\quad
  t\in\Rr\cup(-i/4,i/4).
\end{cases}
\end{equation}
\par
Notice that if we have the freedom to choose $a$ and $b$ very large,
we can ensure that the Bessel transforms of $\phi_{a,b}$ decay faster
than the inverse of any fixed polynomial at infinity.

\section{The amplification method}\label{sec-proof}

\subsection{Strategy of the amplification}

We prove Theorem~\ref{th-from-admissible-to-orthogonality} using the
\emph{amplification method}; precisely we will embed $f$ in the space
of forms of level $pN$ (a technique used very successfully by Iwaniec
in various contexts \cite{Iw,CI}), as well as by
others~\cite{byk},~\cite{blomer-harcos}. The specific implementation
of amplification (involving the full spectrum, even for a holomorphic
form $f$) is based on~\cite{BHM}.
\par
We consider an automorphic form $f$ of level $N$, which is
either a Maass form with Laplace eigenvalue $1/4+t_f^2$, or a
holomorphic modular form of even weight $k_f\geq 2$, and which is an
eigenform of all Hecke operators $T_n$ with $(n,pN)=1$. 
\par
By viewing $f$ as being of level $2$ or $3$ if $N=1$, we can assume
that $N\geq 2$, which will turn out to be convenient at some point of
the later analysis.  We will also assume that $f$ is $L^2$-normalized
with respect to the Petersson inner product~(\ref{eq-petersson}).
\par
Finally, we can also assume that $p>N$, hence $p$ is coprime with
$N$. We will also assume that $p$ is sufficiently large with respect
to $f$ and $\eps$.  
\par
The form $f$ is evidently a cusp form with respect to the smaller
congruence subgroup $\Gamma_0(pN)$ and the function
\begin{equation}\label{eq-change-level}
\frac{f(z)}{[\Gamma_0(N):\Gamma_0(pN)]^{1/2}}=\frac{f(z)}{(p+1)^{1/2}}
\end{equation}
may therefore be embedded in a suitable orthonormal basis of modular
cusp forms of level $q=pN$, either $\mcB(q)$ or $\mcB_{k_f}(q)$.
\par
Let $a>b\geq 2$ be odd integers, to be chosen later (both will be
taken to be large), let $\phi=\phi_{a,b}$ be the
function~(\ref{defphi}) defined in section \ref{test choice}. We
define ``amplified'' second moments of the sums $\tsum(g,K;p)$, where
$g$ runs over suitable bases of $\mcB(q)$ and
$\mcB_{k_f}(q)$. Precisely, given $L\geq 1$ and any coefficients $(b_{\ell})$ defined
for $\ell\leq 2L$ and supported on $\ell\sim L$, and any modular form
$h$, we define an amplifier $B(h)$ by
$$
B(h)=\sum_{\ell\leq 2L}{b_{\ell}\lambda_h(\ell)}= \sum_{\ell\sim
  L}{b_{\ell}\lambda_h(\ell)}.
$$
\par
We will also use the notation 
\begin{equation}\label{eq-bgt}
  B(g,t)=B(E_{g,\chi}(t))
\end{equation}
for $\chi$ a Dirichlet character modulo $N$ and $g\in\mcB(\chi)$. 
\par
We then let
\begin{multline}\label{momentdef} 
  M(L)= \sum_{k \equiv 0 \mods{2},\ k>0} \dot{\phi}(k)(k-1)M(L;k)
  \\
  + \sum_{g\in\mcB(q)} \tilde{\phi}(t_g)\frac{4 \pi }{\cosh(\pi
    t_g)}|B(g)|^2|\tsum_V(g,K,p)|^2\\
  + \,\sumsum_\stacksum{\chi}{g\in
    \mcB(\chi)}\int_{-\infty}^{\infty}\tilde{\phi}(t)\frac{1}{\cosh(\pi
    t)} |B(g,t)|^2|\tsum_V(E_{\chi,g}(t),K,p)|^2\,dt,
\end{multline}
where
\begin{equation}\label{momentholodef} 
  M(L;k)=\frac{(k-2)!}{\pi(4\pi)^{k-1}}
  \sum_{g\in\mcB_k(q)}
  |B(g)|^2|\tsum_V(g,K,p)|^2,
\end{equation}
for any even integer $k\geq 2$.
\par
We will show:

\begin{proposition}[Bounds for the amplified
  moment]\label{pr-amplified}
  Assume that $M\geq 1$ is such that $K$ is $(p,M)$-good.  Let $V$ be
  a smooth compactly supported function satisfying Condition
  $(V(C,P,Q))$. Let $(b_{\ell})$ be arbitrary complex numbers supported
  on primes $\ell\sim L$, such that $|b_{\ell}|\leq 2$ for all $\ell$.
\par
For any $\eps>0$ there exist $k(\eps)\geq 2$, such that for any $k\geq
k(\eps)$ and any integers $a>b>2$ satisfying
$$
a-b\geq k(\eps),\quad a\equiv b\equiv 1\mods{2},
$$
we have
\begin{equation}\label{eq-amplified}
  M(L),\ M(L;k)\ll \{p^{1+\eps}LP(P+Q)
  +p^{1/2+\eps}L^3PQ^2(P+Q)\}M^{3}
\end{equation}
provided that
\begin{equation}\label{eq-cond-1st}
  p^{\eps}LQ<p^{1/4}.
\end{equation}
\par
The implied constants depend on $(C,\eps,a,b,k,f)$.
\end{proposition}

We will prove Proposition~\ref{pr-amplified} in
Sections~\ref{sec-estim-amplified} and~\ref{ssec-conclusion}, but
first we show how to exploit it to prove the main result.

From now on, we omit the fixed test-function $V$ and use the
simplified notation $\tsum_V(f,K;p)=\tsum(f,K;p)$. Also (and because
we will need the letter $C$ for another variable), we fix the sequence
$C=(C_\nu)_{\nu}$ and we will not mention the dependency in $C$ in our
estimates.

\subsection{From Proposition~\ref{pr-amplified} to
  Theorem~\ref{th-from-admissible-to-orthogonality}}

We assume here Proposition~\ref{pr-amplified} and proceed to the proof
of the main theorem. 
\par
The amplifier we use is due to Venkatesh. We put
\begin{equation}
\label{bldef}
b_{\ell}=
\begin{cases}
  \mathrm{sign}(\lambda_f(\ell)) &\text{ if } \ell\nmid pN
  \text{ is a prime $\ell\sim L$ and $\lambda_f(\ell)\not=0$},
  \\
  0&\text{otherwise.}
\end{cases}
\end{equation} 
(note the use of Hecke eigenvalues, and not Fourier coefficients,
here).
\par
With this choice, the pointwise bound $|b_{\ell}|\leq 1$ is obvious,
and on average we get
\begin{equation*}\label{eq-bound1}
  \sum_{\ell\sim L}{|b_{\ell}|}\leq \pi(2L)\leq 2L.
\end{equation*} 
\par
Moreover, for $L$ large enough in terms of $f$ and $L<p$, we have
\begin{equation}\label{eq-big-amplifier}
B(f)\gg \frac{L}{(\log L)^2}
\end{equation}
where the implied constant depends on $f$. Indeed, we have
$$
B(f)= \sum_{\stacksum{\ell\sim L}{\ell\nmid N}}|\lambda_f(\ell)|,
$$
which we bound from below by writing
$$
\frac{L}{\log L} \ll \sum_{\stacksum{\ell\sim L}{\ell\nmid
    N}}|\lambda_f(\ell)|^2 \ll \frac{L}{(\log L)^3}
+|\mathcal{L}|^{1/2} \Bigl( \sum_{\ell\sim
  L}{|\lambda_f(\ell)|^4}\Bigr)^{1/2}
$$
(using the Cauchy-Schwarz inequality and the Prime Number Theorem for
the Rankin-Selberg $L$-function $L(f\otimes f,s)$) where
$$
\mathcal{L}=\{\ell\sim L\,\mid\, \ell\nmid N,\ |\lambda_f(\ell)|>(\log
L)^{-1}\}.
$$
\par
Thus by~(\ref{fourthpowerbound}), we have
\begin{equation}\label{pluto}
B(f) \geq \frac{|\mathcal{L}|}{\log L}\gg_f \frac{L}{(\log L)^2}.
\end{equation}

Now we apply Proposition~\ref{pr-amplified} for this choice. We recall
from \refs{posdot} that we have
$$
\tilde\phi(t),\ \tilde\phi(t_g)>0,
$$
in the second and third terms of the sum defining $M(L)$, while for
$k\geq 2$, even, we have
$$
 \dot\phi(k)>0\text{ for $k\leq a-b$ or $k>a+b$}
$$
under our conditions on $a$ and $b$.
\par
Given $\eps>0$, we can choose $a$, $b$ large enough, both odd,
depending on $\eps$, so that $a-b\geq k(\eps)$, and we add a finite number of terms to $M(L)$ to form
$$
M(L)+2\sum_\stacksum{a-b<k\leq a+b}{\dot\phi(k)<0}|\dot\phi(k)|(k-1)M(L;k)
$$
which equals
\begin{multline}\label{momentdef2} 
\sum_{k \equiv 0
    \mods{2},\ k>0} |\dot{\phi}(k)|(k-1)M(L;k) + \sum_{g\in\mcB(q)}
  \tilde{\phi}(t_g)\frac{4 \pi }{\cosh(\pi
    t_g)}|B(g)|^2|\tsum(g,K,p)|^2\\
  + \,\sumsum_\stacksum{\chi}{g\in
    \mcB(\chi)}\int_{-\infty}^{\infty}\tilde{\phi}(t)\frac{1}{\cosh(\pi
    t)} |B(g,t)|^2|\tsum(E_{\chi,g}(t),K,p)|^2\,dt\\
  \ll \{p^{1+\eps}LP(P+Q)+p^{1/2+\eps}L^3PQ^2(P+Q)\}M^3,
\end{multline}
where the implied constant depends on $(f,\eps)$.
\par
Now all the terms of the right-hand side of the
equality~(\ref{momentdef2}) are non-negative. Applying positivity and
recalling~(\ref{eq-change-level}), we obtain
$$
(p+1)^{-1}|B(f)|^2|\tsum(f,K;p)|^2\ll
\{p^{1+\eps}LP(P+Q)+p^{1/2+\eps}L^3PQ^2(P+Q)\}M^3
$$
and hence
\begin{equation}\label{eq-final-bound}
  |\tsum(f,K;p)|^2\ll \left\{
p^{2+\eps}\frac{P(P+Q)}{L}+p^{3/2+\eps}LPQ^2(P+Q)
\right\}M^3(\log L)^{6}
\end{equation}
by~(\ref{eq-big-amplifier}), where the implied constant depends on
$(f,\eps)$.
\par
We let
\begin{equation}\label{eq-large-L}
L=\frac{1}{2}p^{1/4-\eps}Q^{-1},
\end{equation}
for arbitrarily small $\eps>0$ so that \refs{eq-cond-1st} is satisfied. Therefore, if
 $L$ is sufficiently large depending on $f$, we obtain
\begin{equation}\label{eq-final-tsum}
  \tsum(f,K;p)\ll M^{3/2}p^{7/8+\eps}(PQ)^{1/2}(P+Q)^{1/2}.
\end{equation}
On the other hand, if $L\ll_f 1$, we have $Q\gg_f \frac{1}2p^{1/4-\eps}$, and the
estimate~(\ref{eq-final-tsum}) is trivial.  Thus we obtain
Theorem~\ref{th-from-admissible-to-orthogonality}.

\begin{remark}
  In~\cite[p. 1707]{FKM2}, we quote a slighlty different choice of
  $L$. This was due to a minor slip in the proof
  of~(\ref{eq-final-bound}) in the first draft of this paper, which is
  corrected above.  Using the value~(\ref{eq-large-L}) in~\cite{FKM2}
  does not affect any of the main results of that paper.
\end{remark}
\par


\subsection{Packets of Eisenstein series} \label{sec-packet}

The above argument also yields a similar bound for packets of unitary
Eisenstein series, i.e., when $f$ is replaced by
$$
E_{\chi,g,\varphi}=\int_\Rr \varphi(t)E_{\chi,g}(t)dt
$$
where $\chi$ is a Dirichlet character of modulus $N$, $g\in\mcB(\chi)$
and $\varphi$ is some smooth compactly supported function. We have the
following:

\begin{proposition}[Twisted sums of Eisenstein packets]\label{packet
    bound}
  Let $p$ be a prime number and $M\geq 1$. Let $K$ be a $(p,M)$-good
  function, and $V$ a function satisfying $(V(C,P,Q))$.
\par
There exists an absolute constant $s\geq 1$ such that
$$
\tsum_V(E_{\chi,g,\varphi},K;p)\ll
M^sp^{1-\delta}(PQ)^{1/2}(P+Q)^{1/2}
$$
for any $\delta<1/8$, where the implied constant depends only on
$(N,\delta,\varphi)$.
\end{proposition}

\begin{proof}
Let $T\geq 0$ be such that the support of $\varphi$ is contained in
$[-T,T]$. Then we have
$$
|\tsum_{V}(E_{\chi,g,\varphi},K;p)|\leq
\int_{\Rr}{
|\tsum_V(E_{\chi,g}(t),K;p)\varphi(t)|dt
}
\leq \|\varphi\|_{\infty}\int_{-T}^T{
|\tsum_V(E_{\chi,g}(t),K;p)|dt
},
$$
and we will bound the right-hand side.
\par
Fix some $t_0\in[-T,T]$. For $t\in [-T,T]$ we let $B(g,t)$ denote the
amplifier~(\ref{eq-bgt}) for the coefficients
$$
b_{\ell}=
\begin{cases}
  \overline{\lambda_{\chi}(\ell,t_0)} &\text{ if $\ell\sim L$ is prime
    and coprime to $pN$},
  \\
  0&\text{ otherwise},
\end{cases}
$$
which satisfy $|b_{\ell}|\leq 2$, where we recall that
$$
\lambda_{\chi}(n,t_0)=\sum_{ab=n}\chi\Bigl(\frac{a}b\Bigr)
\Bigl(\frac{a}b\Bigr)^{it_0}
$$
gives the Hecke eigenvalues of $E_{\chi,g}(t_0)$. 
\par
Let $\alpha_p=\exp(-\sqrt{\log p})$. For $t$ such that $|t-t_0|\leq
\alpha_p$, and for $\ell$ prime with $\ell\sim L$, we have
$$
\ell^{\pm it}=\ell^{\pm i t_0}+O(\alpha_p^{1/2}),
$$
from which we deduce
$$
\lambda_{\chi}(\ell,t)=\lambda_{\chi}(\ell,t_0)+O(\alpha_p^{1/2}),
$$
and then
\begin{equation}\label{winnie}
B(g,t) =B(g,t_0) +O(L\alpha_p^{1/2}).
\end{equation}
\par
Our next task it to give an analogue of \eqref{eq-big-amplifier},
namely we prove lower-bound
\begin{equation}\label{donald}
B(g,t_0)\gg_{N,T} \frac{L}{\log^6 L}, 
\end{equation} 
for $L\geq L_0 (N,T)$, uniformy for $\vert t_0 \vert \leq T$. 
The
argument is similar to~\cite[Lemma 2.4]{FKM2}. We start from the
equality
$$
B(g,t_0)=\sum_{\ell \sim L}\, \bigl\vert \chi(\ell) \ell ^{it_0}
+\overline{\chi}(\ell) \ell^{-it_0}\bigr\vert\geq \frac12 \sum_{\ell \sim L}\, \bigl\vert \chi(\ell) \ell ^{it_0}
+\overline{\chi}(\ell) \ell^{-it_0}\bigr\vert^2.
$$
Restricting the summation to the primes $\ell\equiv 1 \bmod N$, we
obtain the lower bound
\begin{equation}\label{dark}
  B(g,t_0)\geq 2 \sum_{\substack{\ell \sim L \\ \ell\equiv 1 \bmod N}}
  \cos^2 (t_0 \log \ell).
\end{equation}
In~\cite[p. 1705]{FKM2}, the corresponding sum without the condition
$\ell\equiv 1\bmod N$ is shown to be $\gg L/(\log L)^6$. Since $N$ is
fixed, it is easy to include this condition in the proof of loc. cit.,
using the Prime Number Theorem in arithmetic progressions.  We leave
the details to the reader.
\par
Combining \eqref{winnie} and \eqref{donald}, we deduce
\begin{equation}\label{mickey}
B(g,t) \gg \frac{L}{\log^6 L},
\end{equation}
where the implied constant depends only on $N$ and $T$.  We therefore
get
\begin{align*}
  \frac{L^2}{(\log L)^{12}} \int_{|t-t_0|\leq \alpha_p}
  |\tsum(E_{\chi,g}(t),K;p)|^2dt&\ll \int_{|t-t_0|\leq\alpha_p}
  |B(g,t)|^2|\tsum(E_{\chi,g}(t),K;p)|^2dt,
\end{align*}
and the same argument used in the previous section leads to
$$
\int_{|t-t_0|\leq \alpha_p}|\tsum(E_{\chi,g}(t),K;p)|dt \ll
M^{3/2}p^{1-\delta}(PQ)^{1/2}(P+Q)^{1/2},
$$
for any $\delta<1/8$, the implied constant depending on
$(T,M,\delta)$. Finally we get
\begin{align*}
\int_{-T}^T|\tsum_V(E_{\chi,g}(t),K;p)|dt &\ll
M^{3/2}\alpha_p^{-1}p^{1-\delta}(PQ)^{1/2}(P+Q)^{1/2})
\\&\ll
M^{3/2}p^{1-\delta'}(PQ)^{1/2}(P+Q)^{1/2}
\end{align*}
for any $\delta'<\delta<1/8$, the implied constant depending on
$(\delta,T)$, by partitioning the interval $[-T,T]$ into roughly
$\alpha_p^{-1}=\exp(\sqrt{\log p})$ intervals of length $\alpha_p$.
\end{proof}

\begin{remark} The bounds \eqref{pluto} and \eqref{mickey} exhibit a polynomial
dependency in the parameters of $f$ or $E_{\chi,g,\varphi}$. This is due to the direct use of the prime number theorem for various $L$-functions. However, with more sophisticated Hoheisel-type estimates (see
  \cite{Moto} for instance), this dependency can be made
  polynomial. This is important for instance to obtain polynomial
  decay rates in $p$ in Theorem \ref{weightedshorthorocycles}.
\end{remark}

\begin{remark} Using the non-obvious amplifier of \cite{dfi2}
$$
b_{\ell}=
\begin{cases}
  \overline{\lambda_{f}(\ell)} &\text{ if $\ell\sim L$ is prime
    and coprime to $pN$},
  \\
    -1 &\text{ if $\ell=(\ell')^2$ for $\ell'\sim L$ a prime
     coprime to $pN$},
  \\
  0&\text{ otherwise},
\end{cases}
$$
and the identity $|\lambda_{f}(\ell)|^2-\lambda_{f}(\ell^2)=1$ for
$\ell$ prime it is possible to obtain a non trivial bound for the sum
$\tsum_V(f,K;p)$ when $f$ is of level $Np$ (rather than $N$); however
due to the lacunarity of the amplifier the resulting bounds are
weaker: the exponent $1/8$ in Theorem \ref{th-traceweight} and its
corollaries has to be replaced by $1/16$. The proof is a little bit
more involved as one has to consider more than 3 cases in \S
\ref{ssec-m3l} and we will not give it here.
\end{remark}

\section{Estimation of the amplified second moment}
\label{sec-estim-amplified}

We begin here the proof of Proposition~\ref{pr-amplified}. Obviously,
we can assume that $P\leq p$, $Q\leq p$.
\par
We start by expanding the squares in $B(g)$ and $|\tsum(g,K;p)|^2$,
getting
$$
M(L;k)=\frac{(k-2)!}{\pi(4\pi)^{k-1}}
\sum_{\ell_1,\ell_2}b_{\ell_1}\overline{b_{\ell_2}}
\sum_{n_1,n_2}K(n_1)\overline{K(n_2)}
V\Big(\frac{n_1}{p}\Bigr)V\Bigl(\frac{n_2}{p}\Bigr)
\sum_{g\in\mcB_k(q)} \lambda_g(\ell_1)\lambda_g(\ell_2)\rho_g(n_1)
\overline{\rho_g(n_2)}
$$ 
and similarly
\begin{multline*}
  M(L)=\sum_{\ell_1,\ell_2}b_{\ell_1}\overline{b_{\ell_2}}
  \sum_{n_1,n_2}K(n_1)\overline{K(n_2)}V\Bigl(\frac{n_1}{p}\Bigr)
  V\Bigl(\frac{n_2}{p}\Bigr) \\
  \times \Bigl\{ \sumsum_\stacksum{k \equiv 0\mods{2},\
    k>0}{g\in\mcB_k(q)}
  \dot{\phi}(k)\frac{(k-1)!}{\pi(4\pi)^{k-1}}\lambda_g(\ell_1)
  \lambda_g(\ell_2)\rho_g(n_1) \overline{\rho_g(n_2)}
  \\
  +\sum_{g\in\mcB(q)} \tilde{\phi}(t_g)\frac{4 \pi }{\cosh(\pi
    t_g)}\lambda_g(\ell_1)\lambda_g(\ell_2)\rho_g(n_1) \overline{\rho_g(n_2)}\\
  +\,\sumsum_\stacksum{\chi}{g\in
    \mcB(\chi)}\int_{-\infty}^{\infty}\tilde{\phi}(t)\frac{1}{\cosh(\pi
    t)} \lambda_{\chi}(\ell_1,t)\lambda_{\chi}(\ell_2,t)\rho_g(n_1,t)
  \overline{\rho_g(n_2,t)}\,dt\Bigr\}
\end{multline*}
where we used the fact that the Hecke eigenvalues $\lambda_g(\ell_2)$
and $\lambda_{\chi}(\ell_2,t)$ which are involved are real for
$\ell_2$ coprime to $pN$, because of the absence of Nebentypus.

\subsection{First decomposition}

We decompose these two moments as
$$
M(L)=M_d(L)+M_{nd}(L),\quad M(L;k)=M_{d}(L;k)+M_{nd}(L;k)
$$
depending on whether $\ell_1=\ell_2$ or $\ell_1\not=\ell_2$.

We begin with the ``diagonal'' terms $M_d(L),\ M_d(L;k)$ where
$\ell_1=\ell_2$, which are the only cases where $\ell_1$ and $\ell_2$
are not coprime.

\begin{lemma}\label{lm-md}
  Assume that $|K|\leq M$. For any $\eps>0$, we have
$$
M_d(L;k),\ M_d(L)\ll M^2p^{1+\eps}LP(P+1),
$$
where the implied constants depend only on $\eps$.
\end{lemma}

\begin{proof} 
  Consider $M_d(L)$: it decomposes as a sum of the holomorphic, Maass
  and Eisenstein contributions
$$
M_d(L)=M_{d,Hol}(L)+M_{d,Maa}(L)+M_{d,Eis}(L)
$$
where, for instance, we have
$$
M_{d,Maa}(L)=\sum_{g\in\mcB(q)} \tilde{\phi}(t_g)\frac{4 \pi
}{\cosh(\pi t_g)}\sum_{\ell\leq
    L}{|b_{\ell}|^2|\lambda_g(\ell)|^2} \Bigl|
  \sum_{n}K(n)\rho_g(n)V\Bigl(\frac{n}p\Bigr) \Bigr|^2.
$$
\par
By \refs{secondpowerbound} and the bound $|b_{\ell}|\leq 2$, we get
$$
\sum_{\ell\sim L}{|b_{\ell}|^2|\lambda_g(\ell)|^2} \leq
4 \sum_{\ell\sim
  L}{|\lambda_g(\ell)|^2}  \ll_{\eps} (p(1+|t_g|))^\eps L,
$$
where the implied constant is independent of $f$. We can then apply
the rapid decay~(\ref{phitrafo}) of $\tilde{\phi}(t)$ at infinity and
the large sieve inequality of Deshouillers--Iwaniec~\cite[Theorem~2,
(1.29)]{DI} to obtain
\begin{multline*}
M_{d,Maa}(L)\ll_{\eps} p^\eps L\sum_{g\in\mcB(q)}
\tilde{\phi}(t_g)\frac{(1+|t_g|)^\eps }{\cosh(\pi t_g)} \Bigl|
\sum_{n}K(n)\rho_g(n)V\Bigl(\frac{n}p\Bigr) \Bigr|^2\\
\ll
p^{\eps}L\Bigl(1+\frac{pP}{pN}\Bigr)M^2(pP)\ll p^{1+\eps}LPM^2(P+1)
\end{multline*}

where the implied constant depends only on $\eps$.
\par
The bounds for the holomorphic and Eisenstein portion are similar and
in fact slightly simpler as we can use Deligne's bound on Hecke
eigenvalues of holomorphic cusp form (or unitary Eisenstein series)
instead of \refs{secondpowerbound} (still using~\cite[Th. 2, (1.28),
(1.30)]{DI}). And the treatment of $M_d(L;k)$ is essentially included
in that of the holomorphic contribution.
\end{proof}

\subsection{The contribution of $\ell_1\not=\ell_2$}

The modular forms appearing in $M_{nd}(L)$ or $M_{nd}(L;k)$ are
Hecke-eigenforms for the Hecke operators $T(n)$ for $(n,q)=(n,pN)=1$,
hence we can combine the eigenvalues at the primes $\ell_1\not=\ell_2$
using the Hecke relation \refs{eig} and
$$
\lambda_g(\ell_1)\lambda_g(\ell_2)= \lambda_g(\ell_1\ell_2),
$$
obtaining
$$
 \lambda_g(\ell_1\ell_2)\rho_g(n_1)=
\sum_{d\mid (\ell_1\ell_2,n_1)}
{\rho_g\Bigl(\frac{\ell_1\ell_2 n_1}{d^2}\Bigr)}.
$$
\par
By the Petersson formula \refs{pet}, we write
$$
\pi M_{nd}(L;k)=M_1(L;k)+M_2(L;k)
$$
where $M_1(L;k)$ corresponds to the diagonal terms
$\delta(\ell_1\ell_2n_1d^{-2},n_2)$ while
$$ 
M_2(L;k)=\sum_{\ell_1\not=\ell_2}b_{\ell_1}\overline{b_{\ell_2}}
\sum_{{d\mid \ell_1\ell_2}}
\sum_{\stacksum{n_1,n_2}{d\mid n_1}}{
  K(n_1)\overline{K(n_2)}V\Bigl(\frac{n_1}p\Bigr)
  V\Bigl(\frac{n_2}p\Bigr)
  \Delta_{q,k}\Bigl(\frac{\ell_1\ell_2n_1}{d^2},n_2\Bigr)}
$$
where $\Delta_{q,k}$ is given in \refs{DeltaHdef}.
\par
On the other hand, by \refs{Kuz}, there is no diagonal contribution
for $M_{nd}(L)$, and we write
$$
M_2(L)=M_{nd}(L)=\sum_{\ell_1\not=\ell_2}b_{\ell_1}\overline{b_{\ell_2}}
\sum_{d\mid \ell_1\ell_2}
\sum_{\stacksum{n_1,n_2}{d\mid n_1}}{
  K(n_1)\overline{K(n_2)}V\Bigl(\frac{n_1}p\Bigr)V\Bigl(\frac{n_2}p\Bigr)
  \Delta_{q,\phi}\Bigl(\frac{\ell_1\ell_2n_1}{d^2},n_2\Bigr)},
$$
where $\Delta_{q,\phi}(m,n)$ is defined in \refs{Deltadef}. 

\begin{remark}
  One can obtain a ``trivial'' bound for $M_2(L)$ and $M_2(L;k)$ by
  applying the Cauchy-Schwarz inequality and again the large sieve
  inequalities of Deshouillers--Iwaniec~\cite[Theorem~2]{DI}, namely
\begin{align}
  M_2(L), k^{-1}M_2(L;k)&\ll_\eps p^{1+\eps}((P+1)L)^\eps
  LP(P+1)^{1/2}(L^2P+1)^{1/2}\nonumber
  \\
  &\ll pL^2(P+1)M^2\label{largesievebound}
\end{align}
where the implied constant depends on $(C,\eps,a,b)$.
\end{remark}

\subsection{Diagonal terms}\label{ssec-m1l}

We begin with $M_1(L;k)$: we have
\begin{align*}
  M_1(L;k)&=\sum_{\ell_1\not=\ell_2}b_{\ell_1}\overline{b_{\ell_2}}
  \sum_{{d\mid \ell_1\ell_2}}
  \sum_{\stacksum{n_1,n_2\geq 1}{d\mid n_1}}{
    K(n_1)\overline{K(n_2)}V\Bigl(\frac{n_1}p\Bigr)V\Bigl(\frac{n_2}p\Bigr)
    \delta\Bigl(\frac{\ell_1\ell_2n_1}{d^2},n_2\Bigr)}\\
  &=\sum_{\ell_1\not=\ell_2}b_{\ell_1}\overline{b_{\ell_2}}
  \sum_{{de=\ell_1\ell_2}} \sum_{\stacksum{n_1\geq
      1}{d\mid n_1}}
  K(n_1)\overline{K(en_1d^{-1})}V\Bigl(\frac{n_1}p\Bigr)
  V\Bigl(\frac{en_1/d}p\Bigr)\\
  &=\sum_{\ell_1\not=\ell_2}b_{\ell_1}\overline{b_{\ell_2}}
  \sum_{{de=\ell_1\ell_2}} \sum_{m\geq 1}
  K(dm)\overline{K(em)}V\Bigl(\frac{dm}p\Bigr)V\Bigl(\frac{em}p\Bigr).
\end{align*}
\par
Since $V$ has compact support in $[P,2P]$ the sum over $m$ is in fact of length $\ll \min(pP/d,pP/e)$. But since $de=\ell_1\ell_2$ with
$\ell_i\sim L$, we have
$$
\max(d,e)> L.
$$
\par
Thus, simply using the bound $|K(n)|\leq M$ and the boundedness of
$b_{\ell}$, we get:

\begin{lemma}\label{lm-complete}
  Let $K(n)$ be such that $|K|\leq M$ for some $M\geq 1$. 
  Then we have
$$
M_1(L;k)\ll pLPM^2.
$$
\end{lemma}

\subsection{Arranging the off-diagonal terms}

Now comes the most important case of $M_2(L)$ and $M_2(L;k)$. Their
shape is very similar, so we define
\begin{multline}\label{M2phi}
M_2[\phi]=\frac{1}{pN}\sum_{\ell_1\not=\ell_2}b_{\ell_1}\overline{b_{\ell_2}}
  \sum_{{d\mid \ell_1\ell_2}}
  \sum_{\stacksum{n_1,n_2}{d\mid n_1}}
  K(n_1)\overline{K(n_2)}V\Bigl(\frac{n_1}p\Bigr)V\Bigl(\frac{n_2}p\Bigr)
  \\
  \sum_{c\geq 1}c^{-1}S(\ell_1\ell_2n_1 d^{-2},n_2;cpN)
  \phi\left(
  \frac{4\pi}{cpN}\sqrt{\frac{\ell_1\ell_2n_1n_2}{d^2}} \right),
\end{multline}
for an arbitrary function $\phi$. We then have
$$
M_2(L)=M_2[\phi_{a,b}]
\hbox{ and }
M_2(L;k)=M_2[\phi_k]
$$
for $\phi_k=2\pi i^{-k}J_{k-1}$.
\par
We first transform these sums by writing
$$
M_2[\phi]=\sum_{\ell_1\not=\ell_2}b_{\ell_1}\overline{b_{\ell_2}}
\sum_{{de=\ell_1\ell_2}}M_2[\phi;d,e],
$$
where
$$
M_2[\phi;d,e]=\frac{1}{pN}\sum_{c\geq 1}c^{-1}
\tilde{\mathcal{E}}_{\phi}(c,d,e)
$$
and
\begin{align*}
  \tilde{\mathcal{E}}_{\phi}(c,d,e)&= \sum_{n_1}\sum_{n_2}
  S(en_1,n_2;cpN) K(dn_1)\overline{K(n_2)} \phi \Bigl(\frac{4\pi
    \sqrt{en_1n_2}}{cpN}\Bigr)V\Bigl(\frac{dn_1}p\Bigr)
  V\Bigl(\frac{n_2}p\Bigr)\\
  &= \sum_{n_1\geq 1}\sum_{n_2\geq 1} S(en_1,n_2;cpN)
  K(dn_1)\overline{K(n_2)}H_{\phi}(n_1,n_2),
\end{align*}
with
\begin{equation}\label{eq-def-hphi}
H_{\phi}(x,y)=\phi\Bigl(\frac{4\pi \sqrt{exy}}{cpN}\Bigr)
V\Bigl(\frac{dx}p\Bigr)V\Bigl(\frac{y}p\Bigr).
\end{equation}
\par
Having fixed $d$, $e$ as above, let $C=C(d,e)\geq 1/2$ be a
parameter. We decompose further
\begin{equation}\label{blue}
M_2[\phi;d,e]=M_{2,C}[\phi;d,e]+M_3[\phi;d,e]
\end{equation}
where $M_{2,C}[\phi;d,e]$ denotes the contribution of the terms with
$c>C$, and correspondingly
\begin{equation}\label{eq-split-size-c}
M_2[\phi]=M_{2,tail}[\phi]+M_3[\phi].
\end{equation}
\par
We begin by estimating those, assuming that
\begin{equation}\label{eq-bound-phi}
|\phi(x)|\leq B x^{\kappa}
\end{equation}
for some $\kappa\geq 1$, $B\geq 0$ and all $x>0$. Using the trivial bound
for Kloosterman sums and the bound $|K(n)|\leq M$, we get
\begin{align*}
  \tilde{\mathcal{E}}_{\phi}(c,d,e)&\ll M^2\sumsum_{n_1\ll pP/d,\
    n_2\ll pP}cp (en_1n_2)^{\kappa/2}(cp)^{-\kappa}
  \\
  &\ll M^2
  c^{-\kappa+1}\Bigl(\frac{e}{d}\Bigr)^{\kappa/2}p^{3}P^{2+\kappa}
\end{align*}
for all $c\geq 1$, the implied constant depending on $B$.
\par
For our specific choices of $\phi$, we note that we have the
upper-bound
\begin{equation}
\label{Jkbound}
|J_{k-1}(x)|\leq \min(1,x^{k-1})
\end{equation}
where the constant implied is absolute. Recalling the
definition~(\ref{defphi}), we obtain~(\ref{eq-bound-phi}) with
$\kappa=a-b$ for $\phi=\phi_{a,b}$ and with $\kappa=k-1$ for
$\phi=2\pi i^{-k} J_{k-1}$, and we note that in the latter case, the
constant $B$ is independent of $k$.  Then, summing over $c>C(d,e)$ 
, we obtain:

\begin{proposition}\label{pr-m2c}
  With notation as above, assuming that $|K|\leq M$, we have
\begin{gather*}
  M_{2,C}[\phi_{a,b};d,e]\ll M^2p^{2}CP^2
  \Bigl(\frac{P}{C}\sqrt{\frac{e}{d}}\Bigr)^{a-b},\\
  M_{2,C}[\phi_{k};d,e]\ll M^2p^{2}CP^2
  \Bigl(\frac{P}{C}\sqrt{\frac{e}{d}}\Bigr)^{k-1}
\end{gather*}
where the implied constant is absolute.
\end{proposition}

In view of this proposition, we choose
\begin{equation}
\label{Cchoice}
C=\max\Bigl( 1/2, p^{\delta} P\sqrt{\frac{e}{d}}\Bigr)\ll p^{\delta} LP,
\end{equation}
for some small parameter $\delta>0$ which is at our disposal. Then
taking $k=k(\delta)$ and $a=a(\delta)$, $b=b(\delta)$ so that $k$ and
$a-b$ are large enough, and summing over $\ell_1,\ell_2$ we see that the total contribution, $M_{2,tail}$,  to $M(L)$ and $M(L;k)$, of the terms
$M_{2,C}[\phi_{a,b};d,e]$ and $M_{2,C}[\phi_{k};d,e]$
is bounded by
\begin{equation}\label{est43}
  M_{2,tail}\ll p^{-10}L^2P^2M^2,
\end{equation}
so it is negligible.

\subsection{Estimating the off-diagonal terms}\label{ssec-m3l}

It remains to handle the complementary sum (see \eqref{blue}) which is
\begin{equation}\label{eq-m3}
  M_3[\phi;d,e]= \frac{1}{pN}\sum_{1\leq c\leq C}c^{-1}
  \tilde{\mathcal{E}}_{\phi}(c,d,e),
\end{equation}
where $C$ is defined by \eqref{Cchoice}. In particular, we can assume
$C\geq 1$ otherwise the above sum is zero.
\par
Recall that we factored the product of distinct primes $\ell_1\ell_2$
(with $\ell_i\sim L$) as $\ell_1 \ell_2 =de$. Hence we have three
types of factorizations of completely different nature, which we
denote as follows:
\begin{itemize}
\item Type $(L^2,1)$: this is when $d=\ell_1\ell_2$ and $e=1$, so that
  $L^2<d\leq 4L^2$;
\item Type $(1,L^2)$: this is when $d=1$ and $e=\ell_1\ell_2$, so that
  $L^2<e\leq 4L^2$;
\item Type $(L ,L)$: this is when $d$ and $e$ are both $\not=1$ (so
  $d=\ell_1$ and $e=\ell_2$ or conversely), so that $L <d\not= e\leq
  2L$.
\end{itemize}
\par
We will also work under the following (harmless)
restriction
\begin{equation}\label{p-epsilon-P}
  p^{\delta} P< L.
\end{equation}
\par
By the definitions \eqref{Cchoice} and \eqref{eq-m3}, we infer that $C<1$ hence

\begin{proposition}\label{veryeasy}
  Suppose that $(d,e)$ is of {\rm Type} $(L^2,1)$ and that
  \eqref{p-epsilon-P} is satisfied. Then we have the equality
 $$
 M_3[\phi;d,e] =0.
 $$
\end{proposition}
 
It remains to deal with the two types $(L,L)$ and $(1,L^2)$. We will
transform each of the sums $\tilde{\mathcal{E}}_{\phi}(c,d,e)$ to connect them
with the correlation sums $\wwd(K;\gamma)$ for suitable matrices
$\gamma$.  First, observing that $(c,p)=1$ because $C<p$ (by
combining~(\ref{eq-cond-1st}),~(\ref{Cchoice})
and~(\ref{p-epsilon-P})), the twisted multiplicativity of Kloosterman
sums leads to
\begin{equation}\label{eq-ephi}
  \tilde{\mathcal{E}}_{\phi}(c,d,e)=\sum_{0\leq x_1<cN}\sum_{0\leq x_2<cN}
  S(ex_1\bar{p},x_2\bar{p};cN) D(c,d,e,x_1,x_2),
\end{equation}
where
$$
D(c,d,e,x_1,x_2) = \sum_{n_1\geq 0} \sum_{n_2\geq 0}
K(df_1(n_1))\overline{K(f_2(n_2))}
S(ef_1(n_1)\overline{cN},f_2(n_2)\overline{cN};p)
H_{\phi}(f_1(n_1),f_2(n_2)),
$$
with
$$
f_i(x)=x_i+cNx.
$$
\par
We split the double sum over $n_1$, $n_2$ into congruence classes
modulo $p$, and apply the Poisson summation formula and the identity
$$
\frac{\overline{h}_1}{h_2}+\frac{\overline{h}_2}{h_1} \equiv
\frac{1}{h_1h_2}\mods{1}
$$
for non-zero coprime integers $h_1$ and $h_2$. This shows
that\footnote{\ We use the same notation $n_1$, $n_2$ for the dual
  variables, but note that they now range over $\Zz$.}
\begin{align*}
  D(c,d,e,x_1,x_2)&=\dblsum_{n_1,n_2\in\Zz}\frac{1}{(cpN)^2}\widehat
  H_{\phi}\Bigl(\frac{n_1}{cpN},\frac{n_2}{cpN}\Bigr)
  e\Bigl(\frac{x_1n_1+x_2n_2}{cpN}\Bigr)\\
&\quad\quad\quad\quad\quad\quad
  \times e\Bigl(-\overline{cN}\frac{x_1n_1+x_2n_2}p\Bigr)
  E(c,d,e,x_1,x_2,n_1,n_2)\\
&=\dblsum_{n_1,n_2\in\Zz}\frac{1}{(cpN)^2}\widehat
  H_{\phi}\Bigl(\frac{n_1}{cpN},\frac{n_2}{cpN}\Bigr)
  e\Bigl(\frac{\bar{p}x_1n_1+\bar{p}x_2n_2}{cN}\Bigr)
  E(c,d,e,x_1,x_2,n_1,n_2)
\end{align*}
with $\widehat H_{\phi}(x,y)$ the Fourier transform over $\Rr^2$ of
$H_{\phi}$ and
\begin{multline}\label{eq-sum-ek}
  E(c,d,e,x_1,x_2,n_1,n_2):=e\Bigl(\overline{cN}\frac{x_1n_1+x_2n_2}p\Bigr)
  \\
  \times\sum_{u_1,u_2( p)} K(df_1(u_1))\overline{K(f_2(u_2))}
  S(ef_1(u_1)\overline{cN},f_2(u_2)\overline{cN};p)
  e\Bigl(\frac{u_1n_1+u_2n_2}p\Bigr) \\
  =\sum_{u_1,u_2( p)} K(u_1)\overline{K(u_2)}
  S(e\overline{cdN}u_1,\overline{cN}u_2;p)
  e\Bigl(\frac{\overline{cdN}u_1n_1+\overline{cN}u_2n_2}p\Bigr).
\end{multline}
\par
Note that the last expression is now independent of $(x_1,x_2)$, so
that we will be justified to denote this simply by
$E(c,d,e,n_1,n_2)$. Opening the Kloosterman
sums in~(\ref{eq-sum-ek}) and changing the order of summation, we see
that
\begin{equation}\label{eq-e-general}
  E(c,d,e,n_1,n_2)=p
  \sum_{z\in\Ff_p^\times}\hat{K}(\overline{cN}(\overline{d}ez+\overline{d}n_1))
  \overline{\hat{K}(-\overline{cN}(z^{-1}+n_2))},
\end{equation}
and by a further change of variable this becomes
\begin{equation}\label{eq-sum-as-corr}
  E(c,d,e,n_1,n_2)=p
\wwd\Bigl(K;
\begin{pmatrix}
n_1 & (n_1n_2-e)/(cN)\\
cdN & dn_2
\end{pmatrix}
\Bigr).
\end{equation}

Our next step is to implement the summation over $x_1$ and $x_2$
modulo $cN$ in~(\ref{eq-ephi}): we have
$$
\dblsum_{x_1,x_2\mods{cN}}{
S(ex_1\bar{p},x_2\bar{p};cN)
e\Bigl(\frac{\bar{p}x_1n_1+\bar{p}x_2n_2}{cN}\Bigr)
}=
\begin{cases}
  (cN)^2&\text{ if } e\equiv n_1n_2\mods{cN},\ (n_2,cN)=1,
  \\
  0&\text{ otherwise,}
\end{cases}
$$
by orthogonality of characters modulo $cN$. Observe also that, since
$N\geq 2$, the congruence condition $e\equiv n_1n_2\mods{N}$ and the
fact that $(e,N)=1$ implies that $n_1n_2\not=0$ and is coprime with
$N$.
\par 
The outcome of the above computations is, for any $c\geq 1$, the
identity
\begin{equation}\label{outcome}
  \tilde{\mathcal{E}}_{\phi}(c,d,e)=\frac{1}{p} 
  \dblsum_{\stacksum{n_1n_2\not=0,\ (n_2,cN)=1}{n_1n_2\equiv
      e\mods{cN}}} \widehat
  H_{\phi}\Bigl(\frac{n_1}{cpN},\frac{n_2}{cpN}\Bigr) 
  \wwd\Bigl(K;\gamma(c,d,e,n_1,n_2)
  \Bigr)
\end{equation}
where
\begin{equation}
\label{gammadef}
\gamma(c,d,e,n_1,n_2):=\begin{pmatrix}
n_1 & (n_1n_2-e)/(cN)\\
cdN & dn_2
\end{pmatrix}\in \mathrm{M}_2(\Zz)\cap \GL_2(\Qq).
\end{equation}
\par
We make the following definition:

\begin{definition}[Resonating matrix]\label{def-resonating} For
  $n_1n_2\equiv e\mods{cN}$, the integral matrix
  $\gamma(c,d,e,n_1,n_2)$ defined by~(\ref{gammadef}) is called a {\it
    resonating matrix.}
\end{definition}

Observe that 
$$
\det(\gamma(c,d,e,n_1,n_2))=de
$$ 
and since $de$ is coprime with $p$, the reduction of
$\gamma(c,d,e,n_1,n_2)$ modulo $p$ provides a well-defined element in
$\PGL_2(\Fp)$.

\subsection{Estimating the Fourier transform}

Our next purpose is to truncate the sum over $n_1,n_2$ in
\eqref{outcome}. To do this, we introduce a new parameter:
\begin{equation}\label{defZ}
  Z= \frac{P}{cN}\sqrt{\frac{e}{d}}
  \asymp \begin{cases}  
    \frac{P}{cN} & \text{ if } (d,e) \text{ is of } {\rm Type}\, (L,L),\\
    \\
    \frac{LP}{cN} 
    & \text{ if } (d,e) \text{ is of } {\rm Type}\, (1,L^2  ).
\end{cases}
\end{equation}
\par
Note that, since $1\leq c\leq C=p^{\delta}P(e/d)^{1/2}$, we have
\begin{equation}\label{eq-lower-bound-z}
Z\gg_N p^{-\delta}.
\end{equation}
\par
We will use $Z$ to estimate the Fourier transform $\widehat
H_{\phi}(\frac{n_1}{cpN},\frac{n_2}{cpN})$. The first bound is given
by the following lemma:

\begin{lemma}\label{lm-integration-parts}
  Let $(d,e)$ be of {\rm Type} $(L,L)$ or of {\rm Type} $(1,L^2)$. Let
  $H_{\phi}$ and $Z$ be defined by \eqref{eq-def-hphi} and
  \eqref{defZ}. Assume that $V$ satisfies $(V(C,P,Q))$ and that
  $n_1n_2\not=0$.
\par
\emph{(1)} For $\phi=\phi_{a,b}$, we have
$$
\frac{1}{(pN)^2}\widehat H_{\phi_{a,b}}
\Bigl(\frac{n_1}{cpN},\frac{n_2}{cpN}\Bigr) \ll
\frac{P^2}{d}\frac{Z^{a-b}}{(1+Z)^{a+1/2}}
\Bigl(\frac{cdP^{-1}(Q+Z)}{|n_1|}\Bigr)^\mu
\Bigl(\frac{cP^{-1}(Q+Z)}{|n_2|}\Bigr)^\nu
$$
for all $\mu$, $\nu\geq 0$, where the implied constant depends on
$(N,\mu,\nu,a,b)$.
\par
\emph{(2)} For $\phi=2\pi i^{-k}J_{k-1}$, we have
$$
\frac{1}{(pN)^2}\widehat H_{\phi}
\Bigl(\frac{n_1}{cpN},\frac{n_2}{cpN}\Bigr) \ll \frac{P^2}{d}
\Bigl(\frac{cdP^{-1}(Q+Z)}{|n_1|}\Bigr)^\mu
\Bigl(\frac{cP^{-1}(Q+Z)}{|n_2|}\Bigr)^\nu
$$
for all $\mu$, $\nu\geq 0$, where the implied constant depends on
$(N,\mu,\nu)$, but not on $k$.
\end{lemma}

\begin{proof}
  (1) Recalling~(\ref{eq-def-hphi}) and~(\ref{defphi}), we have
\begin{multline}
  \frac{1}{p^2}\widehat H_{\phi_{a,b}}
  \Bigl(\frac{n_1}{cpN},\frac{n_2}{cpN}\Bigr)= \\ \nonumber
  \frac{1}{d}\iint_{\Rr^2} V ({x} )V ({y}) { i^{b-a}\Bigl({4\pi
      \frac{(e/d)^{1/2}}{cN} \sqrt{xy}}\Bigr)^{-b}J_{a}\Bigl({4\pi
      \frac{(e/d)^{1/2}}{cN} \sqrt{xy}}\Bigr)
    e\Bigl(-\frac{(n_1/d)x+n_2y}{cN}\Bigr) dxdy}.
\end{multline}
\par
We use the uniform estimates
$$
\Bigl(\frac{z}{1+z}\Bigr)^{\nu}J_a^{(\nu)}(2\pi z)\ll\frac{z^a}{(1+z)^{a+1/2}}
$$
for the Bessel function, valid for $z>0$ and $\nu\geq 0$, where the
implied constant depends on $a$ and $\nu$
(see~\cite[Chap. VII]{EMOT}). We also remark that $Z$ is the order of
magnitude of the variable inside $J_a (\cdots)$ in the above formula,
then integrating by parts $\mu$ times with respect to $x$ and $\nu$
times with respect to $y$, we get the result indicated.
\par
(2) This is very similar: since we want uniformity with respect to $k$,
we use the integral representation
$$
J_{k-1}(2\pi x)=\int_{0}^1e(-(k-1)t+x\sin(2\pi t))dt
$$
for the Bessel function (\cite[8.411]{gr}). After inserting it in the
integral defining the Fourier transform, we find the desired estimates
by repeated integrations by parts as before.
\end{proof}

Applying this Lemma with $\mu,\nu$ very large, remarking that in both
cases we have $dZ \leq LP$,
and appealing to the bound~(\ref{eq-bound-wwd}), namely
$$
|\wwd\bigl(K;\gamma(c,d,e,n_1,n_2)
\bigr)
|\leq M^2p,
$$
we see that, for any fixed $\eps>0$, the contributions to $
\tilde{\mathcal{E}}_{\phi}(c,d,e)$ of the integers $n_1$, $n_2$ with
\begin{equation}\label{eq-choice-N}
  |n_1|\geq   N_1=p^\eps \frac{cd(Q+Z)}{P},\quad
  \text{ or }\quad |n_2|\geq  N_2=\frac{N_1}{d}=
  p^{\eps}\frac{c(Q+Z)}{P}
\end{equation}
are negligible (see \eqref{outcome}).
\par
Thus we get:

\begin{proposition}[Off-diagonal terms]\label{pr-m3l}
  Let $(d,e)$ be of {\rm Type} $(L,L)$ or of {\rm Type} $(1,L^2)$. Let
  $\delta>0$ and $\eps>0$ be fixed. Let $C$, $N_1$ and $N_2$ be
  defined by \eqref{Cchoice} and \eqref{eq-choice-N}. Then for
  $\phi=\phi_{a,b}$ or $2\pi i^{-k}J_{k-1}$, we have
$$
M_3[\phi;d,e]=\frac{1}{pN}\sum_{c\leq C}c^{-1}
\mathcal{E}_{\phi}(c,d,e) +O(M^2p^{-2})
$$
where $\mathcal{E}_{\phi}$ is the subsum of
$\tilde{\mathcal{E}}_{\phi}$ given by
$$
\mathcal{E}_{\phi}(c,d,e) =\frac{1}{p} \dblsum_{\stacksum{1\leq
    |n_1|\leq N_1,\ 1\leq |n_2|\leq
    N_2}{\stacksum{(n_2,cN)=1}{{n_1n_2\equiv e\mods{cN}}}}} \widehat
H_{\phi}\Bigl(\frac{n_1}{cpN},\frac{n_2}{cpN}\Bigr) \wwd\Bigl(K;
\begin{pmatrix}
n_1 & (n_1n_2-e)/(cN)\\
cdN & dn_2
\end{pmatrix}
\Bigr).
$$
\par
The implied constant depends on $( \delta, \eps, N,a,b)$, but is
independent of $k$ for $\phi=2\pi i^{-k}J_{k-1}$.
\end{proposition}

\subsection{A more precise evaluation} 

In the range $|n_i|\leq N_i,\ i=1,2$ we will need a more precise
evaluation. We will take some time to prove the following result:

\begin{lemma}\label{more precise1L2} 
  Let $(d,e)$ be of {\rm Type} $(L,L)$ or of {\rm Type} $(1,L^2)$. Let
  $H_{\phi}$ and $Z$ be defined by \eqref{eq-def-hphi} and
  \eqref{defZ}. Assume that $V$ satisfies $(V(C,P,Q))$ and that
  $n_1n_2\not=0$.
\par
\emph{(1)} For $\phi=\phi_{a,b}$, we have
$$
\frac{1}{p^2}\widehat H_{\phi_{a,b}}
\Bigl(\frac{n_1}{cpN},\frac{n_2}{cpN}\Bigr)\ll p^{\delta} \frac{
  P^2}d\min\Bigl(\frac{1}{Z^{1/2}},\frac{Q}{Z}\Bigr),
$$
where the implied constant depends on $(C,a,b,N)$.
\par
\emph{(2)} For $\phi=\phi_{k}$, we have
$$
\frac{1}{p^2}\widehat H_{\phi_{k}} \Bigl(\frac{n_1}{cpN},\frac{n_2}{cpN}\Bigr)
\ll k^3p^{\delta}
\frac{ P^2}d\min\Bigl(\frac{1}{Z^{1/2}},\frac{Q}{Z}\Bigr),
$$
where the implied constant depends on $C$ and $N$.
\end{lemma}

\proof We consider the case $\phi=\phi_k$, the other one being
similar. We shall exploit the asymptotic oscillation and decay of the
Bessel function $J_{k-1} (z) $ for large $z$. More precisely, we use
the formula
$$
J_{k-1}(2\pi z)= \frac{1}{\pi z^{1/2}} \Bigl(\cos\Bigl(2\pi
z-\frac{\pi}2(k-1)-\frac{\pi}4\Bigr)+O\Bigl(\frac{k^3}{z}\Bigr)\Bigr)
$$
which is valid uniformly for $z>0$ and $k\geq 1$ with an absolute
implied constant (to see this, use the formula
$$
J_{k-1}(2\pi z)= \frac{1}{\pi z^{1/2}} \Bigl(\cos\Bigl(2\pi
z-\frac{\pi}2(k-1)-\frac{\pi}4\Bigr)+O\Bigl(\frac{1+(k-1)^2}{z}\Bigr)\Bigr)
$$
from, e.g.,~\cite[p.227, (B 35)]{IwI}, which holds with an absolute
implied constant for $z\geq 1+(k-1)^2$, and combine it with the bound
$\vert J_{k-1} (x) \vert \leq 1$.)  
\par
The contribution of the second term in  this expansion to
$$
\frac{1}{p^2}\widehat H_{\phi_{k}}
\Bigl(\frac{n_1}{cpN},\frac{n_2}{cpN}\Bigr)
$$ 
is bounded by
\begin{equation}
\label{Htrivial}
\ll \frac{P^2}d\frac{k^3}{Z^{3/2}}.
\end{equation}
\par
The contribution arising from the first term can be written as a
linear combination (with bounded coefficients) of two expression of
the shape
\begin{multline*}
  \frac{1}{dZ^{1/2}}\int_{\Rr_+^2}\Bigl(\frac{P}{\sqrt{xy}}\Bigr)^{1/2}
  V(x)V(y)e\Bigl(\frac{\pm2\sqrt{(e/d)xy}-(n_1/d)x-n_2y}{cN}\Bigr)dxdy
\\
  =\frac{8P^2}{dZ^{1/2}} \int_{\Rr_+^2}(2{xy})^{1/2}\,V(2Px^2)V(2Py^2)
  e\Bigl(-2P\frac{(n_1/d)x^2\mp 2\sqrt{e/d}xy+n_2y^2}{cN}\Bigr)dxdy.
\end{multline*}
\par
We write these in the form
\begin{equation}
  \frac{8P^2}{dZ^{1/2}}\int_{\Rr_+^2}G(x,y)e(F_\pm(x,y))dx\, dy,
 \label{2976}
\end{equation}
where we note that the function 
$$
G(x,y)=(2xy)^{1/2}V(2Px^2)V(2Py^2)
$$ 
is smooth and compactly supported in $[0,1]^2$, and -- crucially --
the phase
$$
F_\pm (x,y)=-2P\frac{(n_1/d)x^2\mp 2\sqrt{e/d}xy+n_2y^2}{cN}
$$
is a quadratic form.
\par
In particular, since $Z\gg p^{-\delta}$
(see~(\ref{eq-lower-bound-z})), we obtain a first easy bound
\begin{equation}\label{interH}
  \frac{1}{p^2}\widehat H_{\phi_{k}} 
  \Bigl(\frac{n_1}{cpN},\frac{n_2}{cpN}\Bigr)\ll 
 k^3p^{\delta} \frac{P^2}{dZ^{1/2}}.
\end{equation}
\par
We now prove two lemmas in order to deal with the oscillatory
integrals~(\ref{2976}) above, from which we will gain an extra factor
$Z^{1/2}$. We use the notation
$$
\varphi^{(i,j)}=\frac{\partial^{i+j} \varphi}{\partial^i x\partial^j
  y}
$$
for a function $\varphi$ on $\Rr^2$.

\begin{lemma}\label{partialintegration}
  Let $F(x,y)$ be a quadratic form and $G(x,y)$ a smooth function,
  compactly supported on $[0,1]$, satisfying the inequality
$$
\Vert G\Vert_\infty + \Vert G^{(0,1)}\Vert_\infty \leq G_0,
$$
where $G_0$ is some positive constant. Let $\lambda_2$ denote the
Lebesgue measure on $\Rr^2$.
\par
Then, for every $B>0$, we have
$$
\int_0^1\int_0^1 G(x,y) e\bigl( F (x,y)\bigr)dx \, d y\ll G_0
\bigl(\lambda_2 (G(B)) +B^{-1}\bigr),
$$
where
$$
G(B)=\bigl\{ (x,y)\in [0,1]^2\,\mid\, \vert F^{(0,1)}(x,y)\vert \leq
B\bigr\}
$$
and the implied constant is absolute.
\end{lemma}

\begin{proof} 
For $0\leq x \leq 1$, let 
$$
{\mathcal A}(x)=\bigl\{y\in [0,1]\,\mid \, \vert F^{(0,1)}(x,y)\vert
\leq B\bigr\},
$$
and $\overline{\mathcal A} (x)$ its complement in $[0,1]$.  Note that
${\mathcal A}(x)$ is a segment (possibly empty), with length
$\lambda_1 ({\mathcal A}(x))$. Using Fubini's formula, we write
\begin{align}
  \int_0^1\int_0^1 G(x,y) e\bigl( F (x,y)\bigr)dx \, d y&= \int_0^1
  \Bigl( \int_0^1 G(x,y) e\bigl( F (x,y)\bigr)d \, y\Bigr) d x
  \nonumber\\
  &= \int_0^1 I(x) d x,\label{red3}
\end{align}
say.  To study $I(x)$, we use the partition $[0,1]=\mathcal{A}(x)\cup
\overline{\mathcal{A}(x)}$, leading to the inequality
$$
\vert I(x) \vert \leq G_0\, \lambda_1 ({\mathcal A} (x)) +\Bigl\vert\,
\int_{\overline{\mathcal A} (x)} G(x,y)e\bigl( F (x,y)\bigr) dy
\Bigr\vert.
$$
\par
To simplify the exposition, we suppose that $\overline{\mathcal A}(x)$
is a segment of the form $]a(x),1]$ with $0\leq a(x)\leq 1$ (when it
consists in two segments, the proof is similar). Integrating by part,
we get
\begin{multline}
  \int_{\overline{\mathcal A} (x)} G(x,y) e\bigl( F (x,y)\bigr)dy
  =\int_{a(x)}^1 \frac{G }{F^{(0,1)}} (x,y) \cdot F^{(0,1)} (x,y)\cdot
  e\bigl( F (x,y)\bigr)dy
  \\
  = \Bigl[ \frac{G }{F^{(0,1)}} (x,y) \cdot e\bigl( F
  (x,y)\bigr)\Bigr]_{y=a(x)}^{y=1} -\int_{a(x)}^1 \Bigl(\frac{G
  }{F^{(0,1)}} (x,y)\Bigr)^{(0,1)} \cdot e\bigl( F (x,y)\bigr)\,dy.
\label{red2}
\end{multline}
\par
The first term in the right hand side of \eqref{red2} is $\ll
G_0B^{-1}$. The modulus of the second one is
$$
\leq G_0 \int_{a(x)}^1\Bigl\{ \frac{1}{\vert F^{(0,1)} \vert }
+\frac{\vert F^{(0,2)}\vert }{\vert F^{(0,1)}\vert ^2}\Bigr\}(x,y) \,
d y \ll G_0 B^{-1}
$$
since, on the interval of integration, $F^{(0,1)}$ has a constant sign
and $F^{(0,2)}$ is constant.  Inserting these estimations in
\eqref{red3} and using the equality
$$
\int_0^1 \lambda_1 ({\mathcal A}(x))\, d x=\lambda_2 (G(B)),
$$ we complete the proof.
\end{proof}

The following lemma gives an upper bound for the constant $\lambda_2
(G(B))$ that appears in the previous one.

\begin{lemma}\label{red4}  
  Let $F(x,y)= c_0 x^2+2 c_1 xy + c_2 y^2$ be a quadratic form with
  real coefficients $c_i$. Let $B>0$ and let $G(B)$ be the
  corresponding subset of $[0,1]^2$ as defined in Lemma
  \ref{partialintegration}. We then have the inequality
$$
\lambda_2 (G(B)) \leq B/\vert c_1 \vert.
$$
\end{lemma}

\begin{proof} 
  By integrating with respect to $x$ first, we can write
$$
\lambda_2 (G(B))=\int_0^1\lambda_1(\mathcal{B}(y)) \, d y,
$$
where
$$
\mathcal{B}(y)=\{x\in [0,1]\,\mid\, |2c_1x+2c_2y|=|F^{(0,1)}(x,y)|\leq
B\}.
$$
\par
This set is again a segment, of length at most $B/\vert
c_1\vert$. Integrating over $y$, we get the desired result.
\end{proof}

We return to the study of the integral appearing in \eqref{2976}. Here
we see easily that Lemma \ref{red4} applies with
$$
|c_1|=\frac{2P}{cN}\sqrt{\frac{e}{d}}=2Z,\quad\quad G_0\ll Q.
$$
\par
Hence, by Lemma \ref{partialintegration}, we deduce
$$
\int_{\Rr_{\geq 0}^2}G(x,y)e(F_\pm(x,y))dx\, d y\ll {Q}\bigl(B/Z
+B^{-1}\bigr),
$$
for any $B>0$.  Choosing $B=\sqrt{Z}$, we see that the above integral
is $\ll QZ^{-1/2}$.
\par
It only remains to gather \eqref{Htrivial}, \eqref{2976},
\refs{interH} with the bound $Z^{-3/2}\ll p^{\delta/2}Q/Z$ to complete
the proof of Lemma~\ref{more precise1L2}.  \qed

\subsection{Contribution of the non-correlating matrices}

From now on, we simply choose $\delta=\eps>0$ in order to finalize the
estimates. 
\par
We start by separating the terms according as to whether
$$
|\wwd(K;\gamma(c,d,e,n_1,n_2))|\leq Mp^{1/2}
$$ 
or not, {\em i.e.}, as to whether the reduction modulo $p$ of the
resonating matrix $\gamma(c,d,e,n_1,n_2)$ is in the set $\hautk{K}{M}$
of $M$-correlation matrices or not (see~(\ref{eq-hautk})). Thus we
write
$$
\mathcal{E}_{\phi}(c,d,e)=\mathcal{E}^{c}_{\phi}(c,d,e)+
\mathcal{E}^{n}_{\phi}(c,d,e),
$$
where
$$
\mathcal{E}^{c}_{\phi}(c,d,e)=\frac{1}{p} \dblsums_{\stacksum{1\leq
    |n_1|\leq N_1,\ 1\leq |n_2|\leq
    N_2}{\stacksum{(n_2,cN)=1}{{n_1n_2\equiv e\mods{cN}}}}} \widehat
H_{\phi}\Bigl(\frac{n_1}{cpN},\frac{n_2}{cpN}\Bigr) \wwd\Bigl(K;
\gamma(c,d,e,n_1,n_2)\Bigr),
$$
where $\dblsums$ restricts to those $(n_1,n_2)$ such that 
$$
\gamma(c,d,e,n_1,n_2)
\mods{p} \in \hautk{K}{M},
$$
and $\mathcal{E}^{n}_{\phi}$ is the contribution of the remaining
terms. Similarly, we write
\begin{align*}
M_3[\phi;d,e]&= \frac{1}{pN}\sum_{c\leq
  C}{c^{-1}\Bigl(\mathcal{E}^{n}_{\phi}(c,d,e)+
  \mathcal{E}^{c}_{\phi}(c,d,e)\Bigr)}+O(M^2p^{-2})
\\&=
M_3^{n}[\phi;d,e]+M_3^{c}[\phi;d,e]+O(M^2p^{-2}),
\end{align*}
say.
\par
We will treat $M_3^{n}[\phi;d,e]$ slightly differently, depending on
whether $(d,e)$ is of {\rm Type} $(L,L)$ or of {\rm Type} $(1,L^2)$. 
For $\mathsf{T}=(L,L)$ or $(1,L^2)$, we write
$$
M_3^{n,\mathsf{T}}[\phi]=
\sum_{\ell_1\not=\ell_2}b_{\ell_1}\overline{b_{\ell_2}}
\sum_{de=\ell_1\ell_2, \text{ type
    }\mathsf{T}}M_3^n[\phi;d,e].
$$
\par
Notice that in both cases we have
$$
\frac{N_1N_2}{c}=p^{2\eps}\Bigl(\frac{cdQ}{P}+\frac{(de)^{1/2}}{N}\Bigr)
\Bigl(\frac{Q}{P}+\frac{(e/d)^{1/2}}{cN}\Bigr)\gg \frac{L}{P}\gg 1,
$$
by \refs{p-epsilon-P}, \refs{defZ} and \refs{eq-choice-N}; here the implied constant depends on $N$. This
shows that the total numbers of terms in the sum
$\mathcal{E}_{\phi}(c,d,e)$ (or its subsums
$\mathcal{E}_{\phi}^n(c,d,e)$) is $\ll N_1N_2c^{-1}$.
\par
-- When $(d,e)$ is of {\rm Type} $(L,L)$, we appeal simply to Lemma
\ref{lm-integration-parts} with $\mu=\nu=0$, and obtain
\begin{align*}
  c^{-1}\mathcal{E}^{n}_{\phi}(c,d,e)&\ll
  c^{-1}Mp^{3/2}\dblsums_{\stacksum{1\leq |n_1|\leq N_1,\ 1\leq
      |n_2|\leq N_2}{\stacksum{(n_2,cN)=1}{{n_1n_2\equiv
          e\mods{cN}}}}}\frac{1}{p^2}\Bigl|\widehat
  H_{\phi}\Bigl(\frac{n_1}{cpN},\frac{n_2}{cpN}\Bigr)\Bigr|\\
  & \ll Mp^{3/2+2\eps} \frac{P^2}d \frac{N_1N_2}{c^2}\ll
  Mp^{3/2+2\eps}(Q+Z)^2 \ll Mp^{3/2+2\eps}\Bigl(Q+\frac{P}{c}\Bigr)^2,
\end{align*}
for $\phi=\phi_{a,b}$ or $\phi=\phi_k$.
\par
Summing the above over $c\leq C\ll p^{\eps}P$ and then over $(\ell_{1},
\ell_{2})$, and over the pairs $(d,e)$ of {\rm Type} $(L,L)$, we
conclude that 
\begin{equation}\label{contr(L,L)}
  M_3^{n,(L,L)}[\phi] \ll Mp^{1/2+3\eps} L^2 (Q^2P+PQ+P^2)\ll 
  Mp^{1/2+3\eps}L^2PQ(P+Q).
\end{equation}
\par
-- When $(d,e)$ is of {\rm Type} $(1,L^2)$, we have $d=1$ and
$$
c\leq C\ll p^{\eps}LP,\quad
Z\asymp \frac{LP}{cN},\quad
N_1=N_2\asymp p^\eps \frac{c(Q+LP/(cN))}{P}.
$$
\par
We now apply Lemma \ref{more precise1L2}. Considering the case of
$\phi=\phi_k$, we get
\begin{align*}
  c^{-1}\mathcal{E}^{n}_{\phi}(c,d,e)&\ll
  c^{-1}Mp^{3/2}\dblsums_{\stacksum{1\leq |n_1|\leq N_1,\ 1\leq
      |n_2|\leq N_2}{\stacksum{(n_2,cN)=1}{{n_1n_2\equiv
          e\mods{cN}}}}}\frac{1}{p^2}\Bigl|\widehat
  H_{\phi}\Bigl(\frac{n_1}{cpN},\frac{n_2}{cpN}\Bigr)\Bigr|\\
  &
  \ll Mk^3p^{3/2 +2\eps}\frac{P^2Q}Z\frac{N_1N_2}{c^2}
  \ll Mk^3 p^{3/2+2\eps}
  \frac{cQ}{LP}\Bigl(Q+\frac{LP}{c}\Bigr)^2.
\end{align*}
\par
If $\phi=\phi_{a,b}$, we obtain the same bound without the factor
$k^3$, but the implied constant then depends also on $(a,b)$.
\par
We then sum over $c\leq C$, over $(\ell_{1}, \ell_{2})$ and over the
pairs $(d,e)$ of {\rm Type} $(1,L^2)$, and deduce that
\begin{equation}\label{contr(1,L2)}
  M_3^{n,(1,L^2)}[\phi_k] \ll Mk^3p^{1/2+5\eps}
  L^3PQ^3,\quad\quad
  M_3^{n,(1,L^2)}[\phi_{a,b}] \ll Mp^{1/2+5\eps}
  L^3PQ^3.
\end{equation}
\par
Finally, in view of Proposition~\ref{veryeasy}, the combination of
\eqref{contr(L,L)} and \eqref{contr(1,L2)}, and a renaming of $\eps$,
show that
\begin{equation}\label{eq-m3l-bound}
  M_3^{n}[\phi_{a,b}]\ll Mp^{1/2+\eps}L^3PQ^2(P+Q),\quad\quad
  M_3^{n}[\phi_{k}]\ll Mk^3p^{1/2+\eps}L^3PQ^2(P+Q)
\end{equation}
for any $\eps>0$ where the implied constant depends on $(\eps,N,a,b)$
for $\phi=\phi_{a,b}$ and on $(\eps,N)$ for $\phi=\phi_k$.

\section{Contribution of the correlating
  matrices}\label{ssec-conclusion}

To conclude the proof of Proposition~\ref{pr-amplified} we evaluate
the contribution $M_3^{c}[\phi,d,e]$, corresponding to the resonating
matrices whose reduction modulo $p$ is a correlating matrix,
\emph{i.e.}, such that
\begin{equation}
\label{eq-resonating}
\gamma(c,d,e,n_1,n_2)=
\begin{pmatrix}
n_1 & (n_1n_2-e)/(cN)\\
cdN & dn_2
\end{pmatrix} \mods{p}\in \hautk{K}{M}.
\end{equation}
\par
In that case, we will use the estimate
\begin{equation}
\label{badbound}
\Bigl| \wwd\Bigl(K;
\gamma(c,d,e,n_1,n_2)\Bigr)\Bigr|\leq M^2p
\end{equation}
from~(\ref{eq-bound-wwd}).
\par
The basic idea is that correlating matrices are \emph{sparse}, which
compensates the loss involved in this bound.
\par
Corresponding to Definition~\ref{def-admissible}, we write
$$
\mathcal{E}^{c}_{\phi}(c,d,e)=\mathcal{E}^{b}_{\phi}(c,d,e)
+\mathcal{E}^{p}_{\phi}(c,d,e)+
\mathcal{E}^{t}_{\phi}(c,d,e)+\mathcal{E}^{w}_{\phi}(c,d,e)
$$
where the superscripts $b$, $p$, $t$, and $w$ denote the subsums of
$\mathcal{E}^{c}_{\phi}(c,d,e)$ where $(c,n_1,n_2)$ are such that the
resonating matrix $\gamma=\gamma(c,d,e,n_1,n_2)$ is of the
corresponding type in Definition~\ref{def-admissible} (in case a
matrix belongs to two different types, it is considered to belong to
the first in which it belongs in the order $b$, $p$, $t$, $w$).
\par
We write correspondingly
$$
M_3^{cor}[\phi,d,e]=M_3^{b}[\phi,d,e]+M_3^{p}[\phi,d,e]+
M_3^{t}[\phi,d,e]+M_3^{w}[\phi,d,e],
$$ 
and
$$
M_3^{c}[\phi]=M_3^{b}[\phi]+M_3^{p}[\phi]+M_3^{t}[\phi]+M_3^{w}[\phi].
$$
\par
Most of the subsequent analysis works when $d$ and $e$ are fixed, and
we will therefore often write
$$
\gamma(c,d,e,n_1,n_2)=\gamma(c,n_1,n_2)
$$
to simplify notation.
\par
The main tool we use is the fact that, when the coefficients of
$\gamma(c,d,e,n_1,n_2)$ are small enough compared with $p$, various
properties which hold modulo $p$ can be lifted to $\Zz$.

\subsection{Triangular and related matrices}
\label{ssec-borel}

Note that
$$
B(\Ff_p)\cup B(\Ff_p)w\cup wB(\Ff_p)=
\Bigl\{\begin{pmatrix}
a_1&b_1\\c_1&d_1
\end{pmatrix}
\in\PGL_2(\Ff_p)\,\mid\, a_1c_1d_1=0\Bigr\},
$$
so that a matrix $\gamma(c,n_1,n_2)$ can only contribute to
$\mathcal{E}^{b}_{\phi}(c,d,e)$ if $p|cNn_1n_2$. 
\par
If we impose the condition
\begin{equation}
\label{bounded1}
p^{3\eps}LQ<p
\end{equation}
(which will be strengthened later on), noting the bounds
$$
cd\leq dC\leq p^{\eps}P\sqrt{de}\ll p^{\eps}LP,
$$
and
$$
N_1=dN_2=p^{\eps}\frac{cd(Q+Z)}{P}=p^{\eps}
\Bigl(\frac{cdQ}{P}+\frac{cd}{P}\frac{P}{cN}\sqrt{\frac{e}{d}}\Bigr)
\ll p^{2\eps}LQ,
$$
we see that
$$
cdn_1n_2N\equiv 0\mods{p}
$$
is impossible, hence the sum $\mathcal{E}_{\phi}^b(c,d,e)$ is empty
and
\begin{equation}
\label{tribound}
M_3^{b}[\phi;d,e]=0.
\end{equation}

\subsection{Parabolic matrices}
\label{ssec-unipotent}

We now consider $\mathcal{E}_{\phi}^{p}(c,d,e)$, which is also easily
handled. Indeed, a parabolic $\gamma\in\PGL_2(\bar{\Ff}_p)$ has a
unique fixed point in $\Pp^1$, and hence any representative
$\tilde{\gamma}$ of $\gamma$ in $\GL_2(\bar{\Ff}_p)$ satisfies
$\Tr(\tilde{\gamma})^2-4\det(\tilde{\gamma})=0$. 
\par
Now if there existed some matrix $\gamma(c,n_1,n_2)$ which is
parabolic modulo $p$, we would get
$$
(n_1+dn_2)^2=4de=4\ell_1\ell_2\mods{p}.
$$
\par
Under the assumption
\begin{equation}
\label{bounded2}
p^{3\eps}LQ<p^{1/2}
\end{equation}
(which is stronger than~(\ref{bounded1})), this becomes an equality in
$\Zz$, and we obtain a contradiction since the right-hand side
$4\ell_1\ell_2$ is not a square. Therefore, assuming \refs{bounded2},
we have also
\begin{equation}
\label{parbound}
M_3^{p}[\phi;d,e]=0.
\end{equation} 

\subsection{Toric matrices}

We now examine the more delicate case of
$\mathcal{E}_{\phi}^{t}(c,d,e)$. Recall that this is the contribution
of matrices whose image in $\PGL_2(\Fp)$ belong to a set of $\leq M$
tori $\rmT^{x_i,y_i}$. We will deal with each torus individually, so
we may concentrate on those $\gamma(c,n_1,n_2)$ which (modulo $p$) fix
$x\not=y$ in $\Pp^1(\Fp)$. In fact, we can assume that $x$ and $y$ are
finite, since otherwise $\gamma$ would be treated by
Section~\ref{ssec-borel}.
\par
We make the stronger assumption
\begin{equation}
\label{bounded3}
p^{3\eps}LQ<p^{1/3}
\end{equation}
to deal with this case.
\par
We therefore assume that there exists a resonating matrix
$\gamma(c,n_1,n_2)$ whose image in $\PGL_2(\Fp)$ is contained in
$\rmT^{x,y}(\Fp)$. From \refs{bounded1}, we saw already that
$\gamma\mods{p}$ is not a scalar matrix. Now consider the integral
matrix
$$
2\gamma-\Tr(\gamma)\mathrm{Id}=\begin{pmatrix}
n_1-dn_2 & 2(n_1n_2-e)/(cN)\\
2cdN & dn_2-n_1
\end{pmatrix}=\begin{pmatrix}
u & v\\
w & -u
\end{pmatrix}
$$
(which has trace $0$). The crucial (elementary!) fact is that, since
$\gamma$ is not scalar, an element $\gamma_1$ in $\GL_2(\Fp)$ has
image in $\rmT^{x,y}$ if and only $2\gamma_1-\Tr(\gamma_1)\mathrm{Id}$
is proportional to $2\gamma-\Tr(\gamma)\mathrm{Id}$ (indeed, this is
easily checked if $x=0$, $y=\infty$, and the general case follows by
conjugation). 
\par
Hence, if a resonating matrix $\gamma_1=\gamma(c_1,m_1,m_2)$ has
reduction modulo $p$ in $\rmT^{x,y}$, the matrix
$$
2\gamma_1-\Tr(\gamma_1)\mathrm{Id}=
\begin{pmatrix}
  m_1-dm_2 & 2(m_1m_2-e)/(c_1N)\\
  2c_1dN & dm_2-m_1
\end{pmatrix}=\begin{pmatrix}
  u_1 & v_1\\
  w_1 & -u_1
\end{pmatrix}
$$ 
is proportional modulo $p$ to $\begin{pmatrix}u & v\\ w &
  -u\end{pmatrix}$, which gives equations
\begin{equation}\label{eq-equations}
uv_1-u_1v=uw_1-u_1w=vw_1-v_1w=0\mods{p}.
\end{equation}
\par
Because of \refs{bounded3}, one sees that these equalities modulo $p$
hold in fact over $\Zz$.  
We then get
$$
2u^2m_1m_2=u^2(c_1v_1N+2e)=(uc_1N)(uv_1)+2u^2e,
$$
where the first term is also given by
$$
(uc_1N)(uv_1)=\frac{(uw_1)(uv_1)}{2d}=\frac{(u_1w)(u_1v)}{2d}=cNv(m_1-dm_2)^2,
$$
so that 
\begin{equation}\label{eq-corrected-quad}
2u^2m_1m_2-cNv(m_1-dm_2)^2=2eu^2.
\end{equation}
We interpret this relation as $F(m_1,m_2)=2eu^2$, where
$$
F(X,Y)=-cNvX^2+(2u^2+2cNdv)XY-cNd^2vY^2
$$
is an integral binary quadratic form. For $u\not=0$, it is
non-singular, since its discriminant is given by
$$
(2u^2+2Ncdv)^2-4(cNv)(cNd^2v)=
4u^2(u^2+2Ncdv)=4u^2((n_1+dn_2)^2-4de)\not=0.
$$
\par
Note also that all the coefficients of $F(X,Y)$ are $\ll p^A$ for some
$A\geq 0$ and that similarly
$$
|m_1|, |m_2|\leq p^A.
$$
\par
By a classical result going back to Estermann (see,
e.g.,~\cite[Theorem 3]{HB}), the number of integral solutions $(x,y)$
to the equation
$$
F(x,y)=2eu^2
$$
such that $|x|$, $|y|\leq p^A$ is bounded by $\ll p^\eps$ for any
$\eps>0$. But when $m_1$ and $m_2$ are given solutions, the value of
$c_1$ is uniquely determined from the second equation
in~(\ref{eq-equations}). Hence the number of possible triples
$(c_1,m_1,m_2)$ is bounded by $\ll_\eps p^\eps$.
\par
Similarly, if $u=0$, we have $m_1-dm_2=n_1-dn_2=0$, and the
third equation $vw_1-v_1w-0$ becomes
$$
c_1^2(dn_2^2-e)=c^2(dm_2^2-e).
$$
We view this as $G(c_1,m_2)=-ec^2$ where
$$
G(X,Y)=(dn_2^2-e)X^2-(dc^2)Y^2.
$$
\par
This is again a non-degenerate integral quadratic form (note that
$dn_2^2-e\not=0$ since $d$ and $e$ are coprime) with coefficients $\ll
p^A$, and the pairs $(x,y)=(c_1,m_2)$ also satisfy $|x|$, $|y|\ll
p^A$, for some $A\geq 0$. Thus the number of solutions $(c_1,m_2)$ to
$G(c_1,m_2)=-ec^2$ is $\ll p^{\eps}$ for any $\eps>0$. Since
$(c_1,m_2)$ determine $(c_1,m_1,m_2)=(c_1,dm_2,m_2)$, we
get the same bound $\ll p^{\eps}$ for the number of possible triples
$(c_1,m_1,m_2)$.
\par
Using Lemma \ref{more precise1L2} and \refs{badbound}, we then deduce
(for a single torus)
\begin{align*}
  \frac{1}p\sum_{c\leq C}c^{-1}\mathcal{E}^{t}_{\phi_k}(c,d,e) &\ll
  M^2p^{1+\eps} \max_\stacksum{c\leq C}{1\leq |n_i|\leq
    N_i}\frac{1}{cp^2}\Bigl|\widehat
  H_{\phi_k}\Bigl(\frac{n_1}{cpN},\frac{n_2}{cpN}\Bigr)\Bigr|\\
  &\ll M^2k^3 p^{1+\eps} \max_{c\leq C}\frac{P^2}{d}\frac{Q}{cZ} \ll
M^2  k^3p^{1+\eps} \frac{PQ}{L}
\end{align*}
and similarly, without the factor $k^3$, for $\phi_{a,b}$. Hence,
multiplying by the number $\leq M$ of tori and summing up over $\ell_1,\ell_2,d,e$, we have
\begin{equation}
\label{torbound}
M_3^{t}[\phi_{a,b}]\ll M^3p^{1+\eps}LPQ,\quad\quad
M_3^{t}[\phi_{k}]\ll M^3k^3p^{1+\eps}LPQ,
\end{equation}
for any $\eps>0$, where the implied constant depends on
$(\eps,N,a,b)$.

\subsection{Normalizers of tori}

We now finally examine the contribution of $\hauti{w}{K}{M}$,
\emph{i.e.}, of resonating matrices $\gamma(c,n_1,n_2)$ whose image in
$\PGL_2(\Fp)$ are contained in the non-trivial coset of the normalizer
of one of the tori $\rmT^{x_i,y_i}$. Again, we may work with a fixed
normalizer $\rmN^{x,y}$, and we can assume that $x$ and $y$ are
finite. Denote by $R$ the set of resonating matrices with image in
$\rmN^{x,y}-\rmT^{x,y}$.
\par
Suppose that $\gamma=\gamma(c,n_1,n_2)$ is in $R$. We then have
$$
\gamma^2\equiv \det(\gamma)\mathrm{Id}=de\ \mathrm{Id} \mods{p},
$$
and 
$$
\Tr(\gamma)=n_1+dn_2=0 \mods{p}.
$$ 
Assuming, as we do, that \refs{bounded3} holds, then we deduce
$$
n_1=-dn_2,\ \gamma^2=de\mathrm{Id}
$$
over $\Zz$. In particular, $\gamma(c,n_1,n_2)$ only depends on the two
parameters $(c,n_2)$ and we will denote
$$\gamma(c,n_2):=\gamma(c,-dn_2,n_2).$$
\par
Fix some dyadic parameter $D$ with $1\leq D\leq C$. We restrict our
attention first to matrices $\gamma(c,n_2)\in R$ with $D/2\leq c\leq
D$; denote by $R_D$ the set of these matrices. Our aim is to show that
the total number of resonating
matrices in $R_D$ is $\ll_\eps p^{\eps}$ for any $\eps>0$.
\par
We distinguish two cases. If $R_D$ has at most one element up to
multiplication by $\pm 1$ we are obviously done. Otherwise, let
$\gamma_1=\gamma(c_1,n_1)$ and $\gamma_2=\gamma(c_2,n_2)$ be two
elements of $R_D$ with $\gamma_2\not=\pm \gamma_1$. We denote
$$
\gamma=\gamma_1\gamma_2.
$$ 
\par
Because of \eqref{bounded3} we see that the reduction modulo $p$ of
$\gamma_1$ and $\gamma_2$ are not scalar multiples of each other, and
similarly $\gamma\mods p$ is not a scalar matrix. On the other hand,
$\gamma\mods p\in \rmT^{x,y}$ which implies that the matrix
$2\gamma-\Tr(\gamma)\mathrm{Id}\mods p$ anti-commutes with the
elements of $\rmN^{x,y}-\rmT^{x,y}$:
\begin{equation}
\label{anticommute}
\text{ for all }
\sigma\in\ \rmN^{x,y}-\rmT^{x,y},\text{ we have }
\sigma(2\gamma-\Tr(\gamma)\mathrm{Id})=-(2\gamma-\Tr(\gamma))
\sigma\mods{p}.
\end{equation}
\par
Finally, let $\gamma_3=\gamma(c_3,n_3)\in R_D$. Writing
$$
2\gamma-\Tr(\gamma)\mathrm{Id}=\begin{pmatrix}
u & v\\
w & -u
\end{pmatrix}
$$
the anti-commutation relation leads to the relation
$$
-2udn_3+vNdc_3-w\frac{dn_3^2+e}{c_3N}=0\mods{p}.
$$
\par
Looking at the sizes of $u$, $v$, $w$, and using the fact that
$1/2\leq c_i/c_j\leq 2$, we see that if we make the stronger
assumption
\begin{equation}
\label{bounded4}
p^{3\eps}LQ<p^{1/4},
\end{equation}
this equation is valid over $\Zz$ (for instance,
$$
udn_3=\frac{c_1}{c_2}((dn_2)^2+de)dn_3-\frac{c_2}{c_1}((dn_1)^2+de)dn_3,
$$
and the other two are similar). This means that
$$
F(c_3,n_3)=ew
$$
where
$$
F(X,Y)=dvN^2X^2-2duNXY-dwY^2
$$
is again an integral binary quadratic form. Since its discriminant is
$$
(2du)^2-4(dv)(-dw)=4d^2(u^2+vw)\not=0,
$$
it is non-degenerate. 
Hence we can argue as in the previous case, and conclude that, under
the assumption \refs{bounded4}, the total number of resonating
matrices in $R_D$ is $\ll p^{\eps}$ for any $\eps>0$. Summing over the
dyadic ranges, the total number of resonating matrices
$\gamma(c,n_1,n_2)$ for $c\leq C,\ |n_i|\leq N_i,\ i=1,2$ associated
to $\rmN^{x,y}-\rmT^{x,y}$ is also $\ll p^\eps$. 
\par
We deduce then as before the bounds
\begin{align*}
  \frac1p \sum_{c\leq C}c^{-1}\mathcal{E}^{w}_{\phi_k}(c,d,e)&\ll
  M^2p^{1+\eps} \max_{\substack{c\leq C\\1\leq |n_i|\leq N_i,\ i=1,2}}
  \frac{1}{cp^2}\Bigl|\widehat
  H_{\phi_k}\Bigl(\frac{-dn_2}{cpN},\frac{n_2}{cpN}\Bigr)\Bigr|\\
  &\ll M^2p^{1+\eps}k^3 \max_{c\leq C}\frac{P^2}{d}\frac{Q}{cZ}\ll
  M^2k^3p^{1+\eps}\frac{PQ}{L},
\end{align*}
for one normalizer (and similarly with $\phi_{a,b}$ without the $k^3$
factor), and therefore
\begin{equation}
\label{tonbound}
M_3^{w}[\phi_{a,b}]\ll M^3p^{1+\eps}LPQ,\quad\quad
M_3^{w}[\phi_{k}]\ll M^3k^3p^{1+\eps}LPQ,
\end{equation}
for any $\eps>0$, where the implied constants depend on
$(a,b,N,\eps)$. 

\subsection{Conclusion}

We can now gather Lemma~\ref{lm-md}, Lemma~\ref{lm-complete} and
Proposition~\ref{pr-m2c} (choosing $a-b$ and $k$ large enough
depending on $\eps$ so that \refs{est43} holds), together
with~(\ref{eq-m3l-bound}), \refs{tribound}, \refs{parbound},
\refs{torbound} and \refs{tonbound}. We derive, under the assumptions
that \refs{p-epsilon-P} and \refs{bounded4} hold, the bound
\begin{align*}
  M(L),\ k^{-3}M(L;k)&\ll M^3\{ pLP+p^{1+\eps}LP(P+1) +p^{1+\eps}LPQ
  +p^{1/2+\eps}L^3PQ(P+Q)^2\}\\
  &\ll M^3\{p^{1+\eps}LP(P+Q) +p^{1/2+\eps}L^3PQ^2(P+Q)\}
\end{align*}
for any $\eps>0$, where the implied constant depends on $f$ and
$\eps$.
\par
Finally, we observe that if \refs{p-epsilon-P} does not hold, the
above bound remains valid by Lemma~\ref{lm-md} and
\refs{largesievebound}, and this concludes the proof of
Proposition~\ref{pr-amplified}.

\section{Distribution of twisted Hecke orbits and
  horocycles}\label{sec-distrib-twisted}

We prove in this section, the results of
Section~\ref{ssec-orbits}, using the main estimate of
Theorem~\ref{th-from-admissible-to-orthogonality} as basic tool.

\begin{proof}[Proof of Theorem~\ref{weightedshorthorocycles}]
  Let $K=K_p$ be an isotypic trace function with conductor at most $M$
  and $I=I_p\subset[1,p]$ an interval. We have to show that if
  $|I|\geq p^{7/8+\kappa}$ for some fixed $\kappa>0$, we have the limit
$$
\mutw{p}{\tau}{K,I}(\varphi) = \frac{1}{|I|}\sum_{t\in
  I}K(t)\varphi\Bigl(\frac{\tau+t}p\Bigr)\lra 0
$$
as $p\ra +\infty$, for all $\varphi$ continuous and compactly
supported on $Y_0(N)$ and all $\tau\in Y_0(N)$.  By the spectral
decomposition theorem for $Y_0(N)$, it is sufficient to prove the
result for $\varphi$ either constant function $1$, or a Maass
Hecke-eigenforms or $\varphi$ a packet of Eisenstein series.
\par
Let $\varphi=f$ be a Maass cusp form with Fourier expansion
$$
f(z)=\sum_{n \in\Zz-\{0\}} \rho_f(n)|n|^{-1/2} W_{it_f}(4 \pi
\abs{n}y)e(nx).
$$ 
\par
We can assume, by linearity, that $f$ is an eigenfunction of the
involution $z\mapsto -\bar{z}$, so that there exists $\eps_f=\pm 1$
with
\begin{equation}\label{rhosym}
  \rho_f(n)=\eps_f \rho_f(-n)
\end{equation}
for all $n\in\Zz$. We now derive the basic identity relating Hecke
orbits with the twisted sums of Fourier coefficients: we have (for
$p\geq 3$)
$$
\mutw{p}{\tau}{K,I}(f) =\frac{1}{|I|}\sum_{n}\rho_f(n)|{n}/p|^{-1/2}
W_{it_f}\Bigl(\frac{4\pi \Im(\tau)|n|}p\Bigr)
e\Bigl(\frac{n\Re(\tau)}{p}\Bigr) K'_{I}(n)
$$
with
\begin{align*}
  K'_{I}(n)&=\frac{1}{p^{1/2}}\sum_{t\in I}K(t)e\Bigl(\frac{nt}p\Bigr)
  = \frac{1}{p}\sum_{x\in [-p/2,p/2]}\hat K(n-x)\sum_{t\in
    I}e\Bigl(\frac{tx}p\Bigr) \\&= \frac{|I|}p\hat
  K(n)+\frac{1}{p}\sum_{\stacksum{|x|\leq p/2}{x\not=0}} \hat
  K(n-x)\sum_{t\in I}e\Bigl(\frac{tx}p\Bigr),
\end{align*}
where $\hat{K}$ is the unitarily-normalized Fourier transform modulo
$p$, as before.  Hence, by~(\ref{rhosym}), we get
\begin{multline}\label{eq-relation}
  \mutw{p}{\tau}{K,I}(f) =\frac{1}{p}\bigl\{\tsum_{V}(f,\hat
  K;p)+\eps_f
  \tsum_{W}(f,[\times (-1)]^*\hat K;p)\bigr\}\\
  +\frac{1}{|I|}\frac{1}{p}\sum_{\stacksum{|x|\leq
      p/2}{x\not=0}}\bigl\{\tsum_{V}(f,[-x]^*\hat K;p) +\eps_f
  \tsum_{W}(f,[-x]^*[\times (-1)]^*\hat K;p)\bigr\} 
  \sum_{t\in I}e\Bigl(\frac{tx}p\Bigr)
\end{multline}
where, for any function $L\,:\, \Fp\lra \Cc$, we denote
$$
[-x]^*L(n)=L(n-x)=L\Bigl(\begin{pmatrix}1&-x\\0&1
\end{pmatrix}n\Bigr),\quad [\times (-1)]^*L(n)=L(-n),
$$
and $V$ and $W$ are the functions (depending on $t_f$ and on $\tau$)
defined on $]0,+\infty[$ by
$$
V(x)=x^{-1/2}W_{it_f}(4\pi \Im(\tau)x)e(x\Re(\tau)),\quad\quad
W(x)=x^{-1/2}W_{it_f}(4\pi \Im(\tau)x)e(-x\Re(\tau)).
$$
\par
Let $L\,:\, \Fp\lra \Cc$ be one of the functions $[\times
(-1)]^*\hat{K}$ or $[-x]^*\hat K$ or $[-x]^*[\times(-1)]^*\hat{K}$ for
some $x\in\Fp$. By Lemma \ref{lm-weights}, Proposition
\ref{pr-bound-h1} and Proposition \ref{pr-shifted-2}, each such $L$ is
an isotypic trace functions whose conductor is bounded solely in terms
of $\cond(K)$.  Therefore we would like to apply Theorem
\ref{th-from-admissible-to-orthogonality}.

\begin{remark}
  For the rest of this section we will not necessarily display the
  dependency in $M$ or $f$ or $\tau$ of the various constants implicit
  in the Vinogradov symbols $\ll$.
\end{remark}
\par
The functions $V$ and $W$ above do not a priori satisfy a condition of
type $(V(C,P,Q))$, but it is standard to reduce to this
situation. First, we truncate the large values of $n$, observing that
since
$$
W_{it}(x)\ll e^{-x/2},
$$
where the implied constant depends on $t$ (see~\refs{Wit}), the
contribution of the terms with $n\geq p^{1+\eps}$ to any of the sums
appearing in~(\ref{eq-relation}) is
$$
\ll\exp(-p^{\eps/2}),
$$ 
for any $\eps>0$. 
\par
Then, by means of a smooth dyadic partition of the remaining interval,
the various sums $\tsum_{V}(f,L;p)$ and $\tsum_{W}(f,L;p)$ occuring in
\refs{eq-relation}, are decomposed into a sum of $O(\log p)$ sums of
the shape
$$
P^{-1/2}\tsum_{\tilde{V}}(f,L;p)
$$
where $L$ has conductor bounded in terms of $M$ only, for functions
$\tilde{V}$, depending on $\tau$ and $t_f$, which satisfy Condition
$(V(C,P,Q))$ for some sequence $C=(C_{\nu})$, and
$$
P\in\Bigl[\frac{1}2p^{-1},p^\eps\Bigr],\quad\quad Q\ll_{t_f,\eps} 1.
$$
(the normalizing factor $P^{-{1/2}}$ comes from the factorization
$(x/p)^{-1/2}=P^{-{1/2}}(x/pP)^{-1/2}$, and is introduced to ensure
that $\tilde V(x)\ll_{t_f,\eps}1$).
\par
The trivial bound for these sums is $O(P^{-1/2}Pp^{1+\eps})$ and using
\begin{equation}\label{polyavino}
  \frac{1}p\sum_{\stacksum{|x|\leq p/2}{x\not=0}}\Bigl|\sum_{t\in
    I}e\Bigl(\frac{tx}p\Bigr)\Bigr| \ll \log p,
\end{equation}
we see that the contribution to $\mutw{p}{\tau}{K,I}(f)$ of the sums
with $P\leq p^{-1/2}$ is
$$
\ll p^{3/4+\eps}\Bigl(\frac{1}p+\frac{1}{|I|}\Bigr)=o(1)
$$
provided $|I|\geq p^{3/4+2\eps}$.
\par
For the remaining sums, we use
Theorems~\ref{th-from-admissible-to-orthogonality}
and~\ref{th-interpret-admissible}: we have
$$
P^{-1/2}\tsum_{\tilde{V}}(f,L;p) \ll p^{1-\delta+\eps}
$$
for any $\delta<1/8$, where the implicit constants depend on
$(M,C,f,\tau,\delta,\eps)$. we obtain that
\begin{equation}\label{cuspbound}
  \mutw{p}{\tau}{K,I}(f)\ll
  p^{-\delta+\eps}+\frac{1}{|I|}p^{1-\delta+\eps}.
\end{equation}
\par
As long as $|I|\geq p^{7/8+\kappa}$ for some fixed $\kappa>0$, we can
take $\eps>0$ small enough and $\delta>0$ small enough so that we
above shows that $\mutw{p}{\tau}{K,I}(f)\ra 0$ as $p\ra +\infty$, as
desired.
\par
The case where $\varphi$ is a packet of Eisenstein series
$E_{\chi,g}(\varphi)$ is similar, using Proposition \ref{packet
  bound}. Indeed, the contribution of the non-zero Fourier
coefficients are handled in this manner, and the only notable
difference is that we must handle the constant term of this
packet. This is given by
\begin{equation}\label{eq-constant-term}
\rho_{\chi,g}(\varphi,0)(z)=\int_\Rr
\varphi(t)\{c_{1,g}(t)y^{1/2+it}+c_{2,g}(t)y^{1/2-it}\}dt,
\end{equation}
and contributes to $\mutw{p}{\tau}{K,I}(E_{\chi,g}(\varphi))$ by
$$
\frac{1}{|I|}\sum_{t\in
  I}K(t)\rho_{\chi,g}(\varphi,0)\Bigl(\frac{\tau+t}{p}\Bigr)
=\rho_{\chi,g}(\varphi,0)\Bigl(\frac{\tau}{p}\Bigr)
\frac{1}{|I|}\sum_{t\in
  I}K(t)=\rho_{\chi,g}(\varphi,0)\Bigl(\frac{\tau}{p}\Bigr)
\frac{p^{1/2}}{|I|}K'_I(0)
$$
since $\rho_{\chi,g}(\varphi,0)(z)$ does not depend on the real part
of $z$. We have
$$
\rho_{\chi,g}(\varphi,0)\Bigl(\frac{\tau}{p}\Bigr)\ll p^{-1/2}
$$
(since $\Im \tau/p\ll1/p$) and by \eqref{polyavino}, and the fact that
$\hat{K}$ is bounded by a constant depending only on $M$ (a
consequence of Proposition~\ref{pr-bound-h1} (1)), we have
$$
K'_I(0)\ll {\log p}
$$
and therefore the contribution of the constant terms of Eisenstein
series is bounded by
$$
\ll\frac{\log p}{|I|}=o(1).
$$
For $\varphi=1$ the exact same argument yields
$$
\mu_{K,I,\tau}(1)=\frac{p^{1/2}}{|I|}K'_I(0)\ll\frac{p^{1/2}\log
  p}{|I|}
$$ 
which is $o(1)$ as long as $|I|\geq p^{\eta}$ with $\eta>1/2$. This
concludes the proof of Theorem \ref{weightedshorthorocycles}.
\end{proof}

\begin{proof}[Proof of Corollary~\ref{cor-dist-supermorse}]
  We now consider a non-constant polynomial $\phi$ of degree
  $\deg\phi\geq 1$. The probability
  measure~(\ref{eq-short-supermorse}) satisfies
$$
\frac{1}{|I|}\sum_{\stacksum{x\in \Ff_p}{\phi(x)\in
    I}}\delta_{\Gamma_0(N)\phi(x)\cdot \tau}= \mu+
\mutw{p}{\tau}{K,I}
$$ 
where
$$
K(t)=|\{x\in\Ff_p\,\mid\, \phi(x)=t\}|-1
$$
for $t\in\Ff_p$. By \S \ref{ex-fiber-count}, $K$ is a Fourier trace
function (not necessarily isotypic), whose Fourier transform is
therefore also a Fourier trace function, given by
\begin{align*}
  \hat K(n)&
  =\frac{1}{p^{1/2}} \sum_{x\in\Ff_p}{e\Bigl(\frac{n\phi(x)}{p}\Bigr)},\ (n,p)=1\\
  \hat K(0)&=0.
\end{align*}
\par
By Proposition~\ref{pr-shifted}, we can express $\hat K$ as a sum of
at most $\deg(\phi)$ functions $\hat K_i$ which are irreducible trace
functions with conductors bounded by $M$. The contribution from the
terms $\hat K_i$ is then treated by the previous proof.
\end{proof}

\section{Trace functions}\label{sec-trace}

We now come to the setting of Section~\ref{intro-l-adic}. For an
isotypic trace function $K(n)$, we will see that the cohomological
theory of algebraic exponential sums and the Riemann Hypothesis over
finite fields provide interpretations of the sums $\wwd(K;\gamma)$,
from which it can be shown that trace functions are good.
\par
In this section, we present some preliminary results. In the next one,
we give many different examples of trace functions (isotypic or not),
and compute upper bounds for the conductor of the associated
sheaves. We then use the cohomological theory to prove
Theorem~\ref{th-interpret-admissible}.
\par
First we recall the following notation for trace functions: for a
finite field $k$, an algebraic variety $X/k$, a constructible
$\ell$-adic sheaf $\sheaf{F}$ on $X$, a finite extension $k'/k$, and a
point $x\in X(k')$, we define
$$
\frtr{\sheaf{F}}{k'}{x}=\Tr(\frob_{k'}\mid \sheaf{F}_{\bar{x}}),
$$
the trace of the geometric Frobenius automorphism of $k'$ acting on
the stalk of $\sheaf{F}$ at a geometric point $\bar{x}$ over $x$ (seen
as a finite-dimensional representation of the Galois group of $k'$;
see~\cite[7.3.7]{katz-esde}).
\par 
Now let $p$ be a prime number, and let $\ell\not=p$ be another
auxiliary prime. Let
$$
\iota\,:\, \bar{\Qq}_{\ell}\lra \Cc
$$
be a fixed isomorphism, and let $\sheaf{F}$ be an $\ell$-adic
constructible Fourier sheaf on $\Aa^1_{\Ff_p}$ (in the sense of
Katz~\cite[Def.  8.2.1.2]{katz-esde}).  Recall that we consider the
functions
$$
  K(x)=\iota(\frtr{\sheaf{F}}{\Ff_p}{x})
$$
for $x\in \Ff_p=\Aa^1(\Ff_p)$. We also consider the (Tate-twisted)
Fourier transform $\sheaf{G}=\ft_{\psi}(\sheaf{F})(1/2)$ with respect
to an additive $\ell$-adic character $\psi$ of $\Ff_p$. It satisfies
\begin{equation}\label{eq-fourier-transform}
  \frtr{\sheaf{G}}{k}{v}=
  -\frac{1}{|k|^{1/2}}\sum_{x\in k}{\frtr{\sheaf{F}}{k}{x}\psi(\Tr_{k/\Ff_p}(vx))}
\end{equation}
for any finite extension $k/\Ff_p$ and $v\in k=\Aa^1(k)$
(see~\cite[Th. 7.3.8, (4)]{katz-esde}).
\par
We collect here the basic properties of Fourier sheaves and of the
Fourier transform, consequences of works of Deligne, Laumon, Brylinski
and Katz (see~\cite[\S 7.3.5]{katz-esde}, \cite[Th. 8.2.5
(3)]{katz-gkm} and~\cite[Th. 8.4.1]{katz-gkm}).

\begin{lemma}[Fourier sheaves]\label{lm-weights}
Let $p$ and $\ell\not=p$ be primes, and let $\sheaf{F}$ be an
$\ell$-adic Fourier sheaf on $\Aa^1_{\Ff_p}$.
\par
\emph{(1)} The sheaf $\sheaf{F}$ is a middle-extension sheaf: if
$j\,:\, U\injecte \Aa^1$ is the open immersion of a non-empty open set
on which $\sheaf{F}$ is lisse, we have
$$
\sheaf{F}\simeq j_*(j^*\sheaf{F}).
$$
\par
\emph{(2)} Suppose that $\sheaf{F}$ is pointwise
$\iota$-pure\footnote{\ On the maximal open set on which it is lisse.} of
weight $0$, i.e., that it is a trace sheaf. Then
\begin{itemize}
\item[-] $\sheaf{G}=\ft_{\psi}(\sheaf{F})(1/2)$ is pointwise $\iota$-pure\footnote{idem} of weight $0$;
\item[-] at the points $v\in\Aa^1$
where $\sheaf{G}$ is not lisse, it is pointwise
mixed of weights $\leq 0$, i.e., for any finite field $k$ with $v\in
k$, the eigenvalues of the Frobenius of $k$ acting on the stalk of
$\sheaf{G}$ at a geometric point $\bar{v}$ over $v$ are $|k|$-Weil
numbers of weight at most $0$.
\end{itemize}
\par
\emph{(3)} If $\sheaf{F}$ is geometrically isotypic
(resp. geometrically irreducible) then the Fourier transform
$\sheaf{G}$ is also geometrically isotypic (resp. geometrically
irreducible).
\end{lemma}

We defined the conductor of a sheaf in
Definition~\ref{def-conductor}. An important fact is that this
invariant also controls the conductor of the Fourier transform, and
that it controls the dimension of cohomology groups which enter into
the Grothendieck-Lefschetz trace formula. We state suitable versions
of these results:

\begin{proposition}\label{pr-bound-h1}
  Let $p$ be a prime number and $\ell\not=p$ an auxiliary prime.
\par
\emph{(1)} Let $\sheaf{F}$ be an $\ell$-adic Fourier sheaf on
$\Aa^1_{\Fp}$, and let $\sheaf{G}=\ft_{\psi}(\sheaf{F})(1/2)$ be its
Fourier transform.  Then, for any $\gamma\in\GL_2(\Fp)$, the
analytic conductor of $\gamma^*\sheaf{G}$ satisfies
\begin{equation}\label{eq-bound-conds}
  \cond(\gamma^*\sheaf{G})\leq 10\cond(\sheaf{F})^2.
\end{equation}
\par
\emph{(2)} For $\sheaf{F}_1$ and $\sheaf{F}_2$ lisse $\ell$-adic
sheaves on an open subset $U\subset \Aa^1$, we have 
$$
\dim H^1_c(U\times\bar{\Ff}_p,\sheaf{F}_1\otimes\sheaf{F}_2) \leq
r_1r_2(1+m+\cond(\mcF_1)+\cond(\mcF_2)),
$$
where
$$
m=|(\Pp^1-U)(\bar{\Ff}_p)|,
\quad\quad
r_i=\rank(\sheaf{F}_i).
$$
\par
\emph{(3)} Let $\sheaf{F}_1$ and $\sheaf{F}_2$ be middle-extension
$\ell$-adic sheaves on $\Aa^1_{\Fp}$. Then
\begin{equation}\label{eq-cond-tensor}
\cond(\sheaf{F}_1\otimes\sheaf{F}_2)\leq
5\cond(\sheaf{F}_1)^2\cond(\sheaf{F}_2)^2. 
\end{equation}
\end{proposition}

Note that~(\ref{eq-bound-conds}) and~(\ref{eq-cond-tensor}) can
certainly be improved, but these bounds will be enough for us.

\begin{proof}
  (1) Since $\gamma$ is an automorphism of $\Pp^1$, we have
  $\cond(\gamma^*\sheaf{G})=\cond(\sheaf{G})$ and we can assume
  $\gamma=1$.
\par
We first bound the number of singularities
$$
n(\sheaf{G})=|\Pp^1-U|
$$
of $\sheaf{G}$. By~\cite[Cor. 8.5.8]{katz-gkm} (and the remark in its
proof), on $\Gg_m$, the Fourier transform is lisse except at points
corresponding to Jordan-Hölder components of the local representation
$\sheaf{F}(\infty)$ at $\infty$ which have unique break equal to
$1$. The number of these singularities outside of $0$, $\infty$ is
therefore bounded by the rank of $\sheaf{F}$, hence by the conductor
of $\sheaf{F}$, and
\begin{equation}\label{eq-sing-g}
n(\sheaf{G})\leq 2+\rank(\sheaf{F})\leq 3\cond(\mcF).
\end{equation}
\par
Now we bound the rank of $\sheaf{G}$. This is given by~\cite[Lemma
7.3.9 (2)]{katz-esde}, from which we get immediately
$$
\rank(\sheaf{G})\leq
\sum_{\lambda}\max(0,\lambda-1)
+\sum_{x}{(\swan_x(\sheaf{F})+\rank(\sheaf{F}))}
$$
where $\lambda$ runs over the breaks of $\sheaf{F}(\infty)$, and $x$
over the singularities of $\sheaf{F}$ in $\Aa^1$.  The first term is
$\leq \swan_{\infty}(\sheaf{F})$, so that the rank of $\sheaf{G}$
is bounded by
\begin{equation}\label{eq-rank-g}
  \rank(\sheaf{G})\leq \swan(\sheaf{F})+
  \rank(\sheaf{F})n(\sheaf{F})\leq \cond(\mcF)^2. 
\end{equation}
\par
Thus it only remains to estimate the Swan conductors
$\swan_x(\sheaf{G})$ at each singularity. We do this using the local
description of the Fourier transform, due to Laumon~\cite{laumon},
separately for $0$, $\infty$ and points in $\Gg_m$.
\par
\underline{First case.} 
Let $x=\infty$. By~\cite[Cor. 7.4.2]{katz-esde} we can write
$$
\sheaf{G}(\infty)=N_0\oplus N_{\infty}\oplus N_{m}
$$
as representations of the inertia group at $\infty$, where $N_0$,
$N_{\infty}$ are the local Fourier transform functors denoted
$$
\ft_{\psi}\mathrm{loc}(\infty,\infty)\sheaf{F}(\infty),\quad
\ft_{\psi}\mathrm{loc}(0,\infty)(\sheaf{F}(0)/\sheaf{F}_0)
$$
in loc. cit., and $N_m$ is the sum of the similar contributions of the
local Fourier transforms at all $s\in \Gg_m$. Let $s_0$, $s_{\infty}$
and $s_m$ denote the corresponding Swan conductors, which add up to
$\swan_{\infty}(\sheaf{G})$. By~\cite[Cor. 7.4.1.1]{katz-esde}, all
breaks of $N_0$ and $N_m$ are $\leq 1$, hence by \refs{eq-rank-g}
$$
s_0+s_m\leq \dim(N_0)+\dim (N_m)\leq \rank(\sheaf{G}) \leq
\swan(\sheaf{F})+\rank(\sheaf{F})n(\sheaf{F}).
$$
\par
As for $s_{\infty}$, by a further result of Laumon~\cite[Th. 7.5.4
(1)]{katz-esde}, the contribution $s_{\infty}$ is equal to the similar
contribution of breaks $>1$ to the Swan conductor
$\swan_{\infty}(\sheaf{F})$. Hence by \refs{eq-rank-g}
\begin{equation}\label{eq-swinf-g}
  \swan_{\infty}(\sheaf{G})\leq 2\swan(\sheaf{F})+
  \rank(\sheaf{F})n(\sheaf{F})\leq 2\cond(\mcF)^2.
\end{equation}
\par
\underline{Second case.} Let $x=0$. Then, by~\cite[Th. 7.5.4
(5)]{katz-esde}, the Swan conductor $\swan_0(\sheaf{G})$ is equal to
the contribution to $\swan_{\infty}(\sheaf{F})$ of the breaks in
$]0,1[$, so that
\begin{equation}\label{eq-swz-g}
\swan_0(\sheaf{G})\leq \swan_{\infty}(\sheaf{F})\leq\cond(\mcF).
\end{equation}
\par
\underline{Third case.} Let $x\in\Gg_m$. By translation, we have
$$
\swan_x(\sheaf{G})=
\swan_0(\ft_{\psi}(\sheaf{F}\otimes\sheaf{L}_{\psi(xX)})),
$$
so that the previous case gives
$$
\swan_x(\sheaf{G})\leq
\swan_\infty(\sheaf{F}\otimes\sheaf{L}_{\psi(xX)})\leq
\rank(\mcF)+\swan_\infty(\mcF)\leq\cond(\mcF).
$$
\par
\par
By~(\ref{eq-sing-g}) and~(\ref{eq-rank-g}), this leads to
$$
\sum_x\swan_x(\sheaf{G})\leq 2\cond(\mcF)^2+3\cond(\mcF)^2=5\cond(\mcF)^2,
$$
and
$$
\cond(\sheaf{G})\leq 10 \cond(\sheaf{F})^2.
$$
\par
(2) We use the Euler-Poincar\'e formula: for a lisse $\ell$-adic sheaf
$\sheaf{M}$ on an affine curve $U\subset \Pp^1$ over $\Fp$, we have
\begin{equation}\label{eq-euler-poincare}
  \dim H^1_c(U\times\bar{\Ff}_p,\sheaf{M})
   =  \dim H^2_c(U\times\bar{\Ff}_p,\sheaf{M})+
  \rank(\sheaf{M})(-\chi_c(U\times\bar{\Ff}_p))
  +{\swan(\sheaf{M})}
\end{equation}
(see~\cite[2.3.1]{katz-gkm}).
\par
We apply this formula to $\sheaf{M}=\sheaf{F}_1\otimes\sheaf{F}_2$.
Since $H^2_c(U\times\bar{\Ff}_p,\sheaf{M})$ is the space of
co-invariants of $\sheaf{M}$
$$
\dim H^2_c(U\times\bar{\Ff}_p,\sheaf{M})\leq r_1r_2.
$$
\par
For the second term, we note simply that
$$
\rank(\sheaf{M})(-\chi_c(U\times\bar{\Ff}_p))\leq mr_1r_2.
$$
\par
For the last term, we bound the Swan conductor at $x\in \Pp^1-U$ of
$\sheaf{F}_1\otimes\sheaf{F}_2$ in terms of those of the factors. The
existence of such a bound is a well-known result: if $\lambda_1$
(resp. $\lambda_2$) is the largest break of $\sheaf{F}_1$
(resp. $\sheaf{F}_2$) at $x$, then all breaks of
$\sheaf{F}_1\otimes\sheaf{F}_2$ at $x$ are at most
$$
\max(\lambda_1,\lambda_2)\leq \max(\swan_x(\sheaf{F}_1),\swan_x(\sheaf{F}_2)),
$$
(see~\cite[Lemma 1.3]{katz-gkm}) and hence
$$
  \swan_x(\sheaf{F}_1\otimes\sheaf{F}_2)\leq
  \rank(\sheaf{F}_1)\rank(\sheaf{F}_2)
  (\swan_x(\sheaf{F}_1)+\swan_x(\sheaf{F}_2))
$$
and
\begin{equation}\label{eq-swan-sing}
  \swan(\mcF_1\otimes\mcF_2)\leq 
  r_1r_2(\swan(\sheaf{F}_1)+\swan(\sheaf{F}_2)).
\end{equation}
\par
Adding this to the previous contribution, we get
$$
\dim H^1_c(U\times\bar{\Ff}_p,\sheaf{M})
\leq r_1r_2(1+m+\cond(\mcF_1)+\cond(\mcF_2)),
$$
as claimed.
\par
(3) Let $c_i=\cond(\sheaf{F}_i)$, $r_i=\rank(\sheaf{F}_i)$ and $n_ia$
the number of singularities of $\sheaf{F}_i$. The rank of
$\sheaf{F}_1\otimes\sheaf{F}_2$ is $r_1r_2$, and it has $\leq n_1+n_2$
singularities. By~(\ref{eq-swan-sing}), we have also
$$
\swan(\sheaf{F}_1\otimes\sheaf{F}_2)\leq
r_1r_2(\swan(\sheaf{F}_1)+\swan(\sheaf{F}_2)).
$$
The result follows by the roughest estimate:
$$
\cond(\sheaf{F}_1\otimes\sheaf{F}_2)\leq 
c_1c_2+c_1+c_2+c_1c_2(c_1+c_2)\leq 5c_1^2c_2^2.
$$
\end{proof}

We can also explain here how to deal with Fourier trace functions which
are not necessarily isotypic.

\begin{proposition}\label{pr-shifted}
  Let $p$ be a prime number, $\ell\not=p$ an auxiliary prime. Let
  $\sheaf{F}$ be a Fourier trace sheaf modulo $p$ with conductor $\leq
  M$.
\par
There exist at most $\rank(\sheaf{F})$ isotypic trace sheaves
$\sheaf{F}_i$ modulo $p$, each with conductor $\leq M$, such that
$$
\frtr{\sheaf{F}}{\Fp}{x}= \sum_{i}{\frtr{\sheaf{F}_i}{\Fp}{x}}
$$
for all $x\in \Fp$.  In particular, for any $s\geq 1$, the trace
function 
$$
K(n)=\iota(\frtr{\sheaf{F}}{\Fp}{n})
$$
satisfies $\tnorm{K}{s}\leq M^{s+1}$.
\end{proposition}

\begin{proof}
We refer to~\cite[\S 4.4--4.6]{katz-sommes} for basic facts concerning
the correspondance between middle-extension sheaves on $\Aa^1_{\Fp}$
and representations of the \'etale fundamental group.
\par
Let $j:U\injecte \Aa^1$ be an open dense subset, defined over $\Fp$,
such that $\sheaf{F}$ is lisse on $U$, and let $G=\pi_1(U,\bar{\eta})$
and
$$
\rho\,:\, G\lra \GL(V)
$$
the $\ell$-adic representation corresponding to the restriction of
$\sheaf{F}$ on $U$. Let
$$
\rho^{ss}=\bigoplus_{i\in I}{\rho_i}
$$
be the semisimplification of this representation, where $\rho_i$ is an
irreducible representation of $G$. We denote by $\tilde{\sheaf{F}}_i$
the corresponding lisse sheaf on $U$, and let
$\sheaf{F}_i=j_*\tilde{\sheaf{F}}_i$. Then each $\sheaf{F}_i$ is a
Fourier sheaf modulo $p$, with conductor $\leq M$, and we have
\begin{equation}\label{eq-sum-i}
\frtr{\sheaf{F}}{k}{x}=\sum_{i\in I}{\frtr{\sheaf{F}_i}{k}{x}}
\end{equation}
for any finite extension $k/\Fp$ and $x\in k$. Indeed, this holds by
definition for $x\in U(k)$, and this extends to all $x$ by properties
of middle-extension sheaves (see
Proposition~\ref{pr-geo-isom-criterion}).
\par
Each $\rho_i$ is arithmetically irreducible, and there are two
possibilities concerning its restriction $\rho_i^g$ to
$G^g=\pi_1(U\times\bar{\Ff}_p,\bar{\eta})$: (1) either $\rho_i^g$ is
isotypic, and hence $\sheaf{F}_i$ is an isotypic trace sheaf; or (2)
there exists an integer $m\geq 2$, and a representation $\tau_i$ of
the proper normal subgroup $H=\pi_1(U\times\Ff_{p^m},\bar{\eta})$ of
$G$ such that
$$
\rho_i=\Ind_H^{G}\tau_i
$$ 
(see, e.g.,~\cite[Prop. 8.1]{serre-lin}
or~\cite[Prop. 2.8.20]{kow-rt}).  We claim that in this second case,
the trace function of $\sheaf{F}_i$ is identically zero on $\Fp$,
which finishes the proof since we can then drop $\sheaf{F}_i$ from the
decomposition~(\ref{eq-sum-i}).
\par
To check the claim, note that the formula for the character of an
induced representation shows that
$$
\Tr\rho_i(g)=0
$$ 
for any $g\notin H$ (see, e.g.~\cite[Prop. 2.7.43]{kow-rt}). Hence the
trace function vanishes obviously on $U(\Fp)$ since the Frobenius
elements associated to $x\in U(\Fp)$ relative to $\Fp$ are not in $H$.
\par
This property extends to $x\in (\Aa^1-U)(\Fp)$ by a similar argument
(we thank N. Katz for explaining this last point; note that we could
also treat separately the points in $\Aa^1-U$, which would lead at
most to slightly worse bounds for the trace norm of $K$).
\par
Let $\tilde{G}=\pi_1(\Aa^1,\bar{\eta})$ be the fundamental group of
the affine line. There is a surjective homomorphism
$$
\tilde{G}\lra G.
$$
\par
The group $\tilde{G}$ contains as normal subgroups
$$
\tilde{G}^g=\pi_1(\Aa^1\times\bar{\Ff}_p,\bar{\eta}),\quad\quad
\tilde{H}=\pi_1(\Aa^1\times\Ff_{p^m},\bar{\eta}),
$$
with corresponding surjective morphisms $\tilde{G}^g\lra G^g$ and
$\tilde{H}\lra H$. 
\par
Composing these with $\tau_i$ and $\rho_i$ gives representations
$\tilde{\tau}_i$ and $\tilde{\rho}_i$ of $\tilde{H}$ and $\tilde{G}$,
respectively, with
$\tilde{\rho}_i=\Ind_{\tilde{H}}^{\tilde{G}}\tilde{\tau}_i$. 
\par
The stalk of $\sheaf{F}_i$ at a geometric point above $x\in
(\Aa^1-U)(\Fp)$ is isomorphic, as a vector space with the action of
the Galois group of $\Fp$, to the invariant space $\rho_i^{I_x}$ under
the inertia subgroup at $x$, which is a subgroup $I_x$ of $\tilde{G}$.
\par
The space of $\tilde{\rho}_i$ can be written as a direct sum
$$
\bigoplus_{\sigma\in \tilde{G}/\tilde{H}} W_{\sigma}
$$
where the spaces $W_{\sigma}$ are $\tilde{H}$-stable and permuted by
$\tilde{G}$. 
Moreover, any $g\in G-H$ permutes the $W_{\sigma}$ without fixed
points, because $H$ is normal in $G$.
\par
The point is that since $I_x\subset \tilde{G}^g\subset \tilde{H}$ (the
inertia group is a subgroup of the geometric Galois group) and each
$W_{\sigma}$ is $\tilde{H}$-stable, we have
$$
\tilde{\rho}_i^{I_x}=\bigoplus_{\sigma\in G/H}W_{\sigma}^{I_x}.
$$
(in other words, this shows that $\tilde{\rho}_i^{I_x}\simeq
\Ind_{\tilde{H}}^{\tilde{G}}\tilde{\tau}_i^{I_x}$).
\par
The matrix representing the action on $\tilde{\rho}_i^{I_x}$ of any
element $g$ in the decomposition group $D_x$ mapping to the Frobenius
conjugacy class at $x$ in $D_x/I_x$ is block-diagonal with respect to
this decomposition. Since $g\notin \tilde{H}$, this block-diagonal
matrix has zero diagonal blocks, hence its trace, which is the value
of the trace function of $\sheaf{F}_i$ at $x$, also vanishes.
\end{proof}

The following is relevant to Theorem~\ref{weightedshorthorocycles}. 

\begin{proposition}\label{pr-shifted-2}
  Let $p$ be a prime number, $\ell\not=p$ an auxiliary prime. Let
  $\sheaf{F}$ be an $\ell$-adic Fourier trace sheaf modulo $p$ with
  conductor $\leq N$. Let $K(n)$ be the corresponding Fourier trace
  function. Then, for any $x\in\Ff_p$, $[+x]^*K(n)=K(x+n)$ defines a
  Fourier trace function associated to the sheaf
$$
\sheaf{F}^{(x)}=\begin{pmatrix}1&x\\0&1
\end{pmatrix}^*\sheaf{F},
$$
and we have $\cond(\sheaf{F}^{(x)})=\cond(\sheaf{F})\leq N$ for all
$x\in\Ff_p$.
\end{proposition}

\begin{proof}
  It is clear that $\sheaf{F}^{(x)}$ has the right trace function and
  that it is a Fourier trace sheaf, with the same conductor as
  $\sheaf{F}$.
\end{proof}

Finally, we state a well-known criterion for geometric isomorphism of
sheaves, that says that two irreducible middle-extension sheaves are
geometrically isomorphic if their trace functions are equal on
$\Aa^1(\bar{\Ff}_p)$ ``up to a constant depending on the definition
field''. Precisely:

\begin{proposition}[Geometric isomorphism
  criterion]\label{pr-geo-isom-criterion} 
  Let $k$ be a finite field, and let $\sheaf{F}_1$ and $\sheaf{F}_2$
  be geometrically irreducible $\ell$-adic sheaves, lisse on a
  non-empty open set $U/k$ and pointwise pure of weight $0$. Then
  $\sheaf{F}_1$ is geometrically isomorphic to $\sheaf{F}_2$ if and
  only if there exists $\alpha\in\bar{\Qq}_{\ell}^{\times}$ such that
  for all finite extensions $k_1/k$, we have
\begin{equation}\label{eq-isom-criterion}
  \frtr{\sheaf{F}_1}{k_1}{x}=\alpha^{[k_1:k]}\frtr{\sheaf{F}_2}{k_1}{x}
\end{equation}
for all $x\in U(k_1)$.
\par
In particular, if $\sheaf{F}_1$ and $\sheaf{F}_2$ are irreducible
Fourier sheaves, they are geometrically isomorphic if and only if
there exists $\alpha\in\bar{\Qq}_{\ell}^{\times}$ such that for all
finite extensions $k_1/k$, we have
\begin{equation}\label{eq-isom-criterion-2}
\frtr{\sheaf{F}_1}{k_1}{x}=\alpha^{[k_1:k]}\frtr{\sheaf{F}_2}{k_1}{x}
\end{equation}
for all $x\in k_1$.
\end{proposition}

\begin{proof}[Sketch of proof]
This is a well-known fact; it is basically an instance of what is
called ``Clifford theory'' in representation theory. We sketch a proof
for completeness. In the ``if'' direction, note
that~(\ref{eq-isom-criterion}) shows that $\sheaf{F}_1$ and
$\alpha^{\deg(\cdot)}\otimes \sheaf{F}_2$ are lisse sheaves on $U$
with the same traces of Frobenius at all points of $U$; the Chebotarev
Density Theorem shows that the Frobenius conjugacy classes are dense
in $\pi_1(U,\bar{\eta})$, so we conclude that $\sheaf{F}_1\simeq
\alpha^{\deg(\cdot)}\otimes\sheaf{F}_2)$ as lisse sheaves on $U$. But
then restriction to the geometric fundamental group (the kernel of the
degree) gives $\sheaf{F}_1\simeq \sheaf{F}_2$ geometrically on $U$.
\par
Conversely, if $\sheaf{F}_1$ is geometrically isomorphic to
$\sheaf{F}_2$, and $\rho_i$ is the representation of
$\pi_1(U,\bar{\eta})$ associated to $\sheaf{F}_i$, then representation
theory (see, e.g.,~\cite[2.8.2]{kow-rt}) shows that there exists a
character $\chi$ of the abelian group
$\pi_1(U,\bar{\eta})/\pi_1(U\times\bar{\Ff}_p,\bar{\eta})$ such that
$$
\rho_1\simeq \chi\otimes\rho_2.
$$
But such characters are of the type $\alpha^{\deg(\cdot)}$ since the
quotient is isomorphic to the Galois group $\Gal(\bar{\Ff}_p/\Ff_p)$. 
\par
For the second part, apply the first with the fact that
middle-extension sheaves on $\Aa^1$ are geometrically isomorphic if
and only if their restrictions to a common dense open set where they
are lisse are geometrically isomorphic.
\end{proof}

Here is a last definition.  If $\sheaf{F}$ is a Fourier sheaf on
$\Aa^1/k$, we write $\dual(\sheaf{F})$ for the middle-extension dual
of $\sheaf{F}$, i.e., given a dense open set $j\,:\, U\injecte \Aa^1$
where $\sheaf{F}$ is lisse, we have
$$
\dual(\sheaf{F})=j_*((j^*\sheaf{F})'),
$$
where the prime denotes the lisse sheaf on $U$ associated to the
contragredient of the representation of the fundamental group of $U$
which corresponds to $j^*\sheaf{F}$ (see~\cite[7.3.1]{katz-esde}). If
$\sheaf{F}$ is pointwise pure of weight $0$, it is known that
\begin{equation}\label{eq-dual-conj}
  \iota(\frtr{\dual(\sheaf{F})}{k'}{x})
  =\overline{\iota(\frtr{\sheaf{F}}{k'}{x})}
\end{equation}
for all finite extensions $k'/k$ and all $x\in k'$.

\section{Application of the Riemann Hypothesis}
\label{ssec-trace-wd}

We can now prove that correlation sums of trace functions are small,
except for matrices in the Fourier-M\"obius group. This is the crucial
argument that relies on the Riemann Hypothesis over finite fields.

\begin{theorem}[Cohomological bound for correlation sums]
  \label{th-cohomological-bound}
  Let $p$ be a prime number, $\ell\not=p$ another prime. Let
  $\sheaf{F}$ be an isotypic trace sheaf on $\Aa^1_{\Fp}$ and let $K$
  denote its trace function.  We have
\begin{equation}\label{eq-cohomological-bound}
  |\wwd(K;\gamma)|\leq
  M_1+M_2p^{1/2}
\end{equation}
if $\gamma\notin\haut_{\sheaf{F}}$ where
\begin{equation}\label{eq-bounds-m1m2}
M_1\leq 6\cond(\sheaf{F})^5,
\quad\quad
M_2\leq  24\cond(\sheaf{F})^6.
\end{equation}
\end{theorem}

The bounds~(\ref{eq-bounds-m1m2}) are certainly not sharp, but they
show that the result is completely effective and explicit.

\begin{proof} 
  We denote by $\sheaf{G}$ the Fourier transform of $\sheaf{F}$
  computed with respect to some non-trivial additive character $\psi$,
  and by $U$ the largest open subset of $\Aa^1$ where $\sheaf{G}$ is
  lisse.
\par
Let
$$
\gamma=\begin{pmatrix}a&b\\c&d
\end{pmatrix}\in \PGL_2(\Ff_p).
$$
\par
We define the constructible $\ell$-adic sheaf
$$
\sheaf{H}_{\gamma}=\gamma^*\sheaf{G}\otimes \dual(\sheaf{G})
$$
on $\Pp^1_{\Fp}$. This sheaf is lisse and pointwise $\iota$-pure of
weight $0$ on any open subset of $\Pp^1$ where it is lisse, in
particular on the non-empty open set
$$
U_{\gamma}=\gamma^{-1}U\cap U\subset \Aa^1-\{-d/c\},
$$
and for $z\in U_{\gamma}(\Fp)$, we have
$$
\iota(\frtr{\sheaf{H}_{\gamma}}{\Fp}{z})=\hat{K}(\gamma\cdot
z)\overline{\hat{K}(z)}
$$
by the definition~(\ref{eq-fourier-trace}) of the Fourier transform
and by~(\ref{eq-dual-conj}). Thus we have
\begin{equation}\label{eq-coho-sum}
  \wwd(K;\gamma)=\iota
  \Bigl(\sum_{z\in U_{\gamma}(\Fp)}
  \frtr{\sheaf{H}_{\gamma}}{\Ff_p}{z}\Bigr)+
  \sum_{\stacksum{z\in \Fp-U_{\gamma}(\Fp)}{z\not=-d/c}}{
    \hat{K}(\gamma\cdot z)\overline{\hat{K}(z)}
  }.
\end{equation}
\par
\par
According to the Grothendieck-Lefschetz trace formula (see,
e.g.,~\cite[Rapport, Th. 3.2]{deligne}), we have
\begin{multline}\label{eq-gl}
  \sum_{z\in U_{\gamma}(\Fp)}{\frtr{\sheaf{H}_{\gamma}}{k}{z}}
  =\Tr(\frob\mid H^0_c(U_{\gamma}\times \bar{\Ff}_p,\sheaf{H}_{\gamma}))\\
  -\Tr(\frob\mid H^1_c(U_{\gamma}\times \bar{\Ff}_p,\sheaf{H}_{\gamma}))
  +\Tr(\frob\mid H^2_c(U_{\gamma}\times \bar{\Ff}_p,\sheaf{H}_{\gamma}))
\end{multline}
where $\frob$ denotes the geometric Frobenius of $\Fp$ acting on the
cohomology groups of $\sheaf{H}_{\gamma}$.
\par
Since $U_{\gamma}$ is an affine curve, we have $H^0_c(U_{\gamma}\times
\bar{\Ff}_p,\sheaf{H}_{\gamma})=0$ (see,
e.g.,~\cite[(1.4.1)b]{weilii}). Next, the coinvariant formula for
$H^2_c$ on a curve (see~\cite[(1.4.1)b]{weilii}) states that
$H^2_c(U_{\gamma}\times\bar{\Ff}_p,\sheaf{H}_{\gamma})$ is isomorphic
to the space of coinvariants of
$\pi_1(U_{\gamma}\times\bar{\Ff}_p,\bar{\eta})$ acting on
$\sheaf{H}_{\gamma,\bar{\eta}}$. In particular, we have
$$
H^2_c(U_{\gamma}\times\bar{\Ff}_p,\sheaf{H}_{\gamma})=0
$$
if this coinvariant space is zero. We next show that this is the case
if $\gamma\notin\haut_{\sheaf{F}}$. 
\par
The sheaf $\sheaf{F}$ is geometrically isotypic when restricted to an
open set $V$ where it is lisse. Let $j\,:\, V\injecte \Aa^1$ be the
open immersion of $V$ in the affine line. There exists a
(geometrically) irreducible lisse sheaf $\sheaf{F}_1$ on
$V\times\bar{\Ff}_p$ such that
$$
\sheaf{F}\simeq (j_*\sheaf{F}_1)^{\oplus d}
$$
as sheaves on $\Aa^1\times\bar{\Ff}_p$ (since both sides are
middle-extension sheaves which are isomorphic on
$V\times\bar{\Ff}_p$). This formula shows that $j_*\sheaf{F}_1$ is a
Fourier sheaf on $\Aa^1\times\bar{\Ff}_p$.  Taking the Fourier
transforms, it follows that we have a geometric isomorphism
$$
\sheaf{G}\simeq \ft(j_*\sheaf{F}_1))(1/2)^{\oplus d},
$$
and hence (since the Fourier transform of a geometrically irreducible
sheaf is geometrically irreducible) that $\sheaf{G}$ is geometrically
isotypic on $U_{\gamma}$, with irreducible component
$$
\sheaf{G}_1=\ft(j_*\sheaf{F}_1))(1/2).
$$
\par
Applying $\gamma$ and taking dual, we see
that $\gamma^*\sheaf{G}$ and $\dual(\sheaf{G})$ are also lisse and
geometrically isotypic on $U_{\gamma}$. Moreover, the geometrically
irreducible components of $\gamma^*\sheaf{G}$ is
$\gamma^*\sheaf{G}_1$, and that of $\dual(\sheaf{G})$ is
$\dual(\sheaf{G}_1)$. 
\par
Finally, by Schur's Lemma, the coinvariant space of
$\pi_1(U_{\gamma}\times\bar{\Ff}_p,\bar{\eta})$ acting on
$\sheaf{H}_{\gamma,\bar{\eta}}$ is zero unless we have a geometric
isomorphism
$$
\gamma^*\sheaf{G}_1\simeq \sheaf{G}_1,
$$
which holds if and only if $\gamma^*\sheaf{G}$ is geometrically
isomorphic to $\sheaf{G}$.
\par
Thus, if $\gamma\notin\haut_{\sheaf{F}}$, the only contribution to the
expression~(\ref{eq-gl}) comes from the cohomology group
$H^1_c(U_{\gamma}\times\bar{\Ff}_p,\sheaf{H}_{\gamma})$.  But since
$\sheaf{H}_{\gamma}$ is pointwise pure of weight $0$ on $U_{\gamma}$,
it follows from Deligne's fundamental proof of the Riemann Hypothesis
over finite fields (see~\cite[Th. 3.3.1]{weilii}) that all eigenvalues
of $\frob$ acting on $H^1_c(U_{\gamma}\times
\bar{\Ff}_p,\sheaf{H}_{\gamma})$ are algebraic numbers, all conjugates
of which are of modulus at most $p^{1/2}$.
\par
Thus, using~(\ref{eq-coho-sum}), we obtain
$$
|\wwd(K;\gamma)|\leq p^{1/2}\dim
H^1_c(U_{\gamma}\times\bar{\Ff}_p,\sheaf{H}_{\gamma})+
\sum_{\stacksum{z\in \Fp-U_{\gamma}(\Fp)}{z\not=-d/c}}{
  \hat{K}(\gamma\cdot z)\overline{\hat{K}(z)}}
$$
for $\gamma\notin\haut_{\sheaf{F}}$. By Lemma~\ref{lm-weights}, at the
points $z\in \Fp-U_{\gamma}(\Fp)$, we have
$$
|\hat{K}(\gamma\cdot z)|\leq
\rank(\gamma^*\sheaf{G})=\rank(\sheaf{G}),\quad\quad |\hat{K}( z)|\leq
\rank(\sheaf{G}),
$$
since $\sheaf{G}$ and $\gamma^*\sheaf{G}$ have local weights $\leq 0$
at all points. There are at most $2n(\sheaf{G})$ points where we use
this bound, and thus
$$
\Bigl|\sum_{\stacksum{z\in \Fp-U_{\gamma}(\Fp)}{z\not=-d/c}}{
  \hat{K}(\gamma\cdot z)\overline{\hat{K}(z)}}\Bigr| \leq
2n(\sheaf{G})\rank(\sheaf{G})^2.
$$
\par
Finally we have
$$
\dim H^1_c(U_{\gamma}\times\bar{\Ff}_p,\sheaf{H}_{\gamma})
\leq \rank(\sheaf{G})^2(1+n(\sheaf{G})+2\cond(\sheaf{G}))
\leq 24\cond(\sheaf{F})^6
$$
by Proposition~\ref{pr-bound-h1}
and~(\ref{eq-sing-g}),~(\ref{eq-rank-g}), and similarly
$$
2n(\sheaf{G})\rank(\sheaf{G})^2\leq 6\cond(\sheaf{F})^5.
$$
\end{proof}


Theorem~\ref{th-cohomological-bound} justifies the
definition~\ref{def-fm} of the Fourier-M\"obius group
$\haut_{\sheaf{F}}$ of an isotypic trace sheaf. Note that this group
$\haut_{\sheaf{F}}$ depends on $\psi$, although the notation does not
reflect this ($\haut_{\sheaf{F}}$ is well-defined up to
$\Fp$-conjugacy, however).
\par
Now from the definition of the Fourier-M\"obius group and
Theorem~\ref{th-cohomological-bound}, we get our interpretation of
$\hautk{K}{M}$ for irreducible trace functions:

\begin{corollary}\label{cor-interpret}
  Let $p$ be a prime number, $\sheaf{F}$ an isotypic trace sheaf on
  $\Aa^1_{\Ff_p}$. Let $K$ be the corresponding isotypic trace
  function. Then, for
$$
M\geq 6\cond(\sheaf{F})^5+24\cond(\sheaf{F})^6,
$$ 
we have $\hautk{K}{M}\subset \haut_{\sheaf{F}}(\Ff_p)$.
\end{corollary}



Our goal is now to prove Theorem~\ref{th-interpret-admissible}: all
isotypic trace functions are $(p,M)$-good, where $M$ depends only on the
conductor of the associated sheaf.  This is done by distinguishing two
cases, depending on whether the order of the finite subgroup
$\haut_{\sheaf{F}}(\Ff_p)$ is divisible by $p$ or not.
\par
For the first case, we have the following lemma, which is an immediate
consequence of the classification of Artin-Schreier sheaves (or of
Weil's theory, when spelled-out in terms of exponential sums).

\begin{lemma}\label{lm-divisible-by-p}
  Let $p$ be a prime number, $\ell\not=p$ an auxiliary prime, $\psi$ a
  non-trivial $\ell$-adic additive character of $\Ff_p$. Let
  $\gamma_0\in \PGL_2(\Ff_p)$, and let
  $\sheaf{F}=\sheaf{L}_{\psi(\gamma_0(X))}$. Then for $\gamma\in
  \PGL_2(\bar{\Ff}_p)$, we have a geometric isomorphism
  $\gamma^*\sheaf{F}\simeq \sheaf{F}$ if and only if $\gamma$ is in
  the unipotent radical of the stabilizer of $\gamma_0^{-1}\cdot
  \infty$.
\end{lemma}


Below we denote by $\rmU^x\subset \PGL_2$ the unipotent radical of the
Borel subgroup of $\PGL_2$ fixing $x\in \Pp^1$. Recall that, for
$x\not=y$ in $\Pp^1$, we denote by $\rmT^{x,y}\subset \PGL_2$ the
maximal torus of elements fixing $x$ and $y$, and by $\rmN^{x,y}$
its normalizer.

\begin{proof}[Proof of Theorem~\ref{th-interpret-admissible}]
  By Corollary~\ref{cor-interpret}, there exists $M\leq 30N^6$ such
  that
$$
\hautk{K}{M}\subset G=\haut_{\sheaf{F}}(\Ff_p),
$$
which is a subgroup of $\PGL_2(\Ff_p)$. We distinguish two cases:
\par
--- If $p\nmid |G|$, then the classification of finite subgroups of
$\PGL_2(\bar{\Ff}_p)$ of order coprime to the characteristic (see for
instance~\cite{beauville-pgl} and the references there) show that we
have either $|G|\leq 60$, or $G$ is cyclic or dihedral. In the former
situation, the non-trivial elements of $G$ are non-parabolic and
belong to at most $59$ different tori $\rmT^{x_i,y_i}$ and the function
$K$ is $(p,\max(59,M))$-good by Definition~\ref{def-admissible}.  In
the cylic or dihedral situation, one also knows that $G$ is contained
in the normalizer $\rmN^{x,y}$ of a certain fixed maximal torus
$\rmT^{x,y}$ (indeed, if $G$ is cyclic, all its elements are
diagonalizable in a common basis, and it is a subgroup of a maximal
torus; if $G$ is dihedral of order $2r$, the cyclic subgroup of order
$r$ is contained in a maximal torus, and any element not contained in
it is in the normalizer, see
e.g.,~\cite[Prop. 4.1]{beauville-pgl}). Hence $K$ is $(p,M)$-good,
with at most one pair $(x,y)$ in~(\ref{eq-def-admi}).
\par
--- If $p\mid |G|$, we fix $\gamma_0\in G$ of order $p$ and denote by
$x\in\Pp^1(\Ff_p)$ its unique fixed point. let $\sigma\in
\PGL_2(\Ff_p)$ be such that
$$
\sigma\begin{pmatrix}1&1\\0&1
\end{pmatrix}\sigma^{-1}=\gamma_0
$$
and let $\sheaf{G}_1=\sigma^*\sheaf{G}$. We then have a geometric
isomorphism
$$
[+1]^*\mcG_1\simeq\mcG_1.
$$
\par
Suppose first that $\mcG_1$ is ramified at some
$x\in\Aa^1(\ov\Fp)$. Then, by the above, it is ramified at $x$,
$x+1$,\ldots, $x+p-1$, and therefore we obtain
$$
\cond(\mcG)=\cond(\mcG_1)\geq p+\rank(\mcG_1)=p+\rank(\mcG),
$$
and in that case $K$ is $(p,N)$-good for trivial reasons.
\par
Now assume that $\mcG_1$ is lisse on $\Aa^1(\ov\Fp)$. The
geometrically irreducible component $\mcG_2$ of $\mcG_1$ satisfies
also $[+1]^*\mcG_2\simeq \mcG_2$. Hence, by~\cite[Lemma 5.4,
(2)]{FKM3} (applied with $G=\Fp$ and $P_h=0$), either
$$
\cond(\mcG_1)\geq \swan_{\infty}(\mcG_2)\geq p+\rank(\mcG)
$$
(and we are done as above) or else $\mcG_2$ is geometrically
isomorphic to some Artin-Schreier sheaf $\sheaf{L}_{\psi}$ for some
non-trivial additive character $\psi$ of $\Fp$.
\par
In that case, we see that $\mcG_1$ is geometrically isomorphic to a
sum of copies of $\mcL_\psi$. Hence there exists $a\in\Ff^\times_p$
and algebraic numbers $\alpha_1,\cdots,\alpha_{\rank(\mcG)}$, all of
weight $0$, such that
$$
\iota(\frtr{\sheaf{G}_1}{\Ff_p}{n})=
(\alpha_1+\cdots+\alpha_{\rank(\mcG)})e\Bigl(\frac{an}{p}\Bigr)
=\iota(\frtr{\sheaf{G}_1}{\Ff_p}{0}) e\Bigl(\frac{an}{p}\Bigr)
$$
for all $n\in\Fp$.
\par
Hence we get
$$
\hat{K}(n)=e\Bigl(\frac{a\sigma^{-1}(n)}{p}\Bigr) \hat{K}(\sigma\cdot
0)
$$
for all $n\not=x$ in $\Ff_p$. By Proposition \ref{pr-geo-isom-criterion}, the trace function $K(n)$ is a multiple of the trace function of the (possibly)
different Fourier trace sheaf $\tilde{\sheaf{F}}$, whose Fourier
transform is geometrically isomorphic to the irreducible sheaf
$$
\sheaf{L}_{\psi(a\sigma^{-1}(X))}.
$$
\par
But for this sheaf, we know by Lemma~\ref{lm-divisible-by-p} that
$\haut_{\tilde{\sheaf{F}}}=\rmU^x$, and in particular all elements of
$\haut_{\tilde{\sheaf{F}}}$ are parabolic. Furthermore, the conductor
of $\tilde{\sheaf{F}}$ is absolutely bounded (the conductor of its
Fourier transform is $3$, and we apply the Fourier inversion and
Proposition~\ref{pr-bound-h1}, or we could do a direct
computation). Since we have
$$
|\hat{K}(\sigma\cdot 0)|=|\alpha_1+\cdots+\alpha_{\rank(\mcG)}| \leq
\rank(\mcG)\leq 10N^2,
$$
and
$$
\wwd(K;\gamma)=|\hat{K}(\sigma\cdot 0)|^2 \wwd(\tilde{K};\gamma)
$$
where $\tilde{K}$ is the trace function of $\tilde{\sheaf{F}}$, it
follows that $\hautk{K}{aN^4}\subset \haut_{\tilde{\sheaf{F}}}(\Ff_p)$
for some absolute constant $a\geq 1$. It follows by
Definition~\ref{def-admissible} that the function $K$ is
$(p,aN^4)$-good.
\end{proof}

\section{Examples of trace functions}

In this section, we will discuss four classes of functions $K(n)$ that
arise as trace functions. In a first reading, only the definitions of
these functions may be of interest, rather than the technical
verification that they satisfy the necessary conditions.
\par
We note that these examples are by no means an exhaustive list. One
can find more examples, in particular, in~\cite[\S 7.11]{katz-esde}.

\subsection{Additive and multiplicative
  characters}\label{ex-characters}

We recall now how the characters~(\ref{eq-weight-mixed}) of
Corollary~\ref{cor-main} fit in the framework of trace functions.  Let
$\eta$ be an $\ell$-adic-valued multiplicative character
$$
\eta\,:\, \Ff_p^{\times}\lra \bar{\Qq}_{\ell}^{\times}
$$
and let $\psi$ be an $\ell$-adic additive character
$$
\psi\,:\, \Ff_p\lra \bar{\Qq}_{\ell}^{\times}.
$$
\par
The classical constructions of Artin-Schreier and Kummer sheaves show
that, for any $\ell\not=p$, one can construct $\ell$-adic sheaves
$\sheaf{L}_{\psi(\phi)}$ and $\sheaf{L}_{\eta(\phi)}$ on $\Aa^1_{\Ff_p}$
such that we have
$$
\frtr{\sheaf{L}_{\psi(\phi)}}{\Ff_p}{x}=
\begin{cases}
  \psi(\phi(x))&\text{ if } \phi(x) \text{ is
    defined},\\
  0&\text{ if $x$ is a pole of $\phi$},
\end{cases}
$$
and
$$
\frtr{\sheaf{L}_{\eta(\phi)}}{\Ff_p}{x}=
\begin{cases}
  \eta(\phi(x))&\text{ if } \phi(x) \text{ is
    defined and non-zero},\\
  0&\text{ if $x$ is a zero or pole of $\phi$}
\end{cases}
$$
(these are the extensions by zero to $\Aa^1$ of the pullback by $\phi$
of the lisse Artin-Schreier and Kummer sheaves defined on the
corresponding open subsets of $\Aa^1$).
\par
Fix an isomorphism $\iota\,:\, \bar{\Qq}_{\ell}\ra \Cc$. We assume
that $\psi$ is the standard character, so that
$$
\iota(\psi(x))=e\Bigl(\frac{x}{p}\Bigr),
$$
for $x\in\Ff_p$. Similarly, if $\chi$ is a Dirichlet character modulo
$p$, there is a multiplicative character $\eta$ such that
$$
\iota(\eta(x))=\chi(x)
$$
for $x\in\Ff_p$.
\par
Let then $\phi_1$, $\phi_2\in\Qq(X)$ be rational functions as
in~(\ref{eq-weight-mixed}), with $\phi_2=1$ if $\chi$ is trivial. The
$\ell$-adic sheaf
\begin{equation}\label{eq-sheaf-mixed}
\sheaf{F}=\sheaf{L}_{\eta(\phi_2)}\otimes
\sheaf{L}_{\psi(\phi_1)},
\end{equation}
is such that
$$
\iota(\frtr{\sheaf{F}}{\Ff_p}{x})=
\begin{cases}
  \chi(\phi_2(x))e\Bigl(\frac{\phi_1(x)}{p}\Bigr)&
\text{ if $\phi_1$, $\phi_2$ are
    defined at $x$,}\\
  0&\text{ otherwise,}
\end{cases}
$$
which corresponds exactly to~(\ref{eq-weight-mixed}).

\begin{proposition}[Mixed character functions are trace functions]
  \label{pr-char-trace-weights} 
  Assume that either $\phi_1$ is not a polynomial of degree $\leq 1$,
  or if $\chi$ is non-trivial and $\phi_2$ is not of the form
  $t\phi_3^h$, where $h\geq 2$ is the order of $\chi$.
\par
\emph{(1)} The function above is an irreducible trace function.
\par
\emph{(2)} Let $d_1$ be the number of poles of $\phi_1$, with
multiplicity, and $d_2$ the number of zeros and poles of $\phi_2$ (where both are viewed as functions from $\Pp^1$ to $\Pp^1$). The
analytic conductor of the sheaf $\sheaf{F}$ satisfies
$$
\cond(\sheaf{F})\leq 1+2d_1+d_2.
$$
\end{proposition}

\begin{proof}
  (1) The sheaf $\sheaf{F}$ is pointwise pure of weight $0$ on the
  open set $U$ where $\phi_1$ and $\phi_2$ are both defined and
  $\phi_2$ is non-zero, which is the maximal open set on which
  $\sheaf{F}$ is lisse.  Moreover, it is of rank $1$ on this open set,
  and therefore geometrically irreducible.  By~\cite[Proof of Lemma
  8.3.1]{katz-gkm}, $\sheaf{F}$ is a Fourier sheaf provided it is not
  geometrically isomorphic to the Artin-Schreier sheaf
  $\sheaf{L}_{\psi(sX)}$ for some $s\in\Aa^1$, which is the case under
  our assumption.
  \par
  (2) The rank of $\sheaf{F}$ is one. The singular points are the
  poles of $\phi_1$ and the zeros and poles of $\phi_2$, so their
  number is bounded by $d_1+d_2$. Furthermore, the Swan conductor at
  any singularity $x$ is the same as that of
  $\sheaf{L}_{\psi(\phi_1)}$, since all Kummer sheaves are everywhere
  tame. Thus only poles of $\phi_1$ contribute to the Swan
  conductor, and for such a pole $x$, the Swan conductor is at most
  the order of the pole at $x$, whose sum is $ d_1$ (it is equal to
  the order of the pole when $\phi_1$ is Artin-Schreier-reduced at
  $x$, which happens if $p$ is larger than the order of the pole, see,
  e.g.,~\cite[Sommes Trig., (3.5.4)]{deligne}.)
\end{proof}

\subsection{``Fiber counting'' functions and their Fourier
  transforms}\label{ex-fiber-count}

This example is discussed in greater detail in \cite[\S
7.10]{katz-esde}, where a number of variants also appear.
\par
Let $C/\Qq$ be a geometrically connected smooth algebraic curve and
let $\phi:C\lra \Pp^1$ be a non-constant morphism of degree $\geq
2$. Let $D$ be the divisor of poles of $\phi$, $Z\subset C-D$ the
divisor of zeros of $d\phi$ and $S=\phi(Z)$. For $p$ large enough (in particular we assume $p>\deg(\phi)$),
this situation has good reduction modulo $p$ and we may consider the
``fiber-counting function''
$$
\begin{cases}
\Ff_p\lra \Zz\\
x\mapsto N(\phi;x)=|\{y\in C(\Ff_p)\,\mid\, \phi(y)=x\}|.
\end{cases}
$$
\par
Defining $\sheaf{F}=\phi_*\bar{\Qq}_{\ell}$, the direct image of the
trivial $\ell$-adic sheaf, we have
$$
N(\phi;x)=\iota(\frtr{\sheaf{F}}{\Ff_p}{x}).
$$
\par
The sheaf $\sheaf{F}$ is a constructible $\ell$-adic sheaf of rank
$\deg(\phi)$ on $\Aa^1$, and it is lisse and pointwise pure of weight
$0$ outside $S$ and tamely ramified there.  It is
not irreducible, but the kernel of the trace map
$$
\tilde{\sheaf{F}}=\ker(\sheaf{F}\fleche{\Tr}\bar{\Qq}_{\ell})
$$
might be irreducible. This sheaf $\tilde{\sheaf{F}}$ is of rank
$\deg(\phi)-1$, of conductor $\cond(\tilde\mcF)\leq \deg(\phi)+|S|$
and its trace function is
$$
\frtr{\tilde{\sheaf{F}}}{\Ff_p}{x}=N(\phi;x)-1=\tilde{N}(\phi;x).
$$
\par
By~\cite[Lemma 7.10.2.1]{katz-esde}, $\tilde{\sheaf{F}}$ is a Fourier
trace sheaf for $p>\deg(\phi)$.  The situation becomes even clearer if
we assume that $\phi$ is \emph{supermorse}, i.e.:
\begin{enumerate}
\item The zeros of the derivative $d\phi$ are simple;
\item $\phi$ separates the zeros of $d\phi$, i.e., the size of the set
  $S=\{\phi(x)\,\mid\, d\phi(x)=0\}$ of critical values of $\phi$ is
  the same as the number of zeros of $d\phi$.
\end{enumerate}
\par
In this case, by~\cite[Lemma 7.10.2.3]{katz-esde}, the sheaf
$\tilde{\sheaf{F}}$ is geometrically irreducible for $p>\deg(\phi)$,
and thus $\tilde{N}(\phi;x)$ is then an irreducible trace function.
\par
For a given non-trivial $\ell$-adic additive character $\psi$, the
Fourier transform sheaf $\tilde{\sheaf{G}}=\ft_{\psi}(\tilde{\sheaf{F}})(1/2)$ has
trace function given by
\begin{align*}
  |k|^{1/2}\frtr{\tilde{\sheaf{G}}}{k}{v}&= - \sum_{x\in k}{\Bigl( \sum_{\stacksum{y\in
        C(k)-D(k)}{\phi(y)=x}}{1}-1\Bigr)
    \psi(\Tr_{k/\Ff_p}(xv))}\\
  &=- \sum_{y\in C(k)-D(k)}{\psi(\Tr_{k/\Ff_p}(v\phi(y)))} +\sum_{x\in
    k}\psi(\Tr_{k/\Ff_p}(xv))
\end{align*}
for any finite-extension $k/\Ff_p$ and $v\in k$, which gives
$$
\frtr{\tilde{\sheaf{G}}}{k}{v}=-|k|^{-1/2}\sum_{x\in
  C(k)-D(k)}\psi(\Tr_{k/\Ff_p}(v\phi(x)))
$$
for $v\in k^{\times}$ and
$$
\frtr{\tilde{\sheaf{G}}}{k}{0}= |k|^{1/2}-|k|^{-1/2}|C(k)-D(k)|.
$$ 
(note that since $C$ is geometrically connected, we have
$|C(k)|=|k|+O(g_C\sqrt{|k|})$, so this last quantity is bounded.)
\par
Since $\tilde{\sheaf{F}}$ is an irreducible Fourier sheaf, so is
$\tilde{\sheaf{G}}$. Thus, taking $\psi$
the standard character with $\iota(\psi(x))=e(x/p)$, we get a sheaf
$\tilde{\sheaf{G}}$ with associated irreducible trace function given by
\begin{equation}\label{eq-kprime-phi}
K'(n)=-\frac{1}{\sqrt{p}} \sum_{x\in
  C(\Ff_p)-D(\Ff_p)}{e\Bigl(\frac{n\phi(x)}{p}\Bigr)},\quad\quad\text{
  for } 1\leq n\leq p-1,
\end{equation}
and
$$
K'(p)=\frac{p-|C(\Ff_p)-D(\Ff_p)|}{\sqrt{p}}
$$
(as before, this holds under the assumption that $\phi$ is
supermorse).
\par
By the Fourier inversion formula (in this context, this
is~\cite[Th. 7.3.8 (1)]{katz-esde}), the Fourier transform sheaf
$\ft_{\psi}(\tilde{\sheaf{G}})$ (note that we must use the same $\psi$
as was used to construct $\sheaf{G}$) is
$$
[x\mapsto -x]^*\tilde{\sheaf{F}}= [\times
(-1)]^*\tilde{\sheaf{F}}
$$
with trace function
$$
\frtr{[\times (-1)]^*\tilde{\sheaf{F}}}{k}{y}=\tilde{N}(\phi;-y).
$$
\par
We summarize this and estimate the conductors in a proposition.

\begin{proposition}[Fiber counting functions and
  duals]\label{pr-fiber-count}
  Let $C/\Qq$ and $\phi$ be as above, with $\phi$ supermorse.  
\par
\emph{(1)} For $p>\deg(\phi)$ such that there is ``good reduction'',
the functions $K$ and $K'$ defined above are irreducible trace functions
associated to the sheaves $\tilde{\sheaf{F}}$ and $\tilde{\sheaf{G}}$.
\par
Let $S\subset \bar{\Ff}_p$ be the set of critical values of $\phi$
modulo $p$. 
\par
\emph{(2)} The sheaf $\tilde{\sheaf{F}}$ is tame on $\Pp^1$, lisse on
$\Aa^1-S$, and has at most tame pseudo-reflection monodromy at all
$s\in S$. It satisfies
$$
\cond(\tilde{\sheaf{F}})\leq \deg(\phi)+|S|.
$$
\par
\emph{(3)} The sheaf $\tilde{\sheaf{G}}$ has rank $|S|$, it is lisse on
$\Gg_m$ and tamely ramified at $0$. At $\infty$, we have
$$
\swan_\infty(\tilde{\sheaf{G}})=
\begin{cases}
|S|-1&\text{ if } 0\in S\\
|S|&\text{ if } 0\notin S,
\end{cases}
$$
and hence $\cond(\tilde{\sheaf{G}})\leq 2|S|+2$.
\end{proposition}

\begin{proof}
  We have already discussed (1). Then~\cite[proof of Lemma
  7.10.2.3]{katz-esde} shows that $\tilde{\sheaf{F}}$ is tame
  everywhere, lisse on $\Aa^1-S$, and has tame pseudo-reflection monodromy at
  all $s\in S$. This gives
$$
\cond(\tilde{\sheaf{F}})\leq \rank(\tilde{\sheaf{F}})+|S|+1=\deg(\phi)+|S|.
$$
\par
For (3), since we know $\tilde{\sheaf{F}}$ is a tame pseudo-reflection
sheaf, we can use~\cite[Th. 7.9.4]{katz-esde} to see that
$\tilde{\sheaf{G}}$ has rank $|S|$ and is lisse on $\Gg_m$,
and~\cite[Cor. 7.4.5 (2)]{katz-esde} to see that it is tamely ramified
at $0$. Still from~\cite[Th. 7.9.4]{katz-esde}, we get the
decomposition
\begin{equation}\label{eq-supermorse}
  \tilde{\sheaf{G}}(\infty)=\bigoplus_{s\in S}\sheaf{L}_{\psi(sY)},
\end{equation}
as a representation of the wild inertia group at $\infty$. Hence
$$
\swan_\infty(\tilde{\sheaf{G}})=
\begin{cases}
|S|-1&\text{ if } 0\in S\\
|S|&\text{ if } 0\notin S,
\end{cases}
$$
and then
$$
\cond(\tilde{\sheaf{G}})\leq |S|+2+\swan_{\infty}(\tilde{\sheaf{G}}) \leq
2|S|+2.
$$
\end{proof}

To conclude this example, let us first recall that the condition of
being supermorse is generic, in a fairly natural and obvious
sense. For instance, if we consider $C=\Pp^1$ and look at the space
$L_{d_1,d_2}$ of all rational functions with coprime numerator and
denominator of fixed degrees $(d_1,d_2)$, the set of supermorse
functions $\phi\in L_{d_1,d_2}$ will be Zariski-dense.
\par

\subsection{Hyper-Kloosterman sums}\label{hyperklo}

Let $m\geq 2$ and let $p$ be a prime number. By results of Deligne
(see~\cite[11.0]{katz-gkm}), for all $\ell\not=p$, and any non-trivial
$\ell$-adic additive character $\psi$, there exists a sheaf
$\hk{m}$ on $\Aa^1_{\Ff_p}$ such that
$$
\frtr{\hk{m}}{k}{a}=
(-1)^{m-1}|k|^{-(m-1)/2}
\multsum_\stacksum{x_1\cdots
  x_m=a}{x_i\in k} \psi(x_1+\cdots+x_m)
$$
for all finite extensions $k/\Ff_p$ and all $a\in k^{\times}$. This
sheaf is a Fourier sheaf, geometrically irreducible, of rank $m\geq 2$
and pointwise pure of weight $0$, i.e., it is an irreducible trace
sheaf. 
\par
Now fix a non-constant rational fraction, $\phi(T)=R(T)/S(T),\
R(T),S(T)\in \Zz[T]$. Assuming that $p$ is large enough (greater that
the degree of $R,S$ and all their coefficients), the sheaf
$\hk{m,\phi}=\phi^*\hk{m}$ satisfies
$$
\frtr{\hk{m,\phi}}{\Ff_p}{a}=(-1)^{m-1}\hypk_m(\phi(a);p)
$$
for $a\in\Ff_p-\phi^{-1}(\{0,\infty\})$. The following result is the
main input to the proof of the second part of Corollary
\ref{cor-main}.

\begin{proposition}\label{pr-hyperklo}
  If $\phi$ is non-constant, the sheaf $\hk{m,\phi}$ above is
  geometrically irreducible and has conductor $\leq 2m+1+\deg(RS)$.
\end{proposition}

\begin{proof} 
Deligne has shown that $\hk{m}$ has rank $m$, is lisse on
$\Gg_m$, and is tame at $0$ and totally wild at $\infty$ with Swan
conductor $1$, so that
$$
\cond(\hk{m})=m+3
$$
(see, e.g.,~\cite[11.0.2]{katz-gkm}).
\par
It follows therefore that $\hk{m,\phi}$ is of rank $m$, is lisse
outside of the set $\phi^{-1}(\{0,\infty\})$, is tame at the zeros of
$\phi$ and wild at its poles. At a pole $x\in \phi^{-1}(\infty)$ of
order $d_x$, the map $\phi$ is generically \'etale, and hence we know
that $\swan_{x}(\phi^*\hk{m})=d_x\swan_{\infty}(\hk{m})=d_x$
by~\cite[1.13.1]{katz-gkm}. Finally, Katz has shown that $\hk{m}$ is
geometrically Lie-irreducible (see~\cite[Thm. 11.1]{katz-gkm}), i.e.,
that its restriction to any finite-index subgroup of the fundamental
group of $\Gg_m$ is geometrically irreducible. Since $\phi$ is
non-constant, this shows that $\hk{m,\phi}$ is also irreducible.
\end{proof}

\section{Examples of determination of $\haut_{\sheaf{F}}$}

Theorem~\ref{th-interpret-admissible} solves completely the question
of showing that isotypic trace functions are good, reducing it to an
estimation of the conductor of the associated sheaf.  However we find
it instructive to determine $\haut_{\sheaf{F}}$ as precisely as
possible for interesting families of functions, as was already done in
Section~\ref{sec-pedagogical} in simple cases. This gives
illustrations of the various possibilities, and would be a first step
in trying to improve the generic exponent $1/8$. Since we won't need
these results for this paper, we leave the proof to the reader as an
exercise in the theory of the $\ell$-adic Fourier transform (proximity
with \cite{katz-gkm,katz-esde} is strongly advised).
	
\subsection{Mixed characters}
\label{ssec-admissible-characters}

Let 
$$
\sheaf{F}=\sheaf{L}_{\eta(\phi_2)}\otimes
\sheaf{L}_{\psi(\phi_1)}
$$
be a sheaf corresponding to mixed characters, where either $\phi_1$ is
not a polynomial of order $\leq 1$, or $\eta$ is non-trivial of order
$h\geq 2$ and $\phi_2$ is not of the form $t\phi_3(X)^h$ for some
$t\in\Ff_p^{\times}$, and $\phi_3\in \Ff_p(X)$. Then one can show that
$\haut_{\sheaf{F}}$ is contained either in $B$ (the stabilizer of
$\infty$) or in $\rmN^{0,\infty}$ the normalizer of the diagonal
torus. For $\sheaf{F}=\sheaf{L}_{\psi(X^{-1})}$, we have
$\haut_{\sheaf{F}}=1$.

\subsection{Symmetric powers of Kloosterman sums}
\label{ssec-admi-kloos}

Let $\skl{1}_2=\phi^*\HYPK_2$ be the pull-back of the Kloosterman
sheaf $\HYPK_2$ of \S \ref{hyperklo} (relative to some additive
character $\psi$) by the map $x\mapsto x^2$, and for $d\geq 1$,
let $$\skl{d}_2=\mathrm{Sym}^d(\skl1)$$ be the $d$-symmetric power of
$\skl{1}$. The sheaf $\skl{d}$ is an irreducible trace sheaf of rank
$d+1$ and one finds:
\par
\begin{enumerate}
\item If $d\geq 3$, then $\haut_{\skl{d}}=1$;
\item If $d=1$ then $\haut_{\skl{1}}$ is the maximal torus in
$\PGL_2(\bar{\Ff}_p)$ stabilizing the subset $\{-2,2\}$;
\item If $d=2$ then $\haut_{\skl{2}}$ is the subgroup of
  $\PGL_2(\bar{\Ff}_p)$ stabilizing the subset $\{0,\infty,-4,4\}$,
  which is a dihedral group of order $8$ (these four points have
  cross-ratios $\{-1,1/2,2\}$, and one sees that any element of
  $\PGL_2$ stabilizing this set permutes the two pairs $\{0,\infty\}$
  and $\{-4,4\}$). In order to show that $\haut_{\skl{2}}$ is not
  smaller than this dihedral group, one may use the results of Deligne
  and Flicker~\cite[Cor. 7.7]{deligne-flicker} concerning tame local
  systems on $\Pp^1-\{\text{four points}\}$.
\end{enumerate}

\subsection{Fiber-counting functions}
\label{ssec-admi-fiber-count}

Let $C$ and $\phi$ be as in Example~\ref{ex-fiber-count}, with $\phi$
supermorse. Let $p>\deg(\phi)$ be a prime of good reduction, and let
$$
\tilde{\sheaf{F}}=\ker(\phi_*\bar{\Qq}_{\ell}\fleche{\Tr}\bar{\Qq}_{\ell})
$$
be the irreducible trace sheaf corresponding to the trace function
$K(x)=N_0(\phi;x)=N(\phi;x)-1$.
\par
If $\phi$ has degree $\geq 2$ and $0$ is not the unique critical value
of $\phi$, then one finds that $\haut_{\sheaf{F}}$ is a subgroup of
diagonal matrices of order bounded by $\deg(\phi)-1$.


\begin{thebibliography}{CC}

\bibitem{beauville-pgl} A. Beauville: \textit{Finite subgroups of
    $\PGL_2(K)$}, Contemporary Math. 522, 23--29, A.M.S (2010).

\bibitem{BHM} V. Blomer, G. Harcos and Ph. Michel: \textit{Bounds for
    modular $L$-functions in the level aspect}, Ann. Sci. \'Ecole
  Norm. Sup. (4) 40 (2007), no. 5, 697--740.

\bibitem{blomer-harcos} V. Blomer and G. Harcos: \textit{Hybrid bounds
    for twisted $L$-functions}, J. reine und angew. Mathematik 621
   (2008), 53--79.

\bibitem{byk} V.A. Bykovski: \textit{A trace formula for the scalar
    product of Hecke series and its applications}, J. Math. Sciences
  89 (1998), 915--932.

\bibitem{CI} J.B. Conrey and H. Iwaniec: \textit{The cubic moment of
    central values of automorphic $L$-functions}, Ann. of Math. (2)
  151 (2000), no. 3, 1175--1216.

\bibitem{deligne} P. Deligne: \textit{Cohomologie \'etale}, S.G.A
  4$\demi$, L.N.M 569, Springer Verlag (1977).

\bibitem{weilii}
P. Deligne: \textit{La conjecture de Weil, II}, Publ. Math. IH\'ES 52
(1980), 137--252.

\bibitem{deligne-flicker} P. Deligne and Y.Z. Flicker:
  \textit{Counting local systems with principal unipotent local
    monodromy}, Ann. of Math. (2) 178 (2013), no. 3, 921--982.
 
\bibitem{DI} J-M. Deshouillers and H. Iwaniec: \textit{Kloosterman
    sums and Fourier coefficients of cusp forms}, Invent. math.  70
  (1982/83), no. 2, 219--288.

\bibitem{dfi}
W.D. Duke, J. Friedlander and H. Iwaniec: \textit{Bounds for
  automorphic $L$-functions}, Invent. math. 112 (1993), 1--8.

\bibitem{dfi2}
W.D. Duke, J. Friedlander and H. Iwaniec: \textit{Bounds for
  automorphic $L$-functions II}, Invent. math. 115 (1994), 219--239.

\bibitem{DFIa} W.D. Duke, J. Friedlander and H. Iwaniec: \textit{The
    subconvexity problem for Artin $L$-functions}, Invent. math. 149
  (2002), no. 3, 489--577.

\bibitem{elliott} P.D.T.A Elliott, C.J. Moreno and F. Shahidi:
  \textit{On the absolute value of Ramanujan's $\tau$-function},
  Math. Ann. 266 (1984), 507--511.

\bibitem{EMOT} A. Erd\'elyi, W. Magnus, F. Oberhettinger and
  F.G. Tricomi: \textit{Higher transcendental functions}, Vol. II,
  McGraw Hill (1955).

\bibitem{FKM2} \'E. Fouvry, E. Kowalski, Ph. Michel: \textit{Algebraic
    trace functions over the primes}, Duke Math. J. 163 (2014), no. 9,
  1683--1736.

\bibitem{FKM3}\'E. Fouvry, E. Kowalski, Ph. Michel: \textit{An inverse
    theorem for Gowers norms of trace functions over $\mathbf F_{p}$},
  Math. Proc. Cambridge Philos. Soc. 155 (2013), no. 2, 277--295.

\bibitem{FKMd3} \'E. Fouvry, E. Kowalski, Ph. Michel, \emph{On the
    exponent of distribution of the ternary divisor function},
   \url{arXiv:1304.3199}, Mathematika (to appear).
  
 \bibitem{fouvry-ganguly} \'E Fouvry and S. Ganguly:
   \textit{Orthogonality between the M\"obius function, additive
     characters, and Fourier coefficients of cusp forms},
   Compos. Math. 150 (2014), no. 5, 763--797.

 \bibitem{FrIw} J. B. Friedlander and H. Iwaniec: \textit{Incomplete
     Kloosterman sums and a divisor problem}, with an appendix by
   Bryan J. Birch and Enrico Bombieri, Ann. of Math. (2) 121 (1985),
   no. 2, 319--350.
   
\bibitem{gr}
I.S. Gradshteyn and I.M. Ryzhkik: \textit{Tables of integrals, series
  and products}, 5th ed. (edited by A. Jeffrey), Academic Press
(1994). 

\bibitem{HB3} D. R. Heath-Brown: \textit{The divisor function $d_3(n)$
    in arithmetic progressions}, Acta Arith.  47 (1986), no. 1,
  29--56.

\bibitem{HB} D. R. Heath-Brown, \textit{The density of rational points
    on cubic surfaces}, Acta Arith. 79 (1997), 17--30.

\bibitem{Iw} H. Iwaniec: \textit{Fourier coefficients of modular forms
    of half-integral weight}, Invent. math. 87 (1987), no. 2,
  385--401.

\bibitem{IwActa} H. Iwaniec: \textit{Small eigenvalues of Laplacian
    for $\Gamma_0(N)$}, Acta Arith.  56 (1990), no. 1, 65--82.
 
\bibitem{topics} H. Iwaniec: \textit{Topics in classical automorphic
    forms}, Grad. Studies in Math. 17, A.M.S (1997).

\bibitem{IwI} H. Iwaniec: \textit{Introduction to the spectral theory
    of automorphic forms}, Biblioteca de la Revista Matematica
  Iberoamericana, Revista Matematica Iberoamericana, Madrid, 1995.

\bibitem{ant} H. Iwaniec and E. Kowalski: \textit{Analytic number
    theory}, A.M.S. Coll. Publ. 53 (2004).

\bibitem{katz-sommes} N.M. Katz: \textit{Sommes exponentielles},
  Ast\'erisque 79, Soc. Math. France (1980).

\bibitem{katz-gkm} N.M. Katz: \textit{Gauss sums, Kloosterman sums and
    monodromy groups}, Annals of Math. Studies 116, Princeton
  Univ. Press (1988).

\bibitem{katz-esde} N.M. Katz: \textit{Exponential sums and
    differential equations}, Annals of Math. Studies 124, Princeton
  Univ. Press (1990).

\bibitem{kim-sarnak} H. Kim and P. Sarnak: \textit{Refined estimates
    towards the Ramanujan and Selberg conjectures}, J. American
  Math. Soc. 16 (2003), 175--181.

\bibitem{kow-rt} E. Kowalski: \textit{An introduction to the
    representation theory of groups}, Grad. Studies in Math. 155,
  A.M.S (2014).

\bibitem{krw} E. Kowalski, O. Robert and J. Wu: \textit{Small gaps in
    coefficients of $L$-functions and $\mathfrak{B}$-free numbers in
    small intervals}, Rev. Mat. Iberoamericana 23 (2007), 281--326.

\bibitem{laumon} G. Laumon: \textit{Transformation de Fourier,
    constantes d'\'equations fonctionnelles et conjecture de Weil},
  Publ. Math. IH\'ES, 65 (1987), 131--210.

\bibitem{mv} Ph. Michel and A. Venkatesh: \textit{The subconvexity
    problem for $\GL_2$}, Publ. Math. I.H.\'E.S 111, 171--271 (2010).

\bibitem{Moto} Y. Motohashi: \textit{On sums of Hecke-Maass
    eigenvalues squared over primes in short intervals}, preprint
  \url{arXiv:1209.4140v1}.

\bibitem{munshi} R. Munshi: \textit{ Shifted convolution sums
    for $GL(3)\times GL(2)$}, Duke Math. J. 162 (2013), no. 13, 2345--2362.
 
\bibitem{pitt} N. Pitt: \textit{On shifted convolutions of
    $\zeta(s)^3$ with automorphic $L$-functions}, Duke Math. J.  77
  (1995), no. 2, 383--406.

\bibitem{RR} D. Ramakrishnan and J. Rogawski, \textit{ Average values
    of modular $L$-series via the relative trace formula}, Pure
  Appl. Math. Q. 1 (2005), no. 4, Special Issue: In memory of Armand
  Borel. Part 3, 701--735.

\bibitem{SarnakKyoto} P. Sarnak: \textit{Diophantine problems and
    linear groups}, in Proceedings of the I.C.M., 1990, Kyoto,
  Springer (1991), 459--471.

\bibitem{serre-lin} J-P. Serre: \textit{Repr\'esentations lin\'eaires
    des groupes finis}, 2\`eme \'Edition, Hermann, 1971.

\bibitem{Strom} A. Str\"ombergsson: \textit{On the uniform
    equidistribution of long closed horocycles}, Duke Math. J. 123
  (2004), no. 3, 507--547.

\bibitem{US} P. Sarnak and A.  Ubis: \textit{The horocycle flow at
    prime times}, Journal Math. Pures Appl., to appear.

\bibitem{Ven} A. Venkatesh: \textit{Sparse equidistribution problems,
    period bounds and subconvexity}, Ann. of Math. (2) 172 (2010),
  no. 2, 989--1094.


\end{thebibliography}
\end{document}